\documentclass[12pt,reqno]{amsart}
\usepackage{graphicx} 

\usepackage{enumerate}
\usepackage{amsmath, amsthm} 
\usepackage{amsfonts}
\usepackage{amssymb}
\usepackage{fullpage} 
\usepackage{xcolor}

\newtheorem{definition}{Definition}[subsection]

\newtheorem{theorem}{Theorem}[subsection]
\newtheorem{thm}[theorem]{Theorem}
\newtheorem{cor}[theorem]{Corollary}
\newtheorem{prop}[theorem]{Proposition}
\newtheorem{conj}[theorem]{Conjecture}

\newtheorem{lemma}[theorem]{Lemma}
\newtheorem{proposition}[theorem]{Proposition}
\newtheorem{corollary}[theorem]{Corollary}

\usepackage{mathabx,amsmath,amscd, latexsym, amssymb, bbold}

\usepackage{amsmath, amssymb, tikz}

\DeclareMathOperator{\Hig}{Hig}
\DeclareMathOperator{\ev}{ev}
\DeclareMathOperator{\coev}{coev}

\def\bC{\mathbb{C}}
\def\bZ{\mathbb{Z}}
\def\cZ{\mathcal{Z}}

\def\spa{\mathrm{span}}
\def\NK{{\mathcal{K}}}
\def\Kn{{\NK_n}}
\DeclareMathOperator{\tr}{tr}
\DeclareMathOperator{\id}{id}
\DeclareMathOperator{\Hom}{Hom}
\DeclareMathOperator{\Rep}{Rep}
\DeclareMathOperator{\End}{End}

\DeclareMathOperator{\irRC}{\operatorname{Irr}(\mathcal C)}
\DeclareMathOperator{\inRC}{\operatorname{Ind}(\mathcal C)}
\DeclareMathOperator{\fuRC}{\operatorname{Full}(\mathcal C)}

\DeclareMathOperator{\irRB}{\operatorname{Irr}(\mathcal B)}

\DeclareMathOperator{\StRB}{\operatorname{St}(\mathcal B)}
\DeclareMathOperator{\HigRB}{\operatorname{Hig}(\mathcal B)}

\DeclareMathOperator{\lirr}{\operatorname{I}_{\textrm{irr}}}
\DeclareMathOperator{\lind}{\operatorname{I}_{\textrm{ind}}}
\DeclareMathOperator{\lfull}{\operatorname{I}_{\textrm{full}}}

\DeclareMathOperator{\ccenterB}{Nat(\textrm{Id}_{\mcB})}

\newcommand{\mcB}{\mathcal{B}}
\newcommand{\mcC}{\mathcal{C}}

\newcommand{\mcI}{\mathcal{I}}

\newcommand{\mcK}{\mathcal{K}}
\newcommand{\mcL}{\mathcal{L}}

\newcommand{\mcZ}{\mathcal{Z}}

\newcommand{\mbbC}{\mathbb{C}}

\newcommand{\mbbZ}{\mathbb{Z}}

\newcommand{\mfg}{\mathfrak{g}}

\title{Modular Data of Non-semisimple Modular Categories}
\date{}

\author[nankai]{Liang Chang\textsuperscript{(1)}} 
\email{changliang996@nankai.edu.cn} 

\author[ucsb]{Quinn T. Kolt\textsuperscript{(2)}}
\email{quinn@math.ucsb.edu}

\author[ucsb]{Zhenghan Wang\textsuperscript{(2)}}
\email{zhenghwa@math.ucsb.edu}

\author[ucsb]{Qing Zhang\textsuperscript{(2)}}
\email{qingzhang@ucsb.edu}

\address[1]{Department of Mathematics, Nankai University, Tianjin, China}

\address[2]{Department of Mathematics, University of California, Santa Barbara, CA 93106, USA}

\begin{document}
\begin{abstract}
We investigate non-semisimple modular categories with an eye towards a structure theory, low-rank classification, and applications to low dimensional topology and topological physics.  We aim to extend the well-understood theory of semisimple modular categories to the non-semisimple case by using representations of factorizable ribbon Hopf algebras as a case study. We focus on the Cohen-Westreich modular data, which is obtained from the Lyubashenko-Majid modular representation restricted to the Higman ideal of a factorizable ribbon Hopf algebra. The Cohen-Westreich $S$-matrix diagonalizes the mixed fusion rules and reduces to the usual $S$-matrix for semisimple modular categories. The paper includes detailed studies on small quantum groups $U_qsl(2)$ and the Drinfeld doubles of Nichols Hopf algebras, especially the $\mathrm{SL}(2, \mathbb{Z})$-representation on their centers, Cohen-Westreich modular data, and the congruence kernel theorem's validity.
\end{abstract}
\maketitle

\tableofcontents

\section{Introduction}

In this paper, we investigate non-semisimple modular categories\footnote{Modular category is sometimes used as a short name of modular tensor category. We use modular
category to mean a non-degenerate ribbon finite tensor category.  Hence, they include both the semisimple ones (i.e., modular tensor categories) and non-semisimple ones.} 
(NSS MCs) with an eye towards a structure theory, low-rank classification, and applications to low dimensional topology and topological physics in the future.  Our strategy is to pursue a parallel theory to semisimple modular categories (i.e., modular tensor categories) by leveraging our knowledge of modular tensor categories and contrasting them using examples: NSS MCs from representations of factorizable ribbon Hopf algebras.

On one hand, unitary topological quantum field theories (TQFTs) have deep applications to quantum physics and quantum computing, yet no new profound applications to classical topology have been discovered.\footnote{The Rochlin theorem in smooth 4-dimensional topology can be deduced from TQFTs.}   On the other hand, the first Witten-Donaldson TQFT is non-unitary with transforming consequences for smooth 4-dimensional topology.  It suggests that non-semisimplicity might be necessary for applications to classical topological problems such as the classification of smooth $4$-manifolds and the Andrews-Curtis conjecture.  Arguably, the best-understood TQFTs are the semisimple (2+1)-TQFTs from modular tensor categories.  This gives motivation to see how much knowledge about modular tensor categories can be generalized to NSS MCs.  

Modular data is the most powerful invariant of a modular tensor category. The Lyubashenko-Majid modular representation transforms the center $Z(H)$ of any finitely dimensional factorizable ribbon Hopf algebra $H$ into an $\textrm{SL}(2,\mbbZ)$-module \cite{lyubashenko1994braided}. This representation on $Z(H)$ can be restricted to its Higman ideal $\textrm{Hig}(H)$.
Modular data of modular tensor categories are matrices, so a basis needs to be chosen for $\textrm{Hig}(H)$.   The two bases in \cite{cohen2008characters} used for a version of Verlinde formulas are natural choices.   
Based on \cite{cohen2008characters}, we propose that the restriction of the Lyubashenko-Majid modular representation $Z(H)$ to its Higman ideal $\textrm{Hig}(H)$ with the two bases to be the modular data for modular categories. 
We will call the resulting two matrices $(S_{CW}, T_{CW})$ the Cohen-Westreich modular data of modular categories.  They diagonalize the mixed fusion rules and reduce to the usual modular data for semisimple modular categories. 

Every NSS MC $\mathcal{B}$ can be semi-simplified to a ribbon fusion category $\overline{\mathcal{B}}$.  We hope that the Cohen-Westreich modular data of $\mathcal{B}$ and the $(S,T)$-matrices of its semi-simplification $\overline{\mathcal{B}}$ will provide insight into a structure theory for modular categories. As shown in Prop. \ref{non-modular}, the semi-simplification $\overline{\mathcal{B}}$ of a modular category $\mcB$ is not in general modular.

Our idea is to build a structure theory for NSS MCs via their semisimple quotients, analogous to the spin cover approach to fermionic modular categories in \cite{bruillard2017fermionic} and their mixed fusion modules.  Conceptually, this is interesting, as the Swan-Serre theorem shows that projective modules are algebraic incarnations of bundles.  A modular tensor category can be regarded as finitely many quantum points, and pursuing a projective cover can be regarded as a quantum covering space: the simplest quantum bundle theory.  

A large family of modular categories comes from quantum groups at roots of unity \cite{turaev2010quantum,bakalov2001lectures}.  These quantum groups at roots of unity are factorizable ribbon Hopf algebras. Let $\mfg$ be the Lie algebra of a simple Lie group $G$ and $q=e^{\frac{2\pi i}{m(\check{h}+k)}}, k\geq 1$, where $\check{h}$ is the dual Coxeter number of $\mfg$ and $m=1,2,3$ depending on the Lie algebra\footnote{$m=1$ for ADE, $m=2$ for B,C,$F_4$, and $m=3$ for $G_2$.},  then the quantum group at root of unity $q$ will be denoted as $U_q\mfg$.  The resulting modular category from the representation category of $U_q\mfg$ will be denoted as $U_qG_k$ and referred to as the modular category of $G$ at level $=k$ following physics jargon.

Small quantum groups $U_q sl(2)$ have been extensively studied in the literature.  Our interest comes from a conjecture of Kerler on the decompositions of their centers \cite{kerler2003homology}.  There are many proofs of Kerler's conjecture, and ours is the closest to Kerler's original work: complete explicit decompositions of the centers as $\operatorname{SL}(2,\mbbZ)$-module.  We identify one summand in the decomposition as the Higman module and the other as the tensor product of the $\operatorname{SL}(2,\mbbZ)$ defining representation with the modular representation of the even half of $\operatorname{SU}(2)_k$.  We also identify the Higman module as a Weil representation, so its kernel is a congruence subgroup.  This is the content of Section \ref{sec:uqsl2}.

In Section \ref{sec:doubleofnichols}, we study the Drinfeld doubles of the Nichols Hopf algebras\footnote{These are distinct from the Nichols algebras of braided vector spaces. However, both Nichols Hopf algebras and small quantum groups can be naturally constructed from Nichols algebras \cite[Exercises 8.25.13 and 8.25.15]{EGNO}. The use of ``Nichols Hopf algebra'' for $\Kn$ appears in \cite{EGNO}.} $\mcK_n$.   The first Nichols Hopf algebra $\mcK_1$ is the $4$-dimensional Sweedler's Hopf algebra, which is ribbon but not unimodular.  The second Nichols Hopf algebra $\mcK_2$ appeared in Kerler's paper \cite{kerler2003homology}, which is ribbon and unimodular but not factorizable. In general, $\mcK_n$ is always ribbon, unimodular if and only if $n$ is even, and never factorizable (their representation categories are not modular).  To obtain examples of NSS MCs, we study their doubles $D\mcK_n$.  $D\mcK_n$ is a factorizable ribbon Hopf algebra for even $n$, but the ribbon element studied here is not valid for odd $n$. It can be shown that $D\Kn$ with its canonical $R$-matrix is not ribbon for odd $n$. For the even $n$ case, the doubled Nichols Hopf algebras $D\mcK_{n}$ are also known as the symplectic fermion Hopf algebras $Q(n,1)$ \cite{farsad2022symplectic}. A similar investigation of a quasi-Hopf generalization of $D\mcK_{n}$ for even $n$ can be found in \cite{farsad2022symplectic}. 

The contents of the remaining sections are as follows.  Section \ref{sec:NSSMCs} is a collection of basic notions of modular categories with a new proposition on the minimal total rank.  Section \ref{sec:factorizableHA} summarizes fundamental facts on what we call Cohen-Westreich modular data.  Section \ref{sec:SSofMCs} recalls semi-simplification.  Section \ref{sec:comparison} studies which theorems for modular tensor categories generalize to arbitrary modular categories. We observe that rank-finiteness fails.  The congruence kernel theorem holds for the Higman ideals of both small quantum groups $U_qsl(2)$ and doubled Nichols Hopf algebras $D\mcK_n$ for even $n$.  We conjecture that the congruence kernel theorem holds in general for the Higman ideal.

\section{Non-semisimple modular categories}\label{sec:NSSMCs}

In this paper, we work over the complex numbers $\mathbb{C}$.  All algebras are over $\mbbC$, and all modules over them are assumed to be finite-dimensional over $\mathbb{C}$.  Categories are assumed to be  $\mathbb{C}$-linear and finite.  Our pictures read from bottom up. 

We collect some notation and terminology in the sequel into the following.

\subsection{Notations and terminologies}\label{notation}

\begin{enumerate}

\item In general, a finite tensor category (FTC) will be denoted as $\mcC$, and a modular category as $\mathcal{B}$.  Terminologies for FTCs $\mcC$ applying to modular categories $\mcB$ mean for the underlying FTCs of the modular categories $\mcB$.

\item There are three label sets for each FTC $\mcC$.  The set $\lirr$ of isomorphism classes of simple objects, the set $\lind$ of isomorphism classes of projective indecomposable objects, and the union of the two sets $\lfull=\lirr\cup \lind$.  The union is in general not disjoint as there are objects that are both simple and projective.  Such objects will be called the Steinberg objects.  

The cardinality of $\lfull$ will be called the total rank, denoted as $n$, the cardinality of $\lirr$ as the rank, denoted as $r$, and the cardinality of the intersection of the two sets as the Steinberg rank, denoted as $b$.  Note that $n+b=2r$, hence the total rank $n$ and Steinberg rank $b$ have the same parity.

We fix a complete set of representatives of objects labeled by each of the three label sets as $\irRC$, $\inRC$, and $\fuRC$.  Similarly a complete representative set $\StRB$ of isomorphism classes of the Steinberg objects.

\item Throughout the paper, $H$ will denote a finite-dimensional Hopf algebra and $J(H)$ the Jacobson radical of $H$.  We choose a complete set of distinct irreducible $H$-modules $\{V_i\}_{i\in I_{irr}}$ and let $\{P_i\}_{i\in I_{ind}}=\{P(V_i)\}$ denote their projective covers.  It is a fundamental result in algebra that there is well-defined 1-1 correspondence between $I_{irr}$ and $I_{ind}$.  In general, there can be modules of $H$ that are both irreducible and projective.

\item For a Hopf algebra $H$, $H$-mod is the category of finite-dimensional left modules of $H$, therefore $H$-mod is an FTC.  The irreducible modules are the simple objects of $H$-mod with projective covers.

Our main interest is the modular category of finite-dimensional modules over a factorizable ribbon Hopf algebra.  The categorical counterparts of their irreducible modules and projective indecomposable modules will be called the simple objects and projective indecomposable objects of the resulting modular categories.

\item Let $\mfg$ be the Lie algebra of a simple Lie group $G$ and $q$ a root of unity as in the introduction, then the quantum group at the root of unity will be denoted as $U_q\mfg$.  We will denote the corresponding modular category as $U_qG_k$ and call it the modular category of $G$ at level$=k$ following physics jargon.  The semi-simplification of $U_qG_k$ is the usual modular tensor category (MTC) $G_k$.

\item There will be four different flavors of fusion rings/modules for non-semisimple modular categories: the full fusion ring, the Grothendieck ring, the projective Grothendieck ring, and the mixed fusion module.  To distinguish the  different fusion coefficients, we will use $N_{ij}^k(f)$ for the different flavors $f=t,g,p,m$ for the fusion coefficient/multiplicity.  The notation  $N_{ij}^k$ will be reserved for the fusion multiplicity of MTCs or if no confusion would arise.

\item We will use $Z(A)$ to denote the center of an algebra $A$, and $\mcZ(\mcC)$ to denote the Drinfeld center of an FTC $\mcC$.

\item  There are two versions of the Higman ideal: $\HigRB$ for a modular category $\mcB$, and $\Hig(H)$ for a Hopf algebra $H$. Their dimensions are denoted as $m$, which is the same as the rank of the Cartan matrix.  

\item There are many uses of the letter $S$ for modular matrices and antipodes, so we will use $s_H$ to denote the antipode of a Hopf algebra $H$. $S_{LM}$ will denote the Lyubashenko-Majid map on $H$ and its center $Z(H)$. $S_{CW}$ will denote the restriction of $S_{LM}$ to the Higman ideal under some basis.  We will use $S,T$ for the usual modular data of MTCs, and $s,t$ for the usual defining matrices of $\operatorname{SL}(2,\mbbZ)$.

\item Given an FTC $\mcC$,  $\mcC_{QD}$ denotes the full sub-category, not necessarily abelian or monoidal, of all quasi-dominated objects by simple objects.  Quasi-domination is defined in Section \ref{QD:defintion}.

\end{enumerate}

\subsection{Modular categories}

A finite tensor category (FTC) is a generalization of fusion category to the non-semisimple setting and the categorical counterpart of a finite-dimensional weak quasi-Hopf algebra.  We strictly follow the definition in \cite{etingof2004finite}. In particular, an FTC is rigid, and every simple object $X$ has a unique projective cover $P(X)$ up to isomorphism.  Moreover, each projective indecomposable object covers a simple object called its socle. 

An FTC is unimodular if the projective cover $P_1$ of the tensor unit is self-dual.

A ribbon structure on an FTC consists of a braiding and a twist that are compatible with the rigidity \cite{turaev2010quantum}.  A ribbon FTC is an FTC with a ribbon structure.

A ribbon FTC $\mcB$ is non-degenerate if it satisfies one of the following four equivalent conditions in \cite{shimizu2019non}:

\begin{enumerate}
    \item The braided tensor functor $\mathcal{B} \boxtimes \mathcal{B}^{\operatorname{rev}}\rightarrow  \mcZ(\mcB)$ is an equivalence.
    \item The pairing on the coend in \cite{shimizu2019non} is non-degenerate.
    \item The end and coend are isomorphic to each other.
    \item The Müger center of $\mathcal B$ is trivial.
\end{enumerate}

\begin{definition}
A modular category $\mathcal{B}$ is a non-degenerate ribbon FTC.

\end{definition}

We also have:

\begin{enumerate}
    \item The number of distinct isomorphism classes of simple objects is called the rank.
    \item The total number of distinct isomorphism classes of simple and projective indecomposable objects will be called the total rank.
    \item An object is called Steinberg if it is both simple and projective.  The number of distinct isomorphism classes of Steinberg objects will be called the Steinberg rank.
\end{enumerate}

A semisimple modular category is called a modular tensor category (MTC).  MTCs are extensively studied related to anyon theory in topological physics (e.g. see \cite{wang2010topological,rowell2018mathematics} ). Our main focus in this paper is on non-semisimple modular categories (NSS MCs).

Recall the following result about NSS MCs:

\begin{theorem}\label{thm:S-rank}\cite[Theorem 3.4]{gainutdinov2020projective}
    The Steinberg rank of any NSS MC is at least one.
\end{theorem}

\subsection{Jordan-Hölder or composition series}

\begin{definition}

    Let $X$ be an object in an FTC, then there exists a filtration
\begin{equation*}
    0=X_0 \subset X_1 \subset \cdots \subset X_{n-1} \subset X_n=X
\end{equation*}
such that $X_i / X_{i-1}$ is simple for all $i$. Such a filtration is called a Jordan-Hölder series of $X$. 
\end{definition}

By the Jordan-Hölder theorem, the number of any simple object $Y=X_i/X_{i-1}$ appeared in a Jordan-Hölder series is the same in any two Jordan-Hölder series. We call this number the \textit{c-multiplicity} of $Y$ in $X$ and denote it by $[X:Y]$. The length of $X$ is the length of its Jordan-Hölder series.

\subsection{Minimal total rank}

The set of extensions of the tensor unit $\mathbf 1$ in an FTC by itself is trivial, i.e., $\operatorname{Ext}^1(\mathbf 1, \mathbf 1)=0$ \cite[Theorem 2.17]{etingof2004finite}. We use this result to deduce the following:  

\begin{lemma}\label{lem:simi-simple}
    \begin{enumerate}
     \item If the tensor unit $\mathbf 1$ of an FTC $\mcC$ is projective, then $\mcC$ is semisimple.
     \item If all nontrivial simple objects of an FTC $\mathcal C$ are projective, then $\mathcal C$ is semisimple.
    \end{enumerate}
\end{lemma}

\begin{proof}

\begin{enumerate}
\item This follows from \cite[Corollary 4.2.13]{EGNO}.  For completeness, we outline a proof.  First recall that the tensor product of a projective with any object is always projective.  For any projective indecomposable object $P$ of $\mcC$, there is its socle--a simple object $X$ such that $P\cong P(X)$.  Since $X\cong \mathbf 1\otimes X$, so $X$ is also projective, hence it is a projective cover of itself.  The projective cover of any simple object is unique up to isomorphism, it follows that $X\cong P(X)\cong P$.  Therefore, all projective indecomposable objects of $\mcC$ are also simple, i.e. $\mcC$ is semisimple.

\item We show that the tensor unit $\mathbf 1$ must be projective in this case. Assume $\mathbf 1$ is not projective and let $P_{\mathbf 1}$ be its projective cover. Consider the Jordan-Hölder series of $P_{\mathbf{1}}$, i.e., a filtration 
 $$0=X_0 \subset X_1 \subset \cdots \subset X_{n-1} \subset X_n=P_{\mathbf 1}$$
 such that $X_i/X_{i-1}$ is simple for all $i$. Assume there exists a nontrivial projective simple object $Y$ such that $Y\cong X_i/X_{i-1}$ for some $i$.
 Note $Y$ is also injective object by \cite[Prop 6.1.3]{EGNO}. This implies $Y$ is a direct summand of $P_{\textbf{1}}$, a contradiction. Thus we have $X_i/X_{i-1}\cong \mathbf 1$ for all $i$, which contradicts to $\operatorname{Ext}^1(\mathbf{1}, \mathbf{1})=0$.  Therefore $\mathbf 1$ is projective, and $\mathcal C$ is semisimple by (1). 
    
\end{enumerate}

\end{proof}

\begin{prop}\label{smallestrank}

\begin{enumerate}
    \item Any  FTC of total rank $<4$ is semisimple.
    
    \item Any modular category of total rank $<5$ is semisimple.
\end{enumerate}

\end{prop}
\begin{proof}
As $\operatorname{Ext}^1(\mathbf{1}, \mathbf{1})=0$, any FTC with one simple object is equivalent to $\operatorname{Vec}$. Let $\mathcal C$ be an NSS FTC with two simple objects $\mathbf 1$ and $X$. By Lemma \ref{lem:simi-simple}, we know both $\mathbf 1$ and $X$ are not projective. Thus there has to be projective covers $P_{\mathbf 1}, P(X)$ of $\mathbf 1, X$.  Therefore, the minimal total rank for a NSS FTC is $4$.

For (2),  by Thm. \ref{thm:S-rank}, the modular category also has an extra Steinberg object.  Therefore,  the minimal total rank of an NSS MC is $5$.
\end{proof}

Both bounds are sharp.  The $\mathcal{K}_n$-mod of Nichols Hopf algebras $\mathcal{K}_n$ for even $n$ are unimodular ribbon FTCs of total rank=$4$,  while $U_qsl(2)$-mod at level=$1$ is a modular category of total rank=$5$.

\subsection{Cartan matrix of finite tensor category}

\begin{definition} 
Given an FTC $\mcC$ of rank=$r$ with a complete representative set of simple objects $\{V_j\}_{j=1}^r$ and their projective covers $\{P_i=P(V_i)\}_{i=1}^r$, we define 
\begin{equation}\label{eq:CartanMatrix}
    C_{ij}:= [P_i:V_j],
\end{equation}
i.e., the $c$-multiplicity of $V_j$ in $P(X_i)$, then the integral matrix $C=(C_{ij}) \in M_{r\times r}(\mathbb Z)$ is called the Cartan matrix of $\mcC$. 
\end{definition}

The rank of the Cartan matrix of an FTC $\mcC$ will be called its Cartan rank.

The Cartan matrix of an MTC is the identity matrix, but 
the Cartan matrix of any NSS MC is always singular \cite[Theorem 6.6.1]{EGNO}.

\begin{definition}
    Two projective indecomposable objects $X,Y$ are $c$-equivalent if they have the same column in the Cartan matrix, denoted as $X\sim_c Y$.
\end{definition}

Equivalently, $X$ and $Y$ lead to the same element in the Grothendieck ring of $\mcC$.

The set of all projective indecomposable objects that are $c$-equivalent will be called a $c$-class.  Note that the number of $c$-classes of projective indecomposable objects in an FTC can be different from its Cartan rank \cite{grinberg2020critical}.  We will call the number of $c$-classes the $c$-rank of an FTC.  We are not aware of examples of modular categories that the Cartan rank and $c$-rank are different.  We assume all FTCs have the same Cartan and $c$-rank in the sequel.

\subsection{Fusion rings and modules}\label{QD:defintion}

Given an FTC $\mcC$, there are several different fusion rings/modules: the full fusion ring, the Grothendieck ring, the projective Grothendieck fusion ring, and the mixed fusion module. 

Our main interest will be on the mixed fusion module: the action of the simple objects on $c$-classes of the projective indecomposable objects via the tensor product.

Given an FTC $\mcC$, recall that $\operatorname{Irr}(\mathcal C)$ is the set of isomorphism classes of simple objects of $\mcC$. Given $V_i\in\operatorname{Irr}(\mathcal C)$ and $P_i=P(V_i)$ a projective cover. For Steinberg objects $V_i=P_i$, which always exists in an NSS MC \cite{cohen2008characters,gainutdinov2020projective}.
The map $V_i\mapsto P_i$ is a one-to-one correspondence between the simple objects and the projective indecomposable objects of $\mcC$.

\subsubsection{Full fusion ring $F(\mathcal{C})$}

Let $\{V_1,...,V_a, V_{a+1},...,V_{a+b}, P_1, ..., P_n\}$ be a complete set of representatives  $\{V_i\}_{i=1}^a$ of simple object types that are not projective, Steinberg types $\{V_{a+j}\}_{j=1}^b$, and projective covers $\{P_i=P(V_i)\}_{i=1}^a$ of $V_i$.  

An FTC $\mcC$ is an abelian category, therefore all kernels and cokernels are in the category. For a Hopf algebra $H$, this would lead to the Green ring, which could be very complicated in the non-semisimple case.  Instead of the Green ring, we will consider the fusion ring of a sub-category $\mcC_{QD}$ of $\mcC$, not necessarily abelian or monoidal, that consists of only objects that are quasi-dominated by the simple objects. 

An object $X$ of $\mcC$ is quasi-dominated by some simple objects $\{X_i\}$ of $\irRC$ if there exist a family of morphisms $\alpha_{i,m_i}: X_i\rightarrow X, \beta_{i,m_i}: X\rightarrow X_i$ such that
$$ \textrm{Id}_X-\sum_{i, m_i}\alpha_{i,m_i}\beta_{i,m_i}$$ is a negligible morphism \cite{turaev2010quantum}. It is clear that every negligible object is in $\mcC_{QD}$.

Then the full fusion ring $F(\mathcal{C})$ of $\mcC$ is the free abelian group generated by the isomorphism classes $[X]$ of objects in $\mcC_{QD}$ with addition given by direct sum $[X]+[Y]=[X\oplus Y]$, and multiplication by tensor product $[X][Y]=[X\otimes Y]$, where $[X\otimes Y]$ decomposes into a sum of the isomorphism classes of the simple and indecomposable objects in 
$\{V_1,...,V_a, V_{a+1},...,V_{a+b}, P_1, ..., P_n\}$.

Fusion coefficients in the full fusion ring will be denoted as $N_{ij}^k(t)$.

\subsubsection{Grothendieck ring $\operatorname{Gr}(\mathcal{C})$}

\begin{definition}
 The Grothendieck ring $\operatorname{Gr}(\mathcal{C})$ of an FTC $\mathcal{C}$ is the free abelian group generated by the elements in $\operatorname{Irr}\mathcal{C}$. The associated class  $[X] \in \operatorname{Gr}(\mathcal{C})$ of $X\in \mathcal C$ is given by
\begin{equation}
    [X]=\sum_{X_i\in\operatorname{Irr}(\mathcal C)}\left[X: X_i\right] X_i.
\end{equation}
The multiplication on $\operatorname{Gr}(\mathcal{C})$  is defined by 
\begin{equation}
    [X_j] [X_k]:=\left[X_j \otimes X_k\right],
\end{equation}
which gives rise to the ring structure on $\operatorname{Gr}(\mathcal{C})$. 
We will use $X$ instead of $[X]$ when no confusion arises. 
\end{definition}

In the Grothendieck ring $Gr(\mathcal{C})$, a short exact sequence $0\rightarrow X\rightarrow Z\rightarrow Y\rightarrow 0$ leads to $[Z]=[X]+[Y]$.  Therefore, all projective indecomposables are sums of simples in the Grothendieck ring according to their Jordan-Holder series.

Fusion coefficients in the Grothendieck ring will be denoted as $N_{ij}^k(g)$.

\subsubsection{Projective Grothendieck ring $K_0(\mathcal{C})$}

\begin{definition} Let $K_0(\mathcal C)$ be the abelian group generated by the isomorphism classes $[P]$ of projective objects $P$ modulo the relations $[P \oplus Q]=[P]+[Q]$.  The multiplication on $K_0(\mathcal C)$ is defined as $[P] \cdot[Q]=[P \otimes Q]$ for any projective objects $P$ and $Q$ in $\mathcal C$. In general, $K_0(\mathcal{C})$ can be a ring without a unit (rng).

\end{definition}

Fusion coefficients in the projective Grothendieck ring will be denoted as $N_{ij}^k(p)$.

Note $K_0(\mathcal{C})$ is a $\operatorname{Gr}(\mathcal{C})$-bimodule. One can also describe the action via the usual fusion coefficients given by elements in $\operatorname{Gr}(\mathcal{C})$, see \cite[Prop 6.1.2]{EGNO}.

\subsubsection{Mixed fusion module $M(\mathcal{C})$}

Let $M(\mathcal{C})$ be the free abelian group generated by
$\{V_i\}_{V_i\in \operatorname{Irr}(\mathcal{C})}$ and $\{[P_i]_c\}_{V_i\in \operatorname{Irr}(\mathcal{C})}$, where $[P_i]_c$ is the $c$-class of $P_i$.  Then $V_i\otimes P_j$, which is projective, decomposes as a direct sum of indecomposable projective objects. 
The mixed fusion module is obtained by collecting fusion coefficients into $c$-classes of projective indecomposable objects in the full fusion ring: 

$$V_i\otimes P_j \cong \oplus_{k=1}^n N_{ij}^k(m) [P_k]_c,$$ where $n$ is the $c$-rank and $N_{ij}^k(m)=\sum_{P_k\in [P_k]_c} N_{ij}^k(t)$, where $N_{ij}^k(t)$ is the fusion coefficient in the full fusion ring of $\mcC$.

Define the mixed fusion matrix $N^i(m)$ for each $i=1,..,r$ by $(N_i(m))_{jk}=N_{ij}^k(m)$. Then each mixed fusion matrix $N^i(m)$ is an $n\times n$ matrix with non-negative integral entries, where $r$ is the rank and $n$ $c$-rank of $\mathcal{C}$.

\subsubsection{Cartan map}
The Cartan map of an FTC $\mcC$ is the homomorphism defined by 
\begin{equation}\label{eq: Cartanmap}
c:K_0(\mathcal C) \rightarrow \operatorname{Gr}(\mathcal C), \quad [P] \mapsto[P]
\end{equation}
from the projective Grothendieck ring to the Grothendieck ring of $\mcC$.

\subsection{Non-semisimple modular data and Verlinde formulas}

The modular data $(S,T)$ of an MTC is arguably the most important invariant in the study of MTCs.  One way to understand the modular data of an MTC $\mcB$ is via the associated $(2+1)$-TQFT $(V_\mcB, Z_\mcB)$.  The vector space $V_\mcB(T^2)$ of the torus $T^2$ is a projective representation of its mapping class group $\operatorname{SL}(2,\mathbb Z)$ with the well-known generating set $(s,t)$. Moreover, $V_\mcB(T^2)$ has an obvious basis given by the simple object types of $\mcB$, called the particle basis in physics jargon.  Then the modular data $(S,T)$ of $\mcB$ is simply the representation matrices of $(s,t)$ with respect to this particle basis\footnote{There is a subtle phase that can be resolved in many different ways due to the fact that in general the representation of the mapping class group is only projective.}.  Notice that the modular data of an MTC consists of two matrices $(S,T)$.

One important application of the modular data of an MTC is to derive the fusion rules via the Verlinde formulas.  There are numerous proposed generalizations of the Verlinde formulas to the non-semisimple setting, but it is not clear what is a good generalization.  

In \cite{cohen2008characters}, two potential generalizations of the modular $S$-matrix are given.  We will call them premodular $S$-matrices.  The authors found two bases for the Higman ideal and studied the restriction of the Lyubashenko-Majid modular representation to the Higman ideal using these two bases.  We will call the inverse of the resulting Fourier transform matrix {\it Cohen-Westreich (CW) modular $S$-matrix.}  Combining with the $T$-matrix on the Higman ideal, we propose to regard the pair of matrices as a generalization of the modular data to the non-semisimple setting---the CW modular data $(S_{CW},T_{CW})$. The CW modular data reduce to the usual modular data in the semisimple setting, and result in
Verlinde formulas for the mixed fusion rules.  These facts provide evidence that CW modular data is a good generalization.

\subsubsection{Bases of $V(T^2)$}

It is important to recognize the difference between the TQFT representation of $\operatorname{SL}(2,{\mathbb Z})$ and the modular data for an MTC: the modular data consists of two matrices that lead to the same projective representation of the $\operatorname{SL}(2,\mathbb Z)$ of the corresponding TQFT.  Conversely to get modular data from the $\operatorname{SL}(2,\mathbb Z)$ representation of a TQFT, we need to find the correct basis of the vector space $V(T^2)$.  For MTCs, due to the connection to anyon model theory (e.g. see \cite{wang2010topological,rowell2018mathematics}), the correct basis is obvious following the corresponding physical interpretation: use the particle basis.  There are no such obvious physical interpretations of general NSS MCs yet, so it is not clear how one would choose bases.

Two vector spaces are important for our discussion: $\text{Nat}(Id_{\mcB})$ and its Higman ideal $\HigRB$.  The key technical fact that we will explore is the two distinguished bases in the case of factorizable ribbon Hopf algebras  \cite{cohen2008characters}: two bases of the Higman ideal that lead to the CW $S$-matrix that diagonalizes the mixed fusion rules.

There is a categorical formulation of the CW modular data and Verlinde formulas from \cite{cohen2008characters} for general NSS MCs.  We will follow mainly \cite{gainutdinov2020projective} for this categorical formulation.

\subsubsection{Coend and categorical Hopf algebras}

Given a modular category $\mcB$, the bifunctor $\otimes: \mcB^*\otimes \mcB\rightarrow \mcB$ has a coend $\mcL=\int^{X\in \mcB} X^*\otimes X$ and a dinatural transformation $n_X: X^*\otimes X \rightarrow \mcL$ for every $X\in \mcB$.  The coend $\mcL$ has a categorical Hopf algebra structure and a non-degenerate pairing $\omega: \mcL\otimes \mcL \rightarrow 1$.  The product, coproduct, unit, and counit will be denoted as $\mu_{\mcL}, \Delta_{\mcL}, \eta_{\mcL}, \epsilon_{\mcL}$ and $s_{\mcL}$ is the antipode, respectively.

The coend $\mcL$ has an integral $\Lambda_{\mcL}\in \textrm{Hom}(1,\mcL)$ and cointegral $\lambda_{\mcL}\in \textrm{Hom}(\mcL, 1)$ normalized such that $\Lambda_{\mcL}\cdot \lambda_{\mcL}=id$.

Given a factorizable Hopf algebra $H$ and the associated NSS MC $\mcB_H$=$H$-mod, the coend of $\mcB_H$=$H$-mod is just the dual Hopf algebra $H^*$ as an $H$-mod.  Then the center $Z(H)$ acts as natural transformations.  It follows that the categorical incarnations of the center $Z(H)$ in $\mcB_H$ is the algebra $\text{Nat}(Id_{\mcB_H})$ of natural transformations of the identity functor of $\mcB_H$.  Closely related to the center is the vector space $\textrm{Hom}(\mcL, 1)$---the vector space assigned to the torus $T^2$ in the TQFT from $\mcB_H$. 

There is a linear map $\Omega: Gr(\mcB)\rightarrow \text{Nat}(Id_{\mcB})$ from the Grothendieck algebra $Gr(\mcB)$ to $\text{Nat}(Id_{\mcB})$ defined as follows:

$$(\Omega_X)_Y := 
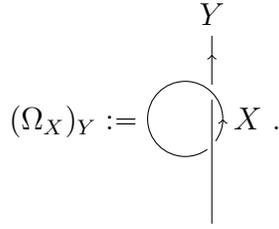
\begin{tikzpicture}[baseline=-0.25cm]
    \draw[->] (0, -1.5) -- (0, 0.15) (0, 0.35) -> (0, 0.75);
    \draw (0, 0.75) -- (0, 1) node[above] {$Y$};
    \draw[->] (0.05, -0.4) arc(-36:0:0.5) node (x) {} node[right] {$X$};
    \draw (x) arc(0:306:0.5);
\end{tikzpicture}.$$

Fixing an object $X$ in $\mcB$, then the picture defines a morphism in $\Hom(Y,Y)$ for each object $Y$ in $\mcB$.  Hence $\Omega_X$ is a natural transformation in $\text{Nat}(Id_{\mcB})$. 

\begin{definition} Given a modular category $\mcB$, 
    \begin{enumerate}
\item the Ray ideal of $\mcB$ is the ideal of $\ccenterB$ spanned by $\{\Omega_X|X\in Gr(\mcB)\}$.
        \item the Higman ideal of $\mcB$ is the ideal of $\ccenterB$ spanned by $\{\Omega_X|X\in Proj(\mcB)\}$.
    \end{enumerate}
\end{definition}

There is an isomorphism $$\gamma: \textrm{Hom}(\mcL,1)\rightarrow \ccenterB$$ defined by the following picture for $f\in \textrm{Hom}(\mcL,1)$:

$$\gamma(f)_Y = 
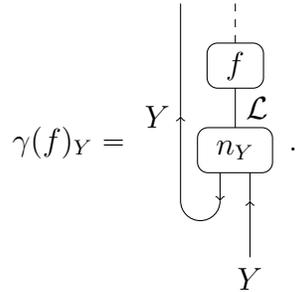
\begin{tikzpicture}[baseline=0.3cm, scale=0.75]
     \draw[rounded corners] (0.33, 0) rectangle (1.66, 0.8) (0.5, 1.5) rectangle (1.5, 2.3);
     \node at (1, 0.4) {$n_{Y}$};
     \node at (1, 1.9) {$f$};
     \draw (1, 0.8) -- node[right] {$\mcL$} (1, 1.5);
     \draw[dashed] (1, 2.3) -- (1, 3); 
     \draw[->] (0.73, -0.5) arc(0:-180:0.35) -- (0.03, 1) node[left]  {$Y$};
     \draw[->] (0.03, 1) -- (0.03, 3) (0.73, 0) -- (0.73, -0.5);
     \draw[->] (1.26, -0.5) -- (1.26, 0) (1.26, -1.5) node[below] {$Y$} -- (1.26, -.5);
\end{tikzpicture}\ .$$

There are also isomorphism between $\textrm{Hom}(\mcL, 1)$ and $\textrm{Hom}(1,\mcL)$

$$\delta: \textrm{Hom}(\mcL, 1)\rightarrow \textrm{Hom}(1, \mcL)$$
and
$$\pi: \textrm{Hom}(1, \mcL)\rightarrow \textrm{Hom}(\mcL,1)$$

defined by the following pictures:

$$\delta(f) = 
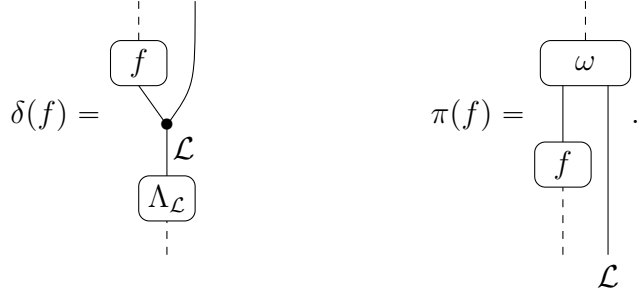
\begin{tikzpicture}[baseline=1cm, scale=0.75]
    \draw[rounded corners] (0, -0.4) rectangle (1, 0.4) (-0.5, 2) rectangle (0.5, 2.8);
    \node at (0.5, 0) {$\Lambda_{\mcL}$};
    \node at (0, 2.4) {$f$};
    \draw[dashed] (0.5, -1) -- (0.5, -0.4) (0, 2.8) -- (0, 3.5);
    \draw (0.5, 0.4) -- (0.5, 1.3) node {$\bullet$} -- (0, 2) (0.5, 1.3) .. controls (1, 2) and (1, 2) .. (1, 3.5);
    \node at (0.8, 0.9) {$\mcL$};
\end{tikzpicture}\hspace{3cm}\pi(f) = \begin{tikzpicture}[baseline=1cm, scale=0.75]
    \draw[rounded corners] (-0.8, 2) rectangle (0.8, 2.8) (-0.9, 0.2) rectangle (0.1, 1);
    \node at (0, 2.4) {$\omega$};
    \node at (-0.4, 0.6) {$f$};
    \draw[dashed] (0, 2.8) -- (0, 3.5) (-0.4, -1) -- (-0.4, 0.2);
    \draw (-0.4, 1) -- (-0.4, 2) (0.4, -1) node[below] {$\mcL$} -- (0.4, 2);
\end{tikzpicture}.$$

The categorical Fourier transform is then given as 
$$F=\pi \cdot \delta: \textrm{Hom}(\mcL, 1)\rightarrow \textrm{Hom}(\mcL,1)$$

The Lyubashenko-Majid $S$-transformation $S_{LM}$ is given as 
$$S_{LM}=\gamma\cdot F\cdot \gamma^{-1}: \ccenterB\rightarrow \ccenterB$$

Restricted to the Ray ideal, $S_{LM}(\phi_X)=\Omega_X$, where $\phi_X$ is defined by the following picture.

$$(\phi_X)_Y = 
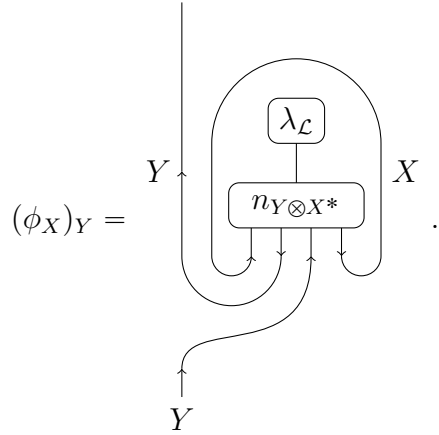
\begin{tikzpicture}[baseline=0cm, scale=0.75]
     \draw[rounded corners] (-0.2, 0) rectangle (2.2, 0.8) (0.5, 1.5) rectangle (1.5, 2.3);
     \node at (1, 0.4) {$n_{Y\otimes X^*}$};
     \node at (1, 1.9) {$\lambda_{\mcL}$};
     \draw (1, 0.8) -- (1, 1.5);
     \draw[<-] (0.2, -0.5) arc (0:-180:0.35) -- (-0.5, 1.5) arc (180:0:1.5) -- (2.5, 1) node[right] {$X$} -- (2.5, -0.5) arc (0:-180:0.35) ;
     \draw[->] (0.2, -0.5) -- (0.2, 0) (1.8, 0) -- (1.8, -0.5);
     \draw[->] (0.73, -0.5) arc(0:-180:0.88) -- (-1.03, 1) node[left]  {$Y$};
     \draw[->] (-1.03, 1) -- (-1.03, 4) (0.73, 0) -- (0.73, -0.5);
     \draw[->] (-1.03, -2.5) .. controls (-1.03,-1.5) and (1.26,-2.5) ..  (1.26, -0.5);
     \draw[->] (1.26, -0.5) -- (1.26, 0) (-1.03, -3) node[below] {$Y$} -- (-1.03, -2.5);
\end{tikzpicture}.$$

For a fixed object $X$ in $\mcB$, then the picture defines a morphism in $\Hom(Y,Y)$ for each object $Y$ in $\mcB$ leading to a natural transformation $\phi_X$ in $\text{Nat}(Id_{\mcB})$.


We fix a basis $B_{\textrm{Hig}}=\{\phi_{P_i}\}_{i=1}^m$ of $\HigRB$, where 
$\HigRB_p=\{P_i\}_{i=1}^m$ is a representative set of projective indecomposable classes of $\mcB$ for each independent column of the Cartan matrix.

\subsubsection{Categorical Cohen-Westreich modular $S$-matrix}

For any $B\in \HigRB_p$,  $$S_{LM}(\phi_{P_B})=\sum_{A\in \HigRB_p} (S^{-1}_{CW})_{AB}\phi_{P_A}$$

\begin{theorem}\cite[Proposition 8.8]{gainutdinov2020projective} 
Given a modular category $\mcB$, the CW modular $S$-matrix $S_{CW}$ diagonalizes the mixed fusion rules.
\end{theorem}

\subsubsection{Gainutdinov-Runkel formulas}

Given a modular category $\mcB$, we define two more matrices $M, J$ as follows:

\begin{enumerate}
    \item For any $B\in \irRB$, $$\phi_{P_B}=\sum_{A\in \HigRB_p} M_{BA}\phi_{P_A}$$

\item For any $B\in \HigRB$, $$\phi_{(P_B)^*}=\sum_{A\in Hig(\mcB)_p}J_{BA}\phi_{P_A}$$
\end{enumerate}

For each Steinberg object $Q$, there is a non-zero scalar $b_Q$ such that $\phi_Q(Q)=b_Q\cdot \text{Id}_Q$.

\begin{theorem} Given any $U,V\in \irRB$, and $P\in \HigRB_p$, 
    $$\sum_{W\in \irRB}M_{WP}\cdot N^W_{UV}(p)=\sum_{Q\in \StRB}b_Q\cdot (MS^{-1}_{CW})_{UQ}\cdot (MS^{-1}_{CW})_{VQ}\cdot(S^{-1}_{CW}J)_{QP}.$$
\end{theorem}

\subsection{Examples}

\subsubsection{$U_q SU(2)_1$}

By Prop.\ref{smallestrank}, the smallest rank of an NSS MC is $5$, which is realized by 
the NSS MC $U_q SU(2)_1$.  This example belongs to the sequence $U_qSU(2)_k$ for any $k\geq 1$, which is studied in Sec. \ref{sec:uqsl2}.

The NSS MCs $U_qSU(2)_k$ contains the irreducible and projective indecomposable modules of the small quantum groups $U_qsl_2$ at the root of unity $q=e^{\frac{2\pi i}{l}}$ for an odd number $l\geq 3$ and $k=l-2$.  The semi-simplification of $U_qSU(2)_k$ is the well-studied MTC $SU(2)_k$.

The NSS MC $U_qSU(2)_1$ for $l=3$ has $3$ simple object representatives: $V_1, V_2, V_3$ with projective covers $P_3=V_3$, $P_1=P(V_1)$, and $P_2=P(V_2)$.

\noindent The full fusion ring has generating fusion rules as:
\begin{align*}
V_1\otimes X\cong&X,\\
V_2\otimes V_2\cong&V_1\oplus V_3,~~~~V_2\otimes V_3\cong P_2,~~~~V_2\otimes P_1\cong 2V_3\oplus P_2,~~~~V_2\otimes P_2\cong2V_3\oplus P_1,\\
V_3\otimes V_3\cong& V_3\oplus P_1,~~~~V_3\otimes P_1\cong 2V_3\oplus 2P_2,~~~~V_3\otimes P_2\cong 2V_3\oplus 2P_2,\\
P_i\otimes P_j\cong& 4V_3\oplus 2P_1\oplus 2P_2, 1\leq i,j\leq 3.
\end{align*}

The Cartan matrix is:
$C=\begin{pmatrix}
2&2&0\\2&2&0\\0&0&1
\end{pmatrix}$.
It follows that $P_1$ and $P_2$ are $c$-equivalent.

From the full fusion ring and Cartan matrix, we obtain the mixed fusion rules and matrices.  Obviously, the usual fusion matrices associated with the simple objects are $N^1(t)= I_{3\times 3}$, $N^2(t)=\begin{pmatrix}0&1&0\\1&0&1\\2&2 &0\end{pmatrix}$ and $N^3(t)=\begin{pmatrix} 0&0&1\\2&2&0\\2&2&1\\
    \end{pmatrix}.$

Let  $[P_1]$ and  $[2V_3]$ be the basis for the class of  $[P_1]$ and  $[V_3]$, respectively.  It follows that:

\[
N^2(m)=
\begin{pmatrix}
0 & 1 \\
2 & 1
\end{pmatrix},~~
N^3(m)=
\begin{pmatrix}
1 & 1 \\
2 & 2
\end{pmatrix}.
\]

From Sec. \ref{sec:uqsl2}, we know 
$S_{CW}=\frac{1}{\sqrt{3}}
\begin{pmatrix}
1 & 1 \\
2 & {-1}
\end{pmatrix}$, hence
\[
S_{CW}  N^2(m) (S_{CW} )^{-1}=
\begin{pmatrix}
2 & 0 \\
0 & -1
\end{pmatrix},~~
 S_{CW}  N^3(m) (S_{CW} )^{-1}=
\begin{pmatrix}
3 & 0 \\
0 & 0
\end{pmatrix}.
\]

\subsubsection{$D{\mcK_2}$-mod}

The family of ribbon Hopf algebras $\{\mcK_n\}_{n=1}^\infty$, called the Nichols Hopf algebras, are not factorizable for any $n$.  $\mcK_1$ is Sweedler's Hopf algebra, and $\mcK_2$ appeared in Kerler's work \cite{kerler2003homology}.  $\mcK_n$ is unimodular for even $n$, but not unimodular for odd $n$.  To obtain NSS MCs, we go to their Drinfeld doubles. The Nichols Hopf algebras $\Kn$ and their doubles are studied in Section \ref{sec:doubleofnichols}.

For the modular category $D{\mcK_2}$-mod, 
there are $4$ simple modules $V_1, V_{K\bar{K}}, V_K, V_{\bar{K}}$ such that $V_K, V_{\bar{K}}$ are also projective.  $V_1, V_{K\bar{K}}$ are of dimension=$1$,$V_K, V_{\bar{K}}$ of dimension=$4$. Moreover, $V_1$ is the tensor unit $\mathbf 1$. The projective covers of $V_1, V_{K\bar{K}}$ will be denoted as $P_1, P_{K\bar{K}}$, which are of dimension=$16$.  The full fusion ring can be described as: 

\begin{align*}
    V_1\otimes X &\cong X & V_{K\bar K}\otimes V_{K\bar K}&\cong V_1\\
    V_{K\bar K}\otimes V_{K}&\cong V_{\bar K}& V_{K\bar K}\otimes V_{\bar K}&\cong V_{K}\\
    V_{K}\otimes V_{K}\cong V_{\bar K}\otimes V_{\bar K}&\cong P_{1} & V_{\bar K}\otimes V_{K}\cong V_{K}\otimes V_{\bar K}&\cong P_{K\bar K}\\
\end{align*}
\vspace{-40pt}
\begin{alignat*}{6}
    V_{K}\otimes P_{1}&\cong V_{K}&\otimes P_{K\bar K}&\cong V_{\bar K}&\otimes P_{1}&\cong V_{\bar K}&\otimes P_{K\bar K}&\cong 8 V_{K}&\oplus& 8V_{\bar K}&\\
    P_{1}\otimes P_{1}&\cong P_{1}&\otimes P_{K\bar K}&\cong P_{K\bar K}&\otimes P_{1}&\cong P_{K\bar K}&\otimes P_{K\bar K}&\cong 8 P_{1}&\oplus& 8P_{K\bar K}&
\end{alignat*}
where $X$ is any object in $D\mathcal{K}_2$-mod. With rows ordered as $P_1, P_{K\bar{K}}, P_K,P_{\bar{K}}$ and columns ordered as $V_1, V_{K\bar{K}}, V_K,V_{\bar{K}}$, the Cartan matrix is:
$$C=\begin{pmatrix}
    8&8&0&0\\8&8&0&0\\0&0&1&0\\0&0&0&1
\end{pmatrix}.$$

The idempotents for the four projective indecomposable modules $e_1, e_{K\bar K}, e_K, e_{\bar K}$ are
\begin{align*}
    e_1 &= 1+K+\bar K+K\bar K,\\
    e_{K\bar K} &= 1-K-\bar K+K\bar K,\\
    e_K &= (1+K-\bar K+K\bar K)\xi_1\xi_2\bar\xi_1\bar\xi_2,\\
    e_{\bar K} &= (1-K+\bar K+K\bar K)\xi_1\xi_2\bar\xi_1\bar\xi_2.
\end{align*}
The Higman ideal is
$$\Hig(D\mcK_2) = \spa\{1 - K\bar K, (K+\bar K)\xi_1\xi_2\bar\xi_1\bar\xi_2, (K-\bar K)(1 - \xi_1\bar\xi_1 - \xi_2\bar\xi_2 - \xi_1\xi_2\bar\xi_1\bar\xi_2)\}.$$
Let $h_1 = 1 - K\bar K$, $h_2 = (K+\bar K)\xi_1\xi_2\bar\xi_1\bar\xi_2$, and $h_3 = (K-\bar K)(1 - \xi_1\bar\xi_1 - \xi_2\bar\xi_2 - \xi_1\xi_2\bar\xi_1\bar\xi_2)$. Then, the two bases arising from the shifted Drinfeld map and Higman trace are (resp.)
\begin{alignat*}{3}
    \{8h_1&, &2(h_3 - h_2),\,& &2&(h_3 + h_2)\}\\
    \{8h_2&, -&2(h_1 + h_2),\,& &2&(h_1 - h_2)\}
\end{alignat*}

The $\mathrm{SL}(2, \bZ)$-action on the Higman ideal is generated by
$$s\mapsto S_{CW}=\begin{pmatrix}
    0 & 1/4 & -1/4\\
    2 & 1/2 & 1/2\\
    -2 & 1/2 & 1/2
\end{pmatrix},\hspace{30pt}t\mapsto T_{CW}=\begin{pmatrix}
    1 & 0 & 0\\
    0 & 1 & 0\\
    0 & 0 & -1
\end{pmatrix}.$$
The mixed fusion matrices and their diagonalizations are: $N^1 = I_3$ and
\begin{align*}
    N^{K\bar K} &= \begin{pmatrix}
        1 & 0 & 0\\
        0 & 0 & 1\\
        0 & 1 & 0
    \end{pmatrix}, &SN^{K\bar K}S^{-1} &= \begin{pmatrix}
        -1 & 0 & 0\\
        0 & 1 & 0\\
        0 & 0 & 1
    \end{pmatrix},\\
    N^{K} = N^{\bar K} &= \begin{pmatrix}
        0 & 1 & 1\\
        8 & 0 & 0\\
        8 & 0 & 0
    \end{pmatrix}, & SN^K S^{-1} = SN^{\bar K}S^{-1} &= \begin{pmatrix}
        0 & 0 & 0\\
        0 & 4 & 0\\
        0 & 0 & -4
    \end{pmatrix}.
\end{align*}

The center decomposes as 
$$Z(D\mcK_2) = \Hig(D\mcK_2)\oplus(1+K\bar K)\spa\{1,\, \xi_1\xi_2,\, \xi_1\bar\xi_1,\, \xi_1\bar\xi_2,\, \xi_2\bar\xi_1,\, \xi_2\bar\xi_2,\, \bar\xi_1\bar\xi_2,\, \xi_1\xi_2\bar\xi_1\bar\xi_2\}.$$
The latter part is denoted $Z_\Lambda$. The $\mathrm{SL}(2,\bZ)$-action on $\Hig(D\mcK_2)$ decomposes as $\bC_{\mathrm{triv}}\oplus N_1$, where $N_1$ is the level 2 Weil representation. The $\mathrm{SL}(2,\bZ)$-action on $Z_\Lambda$ decomposes as $4\bC_{\mathrm{triv}}\oplus(\bC^2_{\mathrm{std}})^{\otimes 2}$. Thus, the $\mathrm{SL}(2,\bZ)$-action on the full center of $D\mcK_2$ decomposes as $5\bC_{\mathrm{triv}}\oplus N_1\oplus(\bC^2_{\mathrm{std}})^{\otimes 2}$.

    To give expressions for the Lyubashenko-Majid and Drinfeld maps, we must define a handful of notations. Define $1^c = \xi_1\xi_2$, $\xi_1^c = \xi_2$, $\xi_2^c = \xi_1$, and $(\xi_1\xi_2)^c = 1$. Define $|1| = 0$, $|\xi_1|=|\xi_2|=1$, and $|\xi_1\xi_2| = 2$. Define $\bar 1 = 1$ and $\overline{\xi_1\xi_2} = \bar\xi_1\bar\xi_2$. For $w_1, w_2\in \{1, \xi_1, \xi_2,\xi_1\xi_2\}$ and $a,b\in\{0,1\}$, the Lyubashenko-Majid maps are generated by
    \begin{align*}
            S_{LM}(K^a\bar K^b \bar w_1 w_2) &= \frac{(-1)^{s(w_1, w_2)}}{2} w_1^c (1+(-1)^b K+(-1)^a\bar K+(-1)^{a+b}K\bar K))\bar w_2^c,\\
            T_{LM}(K^a\bar K^b \bar w_1 w_2) &= K^{a}\bar K^{b}(1+K-\bar K+K\bar K)(1 - \xi_1\bar\xi_1 - \xi_2\bar\xi_2 - \xi_1\xi_2\bar\xi_1\bar\xi_2)\bar w_1w_2,
    \end{align*}
    where $s(w_1, w_2) = \delta_{w_1=1} + \delta_{w_2=1} + \delta_{w_1=\xi_1} + \delta_{w_2=\xi_2} + (a+b)|w_1|$. The Drinfeld map is given by
    \begin{align*}
        F(\chi)&= \sum_{w_1, w_2\in \{1,\xi_1,\xi_2,\xi_1\xi_2\}} \frac{(-1)^{\lfloor |w_1|/2\rfloor + \lfloor |w_2|/2\rfloor}}{4}(\chi(\bar w_1 (1+\bar K)w_2) w_1\bar w_2 (1+\bar K) \\
        &+ \chi(\bar w_1 (1-\bar K)w_2) Kw_1\bar w_2 (1+\bar K)+ \chi(\bar w_1 (1+\bar K)Kw_2)w_1\bar w_2 (1-\bar K) \\
        &+ \chi(\bar w_1 (1-\bar K)Kw_2)Kw_1\bar w_2 (1-\bar K)).
    \end{align*}

\section{Factorizable ribbon Hopf algebras}\label{sec:factorizableHA}

The NSS MCs that we study explicitly are representation categories of finite-dimensional modules of two families of factorizable ribbon Hopf algebras.  One family is the small quantum groups $U_qsl(2)$ at roots of unity, and the other family is the Drinfeld doubles of Nichols Hopf algebras, also referred to as doubled Nichols Hopf algebras.

We carry out a detailed study of the modular data and fusion rules for these two families of NSS modular categories in order to compare them with MTCs.  One hope is that the CW modular data of NSS MCs and the modular data of their semi-simplification can be combined to provide tools for a low-rank classification and insights for a structure theory for NSS MCs.

\subsection{Notations and terminologies for Hopf algebras}

We will follow closely the notation and terminology of \cite{cohen2008characters} as much as possible.  But due to conflicts, we sometimes have to make changes. 

Recall that a ribbon Hopf algebra $(H,R,v)$ is a quasitriangular Hopf algebra $(H,R)$ with a ribbon element $v$, where the universal $R$-matrix is written as $R=\sum r_1\otimes r_2 \in H\otimes H$ and $R^{21}=\sum r_2\otimes r_1 \in H\otimes H$. A quasitriangular Hopf algebra $(H, R)$ is called ribbon if there exists a central element $v \in H$ such that
\begin{equation}\label{eq: ribbon}
    \Delta(v)=\left(R_{21} R\right)^{-1}(v \otimes v), \quad \varepsilon(v)=1, \quad \text { and } \quad S(v)=v .
\end{equation}

Let $\{\chi_i\}_{i=1}^n$, $\{e_i\}_{i=1}^n$, $\{He_i\}_{i=1}^n$, and $\{p_{e_i}\}_{i=1}^n$ be the corresponding
irreducible characters, primitive orthogonal idempotents, projective indecomposable modules, and characters of the projective modules $He_i$.

\subsection{Drinfeld and Frobenius isomorphisms}

\subsubsection{Drinfeld map}

The element $Q=R^{21}R=\sum_i m_i\otimes n_i \in H\otimes H$ is called the Drinfeld matrix, where $R^{21}=\sum r_2\otimes r_1$.  
Given such a $Q$ of a quasitriangular Hopf algebra $H$, the Drinfeld
map $f_{Q}: H^{*}\rightarrow H$ is defined as
$$f_{Q} (\beta)=(\beta\otimes id)Q=\sum_i\beta(m_i)n_i, \beta \in H^*.$$

\subsubsection{Factorizable ribbon Hopf algebra}

\begin{definition}
 A quasitriangular Hopf algebra is factorizable if the Drinfeld map $f_{Q}: H^*\rightarrow H$ is an isomorphism of vector spaces.   
\end{definition}

\subsubsection{Space of q-characters}

For a Hopf algebra $H$, the space $\text{qCh}(H)=\{\beta\in H^*~|~\beta(xy)=\beta(s_H^2(y)x)\ \ \forall x,y\in H\}$ is called the space of q-characters of $H$. Let $\lambda\in H^*$ and $\Lambda\in H$ be the integrals of $H^*$ and $H$ normalized as $\lambda(\Lambda)=1$. 

\begin{theorem}
\cite{Dr}~Restricting to $\text{qCh}(H)$, the Drinfeld map $f_Q$ is an isomorphism of the commutative algebras between $\text{qCh}(H)$ and the center ${Z}(H)$.
\end{theorem}

\subsubsection{Higman ideal}

The Higman ideal $\text{Hig}(H)$ of a Hopf algebra $H$ is defined as the image of the trace map 
\begin{align*}
\tau(x)=\sum_{(\Lambda)}s_H(\Lambda_{(2)})x\Lambda_{(1)}G^{-1}=ad_\Lambda(xG^{-1}).
\end{align*}

For a unimodular Hopf algebra $H$ with $s_H^2$ being inner by some $G \in H$, e.g. $U_qsl(2)$,  it is shown in \cite{cohen2008characters} that the image of $\text{qCh}(H)$ under $f_Q$ coincides with the center $ {Z}(H)$, and $I(H)\subseteq H^*$ spanned by the q-characters of projective modules is mapped to the Higman ideal.

\subsubsection{Frobenius map}

The Frobenius map $$\Psi: H_{H^*}\rightarrow H^*_{H^*} $$ by $\Psi(a)(b)=\lambda(S(a)b)$,
where $H^*$ is a right $H^*$-module under multiplication.

\subsection{Lyubashenko-Majid modular representations}

Lyubashenko and Majid defined two remarkable representations of the modular group $\operatorname{SL}(2,\mathbb Z)$ on the center $Z(H)$ for any factorizable ribbon Hopf algebra $H$ in \cite{lyubashenko1994braided}. These two representations are essentially the TQFT representations of the mapping class groups of the torus and there are no obvious bases of $Z(H)$ to make them into matrices.

For MTCs, there are different conventions of the modular $S$-matrices.  Closely related is the fact that for the standard generators $(s,t)$ of $\operatorname{SL}(2,\mathbb Z)$, we have $(st^{-1})^3=1, (st)^3=s^2$.  For MTCs, anyon theory leads to a central charge of an MTC, so different choices can be distinguished by the central charges.  But for NSS MCs, we do not have such an invariant, so in principle either one of the Lyubashenko-Majid $S$ map is good, which are inverse to each other.  We will focus on the $-$-version of \cite{lyubashenko1994braided} here.  The $-$-version of the $S,T$ maps will be denoted as  ${S}_{LM}, {T}_{LM}$, and the $+$-version as  ${S}^+_{LM}, {T}^+_{LM}$.

\subsubsection{Formulas for Lyubashenko-Majid maps}

The $+$-version of the Lyubashenko-Majid $S$-map is simply $$S^+_{LM}=f_Q\cdot \Psi: H\rightarrow H,$$ i.e. the composition of the Frobenius map followed by the Drinfeld map. 

In explicit formulas, the Lyubashenko-Majid maps are given as:
for all $x\in H$,
$${S}^+_{LM}(x)=(\text{id}\otimes\lambda)(R^{21}(1\otimes x)R), \ \ {T}^+_{LM}(x)=v\cdot x.$$

The $-$-version is defined similarly:
for all $x\in H$, 
$${S}_{LM}(x)=(\text{id}\otimes\lambda)(R^{-1}(1\otimes x)R_{21}^{-1}), \ \ {T}_{LM}(x)=v\cdot x.$$

The maps ${S}_{LM}, {T}_{LM}: H\rightarrow H$ satisfy the
modular identities
$$({S}_{LM} {T}_{LM})^3=\kappa {S}_{LM}^2,\ \ {S}_{LM}^2=s_{H}^{-1}$$
where $\kappa$ is some constant, and $s_H$ is the antipode of $H$.

\subsubsection{Center $Z(H)$ as an $\operatorname{SL}(2,\mathbb Z)$ module}
 Restricted to the center $Z(H)$ of $H$,
${S}^4=s_H^{-2}=id_{{Z}(H)}$ since
$s_H^{-2}(x)=G^{-1}xG$ for all $x\in H$. Here $G$ is the balancing element of $H$. It follows that the Lyubashenko-Majid maps $S_{LM}$ and $ T_{LM}$ give rise to a
projective representation of $\operatorname{SL}(2,\mathbb{Z})$ on ${Z}(H)$.

As in \cite{lachowska2021remarks}, when restricted to the center ${Z}(H)$ of
$H$, the definition of this representation can be slightly modified as follows: for
$a\in {Z}(H)$
$${S}_{LM}(a)=(\lambda\otimes 1)(S(a)\otimes 1)R_{21}R=\phi_R(f_Q^{-1}(a)),$$
$${T}_{LM}(a)=\kappa {S}_{LM}^{-1}(v^{-1}({S}_{LM}(a))).$$

\subsubsection{Higman ideal as an $\operatorname{SL}(2,\mathbb Z)$-module}

The Higman ideal is invariant under $S_{LM}$ (see e.g. \cite{cohen2008characters}), hence the Higman ideal is a $\operatorname{SL}(2,\mathbb Z)$-submodule of $Z(H)$. Moreover, it is a direct summand of $Z(H)$ (see e.g. \cite{lachowska2021remarks}).

\subsection{Cohen-Westreich modular data and Verlinde formulas}

\subsubsection{Cartan matrix of Hopf algebras}

The Cartan matrix of a factorizable Hopf algebra $H$ can also be computed directly from the Hopf algebra data.  The Cartan matrix
of $H$ is the $r\times r$ matrix $C$ with $c_{ij} = \text{dim} e_iAe_j$, which is symmetric. We have 
$$p_{e_j}=\sum_{i=1}^r c_{ij}\chi_i.$$

Note also that for any idempotent $e, p_e = \chi_{Ae}$. Hence, 
$\langle p_{e_i}, e_j\rangle= c_{ij} .$

\subsubsection{Pre-$S$ matrices}

There are two potential generalizations of the $S$-matrix that we will call pre-$S$-matrices. These two $r\times r$ matrices together with the Cartan matrix $C$  are used to derive the CW modular $S$-matrix $S_{CW}$, which is later used for the diagonalization of the mixed fusion rules. 

The two pre-$S$-matrices will be denoted by $\tilde{S}$ and $\hat{S}$, and note that for MTCs, these two matrices are equal and are the usual $S$-matrix.  Their entries are given as \cite[Eq(21)]{cohen2008characters}
\begin{equation}
\tilde{s}_{i j}=\left\langle\widehat{f}_Q\left(\chi_i\right), \chi_j\right\rangle, \quad \widehat{s}_{i j}=\left\langle\widehat{f}_Q\left(\chi_i\right), \widehat{\Psi} s_H\left(e_j\right)\right\rangle .
\end{equation}
Here $\chi_i$'s are the irreducible character, $s_H$ is the antipode for $H$. The shifted Frobenius map $\widehat{\Psi}: H \rightarrow H^*$ is  defined in \cite[Eq(17)]{cohen2008characters}. Particularly, $\widehat{\Psi}^{-1}(I)=\operatorname{Hig}(H)$, where $I:=I(H)$ is spanned by the characters of all projective indecomposable $H$-modules \cite[Th 2.8]{cohen2008characters}. The shifted Drinfeld map $f_Q: H^* \rightarrow H$ is the Drinfeld map defined as above and in \cite[Section 3]{cohen2008characters}. $\widehat{f}_Q$ is defined in \cite[Eq(20)]{cohen2008characters}.   Note $\widehat{f}_Q(I)=\operatorname{Hig}(H)$\cite[Prop 3.5]{cohen2008characters}.

To obtain a symmetric pre-$S$-matrix, one forms the following $r\times r$ matrix $$(C\widehat{S})_{ij}=\left\langle\widehat{f}_Q\left(p_{e_i}\right), \widehat{\Psi} s_H\left(e_j\right)\right\rangle.$$ It is symmetric by \cite[Lemma 3.10]{cohen2008characters}. Note if the category is semisimple, this matrix is also the same as the usual modular $S$.

Finally, to get the CW modular $S$-matrix that diagonalizes the mixed fusion matrices in Eq. (\ref{eq:fusion}), one computes $S_{CW}=C_n^{-1}(C\widehat{S})_n$, which is an $n\times n$ matrix.  The subscript $n$ indicates we are restricting to the major minor of the matrices. Moreover, $S_{CW}$ diagonalizes the fusion matrices $N^i(m)$.
Notice that $S_{CW}$ is the usual $S$-matrix if the category is semisimple.

\subsubsection{Two bases of Higman ideal}
The following two bases for the Higman ideal $\operatorname{Hig}(H)$ is from \cite[Eq (23)]{cohen2008characters}:
\begin{equation}\label{eq:twobases}
    B_{\chi}=\left\{\widehat{f}_Q\left(p_{e_j}\right)\right\}_{j=1}^n, \quad B_{\tau}=\left\{\tau\left(e_j\right)\right\}_{j=1}^n .
\end{equation}

\subsubsection{CW modular $S$-matrix}

$S_{CW}$ is in fact the change of bases matrix of the two bases above:\footnote{Our $S_{CW}$ is the inverse of the $F$ matrix in \cite{cohen2008characters}.}

$$\tau({e_j})=\sum_{k=1}^n (S_{CW})_{kj} f_Q(p_{e_k}).$$

\subsubsection{Mixed fusion rules}

The representation category of finite-dimensional modules of $H$ is a modular category $\mcB_H$.  Let $r$ be the rank of $\mcB_H$. The mixed fusion rules of $\mcB_H$ are given by 
\begin{equation}\label{eq:fusion}
    \chi_i p_{e_j}=\sum_{k=1}^n N_{k j}^i(m) p_{e_k}.
\end{equation}
Here $\chi_i$ is an irreducible character, $p_{e_j}$ is a character of an indecomposable projective module. Note $n$ is the rank of the Cartan matrix. So we obtain $r$ many $n\times n$ matrices $N^i(m)$.

\subsubsection{Diagonalizing mixed fusion rules}
Let $L_i$ denote the restriction of the operator left multiplication by $\widehat{f_Q}\left(\chi_i\right)$ to $\operatorname{Hig}(H)$, see \cite[Eq (27)]{cohen2008characters}. 
The matrix $L_i$ with respect to the basis $\mathcal{B}_\chi$ is 
\begin{equation}
    \mathbf{N}^i(m)=\left(N_{j k}^i(m)\right)
\end{equation}
and the matrix of $L_i$ with respect to the basis $\mathcal{B}_\tau$ is
\begin{equation}
    \mathbf{D}^i=\operatorname{Diag}\left\{d_1^{-1} \tilde{s}_{i 1}, \ldots, d_n^{-1} \tilde{s}_{i n}\right\}
\end{equation}

\subsubsection{Verlinde formulas}

From \cite{cohen2008characters}, we have 
$$\mathbf{N}^i(m)=S^{-1}_{CW}\mathbf{D}^i S_{CW}.$$
Let $S_{CW}=(s_{jk})_{1\leq j,k\leq n}$, $S^{-1}_{CW}=(f_{jk})_{1\leq j,k\leq n}$, then for any $1\leq i\leq r, 1\leq j,k\leq n$,  
$$ N^i_{jk}(m)=\sum_{1\leq l\leq n} \frac{f_{jl}\tilde{s}_{il}s_{lk}}{d_l}.$$

\section{Semi-simplification of modular categories}\label{sec:SSofMCs}

There are ample applications of NSS MCs to physics.  The standard application is to non-unitary physics via logarithmic conformal field theory or mathematically non-semisimple vertex operator algebra.  It is not inconceivable that every modular category can be realized as the representation category of a vertex operator algebra, which has been conjectured for MTCs \cite{tener2017classification}.  

We are interested in a possible different application of NSS MCs to unitary physics via symmetry breaking.  Suppose an NSS MC $\mcB_H$ results from the representations of a finite-dimensional factorizable ribbon Hopf algebra $H$.  The projective cover $P(V_i)$ of an irreducible module $V_i$ is indecomposable as $H$-modules.  But  $P(V_i)$ obviously decomposes as vector spaces, i.e. $P(V_i)\cong V_i\oplus V_i'$ as vector spaces for some vector space $V_i'$.  If there is a Hopf subalgebra $H'$ such that every projective $P(V_i)$ decomposes into $V_i$ and another $H'$-module $V_i'$, we will say that $H$ symmetry breaks to the subalgebra $H'$, and will leave an investigation of this symmetry breaking mechanism to the future.  

Symmetry breaking dually leads to semisimple quotients of the representation categories.  This is called purification or semi-simplification in the literature \cite{turaev2010quantum,etingof2021semisimplification}.  One interesting observation follows from the doubled Nichols Hopf algebra examples: the semi-simplification of an NSS MC is not always non-degenerate.   

\subsection{Semi-simplification tensor functor}

It is well-understood that the objects of a category are merely for convenience: a category can be defined using just the morphisms.  Given a modular category $\mcB$, one is interested in a quotient category that is semisimple.  One way to do this is to keep the same class of objects and change the set of morphisms.  Such a "quotient" is one way to define the condensation of algebras in an MTC, which corresponds well to physics.  The semi-simplification or purification of an NSS MC can be performed following the same strategy: we will keep the same class of objects while quotient of negligible morphisms.

Given a modular category $\mcB$ and two objects $X,Y$, a morphism
$f : X\rightarrow Y$  is negligible if for any morphism $g : Y\rightarrow X$,  we have $tr( f g) = 0$.  The tensor product of a negligible morphism with an arbitrary morphism is negligible.
Hence, the negligible morphisms in $\mcB$ form a two-sided ideal with
respect to composition and tensor product.  We will denote the negligible ideal of $Hom(X,Y)$ as $\mcI_N(X,Y)$.

In general, an FTC has many more objects than simples and projective indecomposables (e.g. see Section 1.6 of \cite{gainutdinov2006kazhdan}).  For our semi-simplification of a modular category $\mcB$, we will focus on the purification of $\mcB_{QD}$ of quasi-dominated objects.

\begin{definition}
Given a modular category $\mcB$, the semi-simplification $\bar{\mcB}$ of $\mcB$ is the category $\bar{\mcB}$ with the same objects as $\mcB_{QD}$, and the morphism set of any two objects $X,Y$ of $\bar{\mcB}$ is the quotient morphism set $Hom_{\bar{\mcB}}(X,Y)=Hom_\mcB(X,Y)/ \mcI_N(X,Y).$  The semi-simplification functor from $\mcB$ to $\bar{\mcB}$ is the identity map on objects and sends each morphism in $\mcB$ to its quotient class in $\bar{\mcB}$.
\end{definition}

It is straightforward to verify that 
\begin{proposition}
    The semi-simplification functor from $\mcB$ to $\bar{\mcB}$ is a tensor functor.
\end{proposition}

Interesting subtleties arise for semi-simplification.  First the semi-simplification of an NSS MC is not necessarily modular as Prop. \ref{non-modular} shows.  Secondly the semi-simplification could be wild as footnote 9 of \cite{berger2023non} shows if we do not restrict to $\mcB_{QD}$.

\subsection{Quantum groups at roots of unity}

Let $\mfg$ be the Lie algebra of a simple Lie group $G$ and $q=e^{\frac{2\pi i}{m(\check{h}+k)}}, k\geq 1$, where $\check{h}$ is the dual Coxeter number of $\mfg$ and $m=1,2,3$ depending on the Lie algebra\footnote{$m=1$ for ADE, $m=2$ for B,C,$F_4$, and $m=3$ for $G_2$.}.  The quantum group at root of unity as defined in \cite{turaev2010quantum,bakalov2001lectures} will be denoted as $U_q\mfg$, which is a factorizable ribbon Hopf algebra.  We will denote the resulting modular category from the representation category of $U_q\mfg$ as $U_qG_k$ and refer it as the modular category of $G$ at level$=k$ following the physics convention.

As in \cite{turaev2010quantum,bakalov2001lectures}, a unitary MTC  $G_k$\footnote{In the literature, $G_k$ can denote any one of the three essentially equivalent things: the Witten-Chern-Simons theory, the WZW minimal model CFT, and the unitary MTC here.} can be constructed from the quantum group at roots of unity $U_q\mfg$.  It follows that the semi-simplification of $U_qG_k$ is $\overline{U_qG_k}=G_k$ \cite{turaev2010quantum}.

\subsection{Decomposition of the center as $\operatorname{SL}(2,\mbbZ)$ module}

Given a factorizable ribbon Hopf algebra $H$ with resulting modular category $\mcB_H$, its center $Z(H)$ is remarkably an $\operatorname{SL}(2,\mbbZ)$-module \cite{lyubashenko1994braided}. Moreover, the Higman ideal $\operatorname{Hig}(H)$ is an $\operatorname{SL}(2,\mbbZ)$ submodule.  It would be interesting to decide to what extent the center $Z(H)$ as $\operatorname{SL}(2,\mbbZ)$-module is determined by three things: the modular data of the semi-simplification $\overline{\mcB}_H$, the CW modular data of $\mcB_H$, and the Gainutdinov-Runkel formulas.

\section{Comparison with modular tensor categories}\label{sec:comparison}

MTCs are interesting in their own right and find many applications in physics and quantum computing.  We expect MCs will have a parallel structure theory, and similar applications.  As a first step, we wonder what properties of MTCs generalize and what fail? In this section, we investigate this question for a few basic theorems for MTCs: the rank-finiteness theorem, the abelian-ness of Galois fields, and the congruence kernel of the modular representation.

\subsection{Refined rank-finiteness}

Rank-finiteness for MTCs is a refinement of the Ocneanu rigidity for fusion categories \cite{bruillard2016rank}.  One key difference between fusion categories and MTCs is that the number field of adjoining the eigenvalues of the fusion rules of MTCs has an abelian Galois group.  This fact results in strong constraints on the quantum dimensions of MTCs.  By Corollary \ref{cor:rank-finiteness-fails}, rank-finiteness fails for modular categories; the infinite family of distinct modular categories from the representations of doubled Nichols Hopf algebras are all of total rank=$6$ and are nonequivalent even as categories.  It begs an interesting question: is there a generalization of rank-finiteness from MTCs to general MCs?

An interesting first step is: does finiteness hold for MCs with the same Cartan matrix? The categories $D\Kn$-mod have different fusion rules for every $n$, so they do not provide a counterexample to this proposal.

\subsection{Galois group for modular categories}

A Galois group can be defined for the fusion rules of an MTC, the modular data of an MTC, and even the full data of an MTC \cite{davidovich2013arithmetic}.  In this case, the Galois groups of the fusion rules and modular data are closely related.  The generalization to MCs is complicated, and we will define potentially many different Galois groups of an MC.

\subsubsection{Galois group of mixed fusion rules}

The eigenvalues of the mixed fusion matrices $N^i$ will be denoted as $\{\lambda_{i,j}\}, 1\leq i\leq r, 1\leq j\leq n$.  We could consider the Galois field $K_M=\mathbb{Q}(\lambda_{ij}),1\leq i\leq r, 1\leq j\leq n$.

In the case of factorizable ribbon Hopf algebras, the mixed fusion matrices $N^i$ have eigenvalues $\{d_1^{-1} \tilde{s}_{i 1}, \ldots, d_n^{-1} \tilde{s}_{i n}\}.$  Since $d_j$ are integers, the Galois fields become  
$K_M=\mathbb{Q}(\tilde{s}_{ij}),1\leq i\leq r, 1\leq j\leq n$.  

We do not know if the Galois group of this extension is abelian, though all our examples have abelian Galois groups.

\subsubsection{Galois group of Cohen-Westreich modular data}

Similarly, we define $K_{CW}=\mathbb{Q}({s}_{ij}),1\leq i\leq n, 1\leq j\leq n$.
We also do not know if the Galois group of this extension is abelian, and again all our examples have abelian Galois groups.

\subsection{Congruence subgroup for Higman representation}
The group $\operatorname{SL}(2,\mathbb Z)$ is generated by the elements $ s=\begin{pmatrix}0&-1\\1&0\end{pmatrix} $ and $ t = \begin{pmatrix}1&1\\0&1\end{pmatrix}$. Let $\mathcal C$ be an MTC with modular data $(S, T)$. The assignments $ s\mapsto S$ and $  t\mapsto T$ define a projective representation of $\operatorname{SL}(2,\mathbb Z)$ which can be lifted to linear representations (see e.g.,\cite[Sec. 3.1]{bakalov2001lectures}).

One of the deep theorems about MTCs is the congruence kernel subgroup theorem \cite{ng2010congruence, dong2015congruence}.  This has been generalized to super-modular categories and has a complicated picture for representations of the braid groups \cite{ricci2017congruence, bonderson2018congruence}.  We verify that the congruence kernel statement holds for the Higman representations for both families of modular categories.  Therefore, we conjecture that the congruence kernel conjecture generalizes to the Higman representation.

\subsubsection{Quadratic modules}
\begin{definition}
    Let $M$ be a finite $\mathbb Z/p^\lambda\mathbb Z$ module, where $\lambda$ is a natural number and $p$ is a prime. A quadratic form on $M$ is a map $Q: M\to \mathbb Q/\mathbb Z$ such that 
\begin{enumerate}
    \item  $Q(-x)=Q(x)$ for all $x \in M$.
\item  $B(x, y):=Q(x+y)-Q(x)-Q(y)$ defines a bilinear map from $M\times M\to \mathbb Q/\mathbb Z$. 
\end{enumerate}
\end{definition}
The quadratic form $Q$ is called non-degenerate if $B(x,y)$ is non-degenerate, that is, $B(x, y)=0$ for all $y=0$ implies $x=0$. The pair $(M, Q)$ is called a quadratic module.

Quadratic modules are closely related to pointed premodular categories (E.g., \cite[Section 8.4]{EGNO}). Specifically, on one side, for any pointed modular category $\mathcal{C}$, the isomorphism classes of simple objects, along with the function determined by their braidings, constitute a quadratic module. Conversely, starting with a quadratic module $(M, Q)$, one can apply the Eilenberg-MacLane theorem \cite{eilenberg1947cohomology, eilenberg1947cohomology2} to obtain a pointed premodular category $\mathcal{C}(M, Q)$ \cite[Section 3]{joyal1993braided}.  The category $\mathcal C(M, Q)$ is modular if $Q$ is non-degenerate.

\subsubsection{Weil representations} Let $\mathbb C^M$ denote the space of $\mathbb C$-valued functions on $M$.
By \cite[Satz 2, p. 467]{Nobs1}, there is a projective representation 
$$
W(M, Q): \operatorname{SL}\left(2, \mathbb Z/p^\lambda\mathbb Z\right) \rightarrow \operatorname{PGL}(\mathbb C^M)$$
defined by
$$\begin{aligned}
  tf(x) & =e^{2 \pi i Q(x)} f(x) \\
 s^{-1} f(x) & =\frac{\tau_{M,Q}}{|M|} \sum_{y \in M} e^{2 \pi i B(x, y)} f(y)
\end{aligned}$$
where $\tau_{M,Q} = \sum_{a\in M} e^{2\pi i Q(a) } $ is the Gauss sum of the quadratic module $(M,Q)$ and $f\in\mathbb C^M$. Through investigations into Weil representations, it is proved in \cite{ng2023symmetric} that every finite-dimensional congruence representation of $\operatorname{SL}(2,\mathbb{Z})$ is symmetrizable.

Note $W(M, Q)$ coincides with the projective representation derived from the pointed MTC $\mathcal{C}(M, Q)$.  Let $\{\delta_x \in \mathbb{C}^M \, \colon \, x \in M\}$  be the standard basis of $\mathbb C^M$. The modular data for $\mathcal{C}(M, Q)$ is described by
\[ S_{x,y}=e^{-2\pi iB(x,y)},\qquad T_{x,y}=\delta_{x,y}e^{2\pi iQ(x)},\]
where $x,y\in M$.

\subsubsection{$\operatorname{SL}(2,\mathbb Z)$-representations from $U_qsl(2)$}

\begin{lemma}\label{lem:decomposition} Let $l = 2h+1\ge3$ be an odd number. 
  The projective representation of $\operatorname{SL}(2, \mathbb Z)$ from some pointed MTC $\mathcal C(\mathbb Z/l\mathbb Z, Q)$ can be decomposed as $V^{\operatorname{even}}\oplus V^{\operatorname{odd}}$,where $V^{\text {even }}$ is the same as $\mathcal{P}_{h+1}$ and $V^{\text {odd }}$ is the same as $\mathcal{V}_h$ in Thm. \ref{Kerler}.
\end{lemma}
\begin{proof}
    Let $q$ be a primitive $l^{th}$ root of unity. Without loss of generality, let $q=e^{2\pi i/l}$.  Consider the quadratic form $Q: \mathbb{Z} / l \mathbb{Z} \rightarrow \mathbb{Q} / \mathbb{Z}$ defined as $Q(a)=-ha^2 / l$.   The modular data of $\mathcal C(\mathbb Z/l\mathbb Z, Q)$ is given by 
    $S=(q^{ab})_{0\le a, b<l}$, $T= \operatorname{diag}(q^{-ha^2})_{0\le a<l}$. 
    Let $w_0=v_0, w^+_i=\frac{1}{2}(v_i+v_{h+i})$ and $w_i^-=\frac{1}{2}(v_{h+1-i}-v_{2h+1-i})$, where $i=1,\cdots, h$. Then the span of $\{w_0, w^+_1, \cdots, w^+_{h}\}$ and $\{w^-_1, \ldots, w^-_{h}\}$ form invariant subspaces $V^{\operatorname{even}}$ and $V^{\operatorname{odd}}$, respectively. In terms of the change of basis, we have
   \begin{align*}
	S^{\operatorname{even}}=&
	\begin{pmatrix}
	1 & 1 &   & \cdots &   & 1\\
	2 & q+q^{-1} &   & \cdots &   & q^{h}+q^{-h}\\
	&   & \ddots &   &   &  \\
	\vdots & \vdots &   & q^{ab}+q^{-ab} &   & \vdots\\
	&   &   &   & \ddots &  \\
	2 & q^{h}+q^{-h} &   & \cdots &   & q^{h^2}+q^{-h^2}\\
	\end{pmatrix}_{{(h+1)\times (h+1)}}\\
	T^{\operatorname{even}}=& 
	\begin{pmatrix}
	1 & 0 &   & \cdots &   & 0\\
	0 & q^{-h} &   & \cdots &   & 0\\
	&   & \ddots &   &   &  \\
	\vdots & \vdots &   & q^{-ha^2} &   & \vdots\\
	&   &   &   & \ddots &  \\
	0 & 0 &   & \cdots &   & q^{-hh^2}\\
	\end{pmatrix}_{(h+1)\times (h+1)}
\end{align*}

\begin{align*}
S^{\operatorname{odd}}=&
\begin{pmatrix}
q-q^{-1} & q^2-q^{-2} &   & \cdots &   & q^{h}-q^{-h}\\
q^{2}-q^{-2}& q^{4}-q^{-4} &   & \cdots &   & q^{2h}-q^{-2h}\\
&   & \ddots &   &   &  \\
\vdots & \vdots &   & q^{ab}-q^{-ab} &   & \vdots\\
&   &   &   & \ddots &  \\
q^{h}-q^{-h} & q^{2h}-q^{-2h} &   & \cdots &   & q^{h^2}-q^{-h^2}\\
\end{pmatrix}_{h \times h}\\
\end{align*}
\begin{align*}
    T^{\operatorname{odd}}=&
\begin{pmatrix}
q^{-h} & 0 &   & \cdots &   & 0\\
0 & q^{-h2^2} &   & \cdots &   & 0\\
&   & \ddots &   &   &  \\
\vdots & \vdots &   &  q^{-ha^2} &   & \vdots\\
&   &   &   & \ddots &  \\
0 & 0 &   & \cdots &   & q^{ -hh^2}\\
\end{pmatrix}_{h \times h}
\end{align*} 

It is straightforward to check that up to a scalar, the $S$ matrices match with those in Eq. (\ref{Uqsl2SHig}) and Eq. (\ref{Uqsl2SSO3}).  Multiply $T^{\operatorname{even}}$ by $q^h$, then the $a$th term is $q^{h(1-a^2)}$, $a=0,\ldots, h$. On the other hand, we write the $a$th term for $T_N$ in Eq. (\ref{Uqsl2THig}) as $q^{\frac{1}{2}(a+1)(2h+a)}$ after dividing by $\kappa$. As $q^{\frac{1}{2}(a+1)(2h+a)}/q^{h(1-a^2)}=q^{\frac{1}{2}la(a+1)}=1$, we know $T^{\operatorname{even}}$ and $T_N$ are the same up to multiplying with a scalar. The same argument works for comparing $T^{\operatorname{odd}}$ and $T_V$.
\end{proof}

\begin{theorem}
     Let $l$ be an odd number and $q$ be a primitive $l^{th}$ root of unity. The $\operatorname{SL}(2,\mathbb Z)$-representation on the Higman ideal of $U_qsl(2)$ has kernel congruence subgroup of $\operatorname{SL}(2,\mathbb Z)$.
\end{theorem}
\begin{proof}
   By Lemma \ref{lem:decomposition}, we know such representation is a subrepresentation of the $\operatorname{SL}(2,\mathbb Z)$-representation associated with a  pointed MTC, which has a congruence kernel.
\end{proof}

\subsubsection{Doubled Nichols Hopf algebras}  
\begin{proposition}
    The $\operatorname{SL}(2,\mathbb Z)$-representation on the Higman ideal of $D \mathcal{K}_n$ has kernel congruence subgroup of $\operatorname{SL}(2, \mathbb Z)$ of level 2.
\end{proposition}
\begin{proof}
    This follows from Theorem \ref{DKn-CW-decomp}.
\end{proof}

Based on the results on the modular data computations,
we propose the following conjecture:

\begin{conj}
Let $\mathcal{C}$ be a non-semisimple modular category and let $\rho:(\mathfrak{s}, \mathfrak{t}) \mapsto\left(S_{\mathrm{CW}}, T_{\mathrm{CW}}\right)$. Then the projective representation of $\operatorname{SL}(2,\mathbb Z)$ given by $\rho$ has kernel that is a congruence subgroup of level $\operatorname{ord}\left(T_{\mathrm{CW}}\right)$.

\end{conj}

\section{Small quantum groups $U_qsl(2)$}\label{sec:uqsl2}

The family of factorizable ribbon Hopf algebras $U_qsl(2)$, called small quantum groups, has been well-studied.  By examining explicitly their modules and resulting modular categories, we identify their Higman ideals with Weil representations.

\subsection{Small quantum groups $U_qsl(2)$}
Let $l=2h+1\geq 3$ be an odd natural number and $q$ be
a primitive $l^{th}$ root of unity. The small quantum group
$U_qsl(2)$ is a Hopf algebra generated by $E$, $F$, and $K$ with the
relations
$$E^l=F^l=0,\ \ K^l=1$$ and the Hopf algebra structure given by
$$KE=q^2EK,\ \ KF=q^{-2}FK,\ \ [E, F]=\frac{K-K^{-1}}{q-q^{-1}},$$
$$\Delta(E)=1\otimes E+E\otimes K,\ \ \Delta(F)=K^{-1}\otimes
F+F\otimes 1, \ \ \Delta(K)=K\otimes K,$$
$$\epsilon(E)=\epsilon(F)=0,\ \ \epsilon(K)=1,$$ $$S(E)=-EK^{-1},\ \ S(F)=-KF, \ \ S(K)=K^{-1}.$$
The PBW-basis of $U_qsl(2)$ are given by $\{F^mE^nK^k\}$ for $0\leq
m,n,k\leq l-1$. So the dimension of $U_qsl(2)$ is $l^3$. 

By induction,
$$\Delta(F^mE^nK^k)=\sum_{r=0}^m\sum_{s=0}^nq^{2(n-s)(r-m)+r(m-r)+s(n-s)}{m\brack r}{n\brack s}F^rE^{n-s}K^{r-m+k}\otimes
F^{m-r}E^sK^{n-s+k}.$$ 
The integrals $\lambda$ and $\Lambda$ are given by
$$\lambda(F^mE^nK^k)=\frac{1}{\zeta}\delta_{m,l-1}\delta_{n,l-1}\delta_{k,1},~~\text{and}~~\Lambda=\zeta F^{l-1}E^{l-1}\sum_{j=0}^{l-1}K^j,$$
where
$\zeta=\frac{\sqrt{l}}{([l-1]!)^2}$ is for normalization.

The small quantum group $U_qsl(2)$ is a ribbon Hopf algebra with a universal $R$-matrix \cite[Theorem IX.7.1.]{kassel2012quantum}
$$R=\frac{1}{l}\sum_{0\leq m,i,j\leq l-1}\frac{(q-q^{-1})^m}{[m]!}q^{\frac{m(m-1)}{2}+2m(i-j)-2ij}E^mK^i\otimes F^mK^j.$$
and a ribbon element \cite[Proposition XIV.6.5.]{kassel2012quantum}
$$\theta=\frac{1}{l}(\sum_{r=0}^{l-1}q^{hr^2})(\sum_{0\leq m,j\leq l-1}\frac{(q-q^{-1})^m}{[m]!}(-1)^mq^{-\frac{1}{2}m+mj+\frac{1}{2}(j+1)^2}F^mE^mK^j)$$

The Drinfeld $Q$-matrix $Q=R_{21}R$ is
$$Q=\frac{1}{l}\sum_{0\leq m,n,i,j\leq l-1}\frac{(q-q^{-1})^{m+n}}{[m]![n]!}q^{\frac{m(m-1)}{2}+\frac{n(n-1)}{2}-m^2-mj+mi+2nj-2ni-ij}F^mE^nK^j\otimes E^mF^nK^i.$$
When $Q$ is represented as $Q=\sum_Im_I\otimes n_I$, $\{m_I\}$ and
$\{n_I\}$ are two bases of $U_qsl(2)$. It follows that the small quantum group $U_qsl(2)$ is a factorizable ribbon Hopf
algebra.

\subsubsection{Radford isomorphism}
The Radford isomorphism $\phi_{R}: H^*\rightarrow H$ is given by 
\begin{align*}
\phi_R(\alpha)=\Lambda\leftharpoonup \alpha=(\alpha\otimes id)\Delta(\Lambda)=\sum_{(\Lambda)}\alpha(\Lambda_{(1)})\Lambda_{(2)},
\end{align*}
which is an isomorphism between left $H$-modules $H^*$ and $H$. The left action of $H$ on itself is given by left multiplication and the left
action of $H$ on $H^*$ is given by $h(\beta)=\beta\leftharpoonup S(h)$.

Note that the Radford isomorphism and the Frobenius map are inverse of each other.

\subsubsection{The center of $U_qsl(2)$}
Let $ {Z}$ denote the center of $U_qsl(2)$. In \cite{Kerler1}, Kerler constructed a canonical base for ${Z}$ by study the Casimir element of $U_qsl(2)$, which is
$$C_H=EF+\frac{q^{-1}K+qK^{-1}}{(q-q^{-1})^2}=FE+\frac{qK+q^{-1}K^{-1}}{(q-q^{-1})^2},$$
and satisfies the minimal polynomial $\Psi(C_H)=0$ where
$$\Psi(x)=\prod_{j=1}^l(x-b_j),\ \ \ \ b_j=\frac{q^{j}+q^{-j}}{(q-q^{-1})^2}$$
Using the polynomials $\phi_j(x)=\prod_{b_i\neq
	b_j}(x-b_i), \ j=0,\ldots, h$, any polynomial $R(C_H)$ in $C_H$ can be expressed in the forms of
$$R(C_H)=\sum_{j=0}^hR(b_j)P_j+\sum_{j=1}^{h}R'(b_j)N_j,$$
where
$$P_j=\frac{1}{\phi_j(b_j)}\phi_j(C_H)-\frac{\phi_j'(b_j)}{\phi_j(b_j)^2}(C_H-b_j)\phi_j(C_H), \ \ j=0,\ldots, h$$
$$N_j=\frac{1}{\phi_j(b_j)}(C_H-b_j)\phi_j(C_H), \ \ j=1,\ldots, h$$
This can be understood as some kind of Lagrange interpolation.

In order to describe $ {Z}$, we introduce the
projections
$$\pi^{+}_j=\frac{1}{l}\sum_{n=0}^{j-1}\sum_{i=0}^{l-1}q^{(2n-j+1)i}K^i,\ \ \pi^{-}_j=\frac{1}{l}\sum_{n=j}^{l-1}\sum_{i=0}^{l-1}q^{(2n-j+1)i}K^i, \ \ j=1,\ldots, h$$
Let $N_j^+=\pi^{+}_jN_j$ and $N_j^-=\pi^{-}_jN_j$ for $j=1,\ldots,h$.

\begin{prop}
	\cite{Kerler1}~The center ${Z}$ of $U_qsl(2)$ is a $3h+1$ dimensional
	commutative algebra with basis $\{P_i, N^{+}_j,
	N^{-}_j~|~i=0,\ldots,h;j=1,\ldots,h\}$ such that
	$$P_iP_j=\delta_{ij}P_j,\ \ P_iN^{\pm}_j=\delta_{ij}N^{\pm}_j,\ \ N^{\pm}_iN^{\pm}_j=N^{\pm}_iN^{\mp}_j=0$$
\end{prop}

\begin{prop}
	\cite{Kerler1}~The ribbon element is decomposed in terms of the canonical central
	elements as
	$$\theta=\sum_{r=0}^{h}(-1)^{r+1}q^{-\frac{1}{2}(r^2-1)}P_r-(q-q^{-1})\sum_{r=1}^{h}(-1)^{r+1}q^{-\frac{1}{2}(r^2-1)}(\frac{l-r}{[l-r]}N^{+}_r+\frac{r}{[r]}N^{-}_r)$$
\end{prop}

\subsection{Representations of $U_qsl(2)$}
Now we review the representation theory of $U_qsl(2)$. First, there are $l$ irreducible $U_qsl(2)$-modules $V_r$'s for $r=1,\ldots,l$. $V_r$ is spanned by $v^{(r)}_n$, $0\leq n\leq r-1$
with $v^{(r)}_0$ the highest weight vector. The $U_qsl(2)$-action is
given by
$$Kv^{(r)}_n=q^{r-1-2n}v^{(r)}_n$$
$$Ev^{(r)}_n=[n][r-n]v^{(r)}_{n-1}$$
$$Fv^{(r)}_n=v^{(r)}_{n+1}$$
where we set $v^{(r)}_0=v^{(r)}_r=0$. In particular, $V_1$ is the
trivial module and $V_l$ is also a projective module.

There are $l-1$ indecomposable projective $U_qsl(2)$-modules $P_r$'s for $r=1,\ldots,l-1$. $P_r$ is generated by $\{x^{(r)}_k,y^{(r)}_k\}_{k=0}^{l-r-1}\cup\{a^{(r)}_n,b^{(r)}_n\}_{n=0}^{r-1}$. The action of
$U_qsl(2)$ on $P_r$ is given by
$$Kx^{(r)}_k=q^{l-r-1-2k}x^{(r)}_k,\ \ Ky^{(r)}_k=q^{l-r-1-2k}y^{(r)}_k,\ \ 0\leq k \leq l-r-1,$$
$$Ka^{(r)}_n=q^{r-1-2n}a^{(r)}_n,\ \ Kb^{(r)}_n=q^{r-1-2n}b^{(r)}_n,\ \ 0\leq n \leq r-1,$$
$$Ex^{(r)}_k=[k][l-r-k]x^{(r)}_{k-1}, \ \ 0\leq k \leq l-r-1, \ \ (with \ \ x^{(r)}_{-1}=0)$$
$$Ey^{(r)}_k=[k][l-r-k]y^{(r)}_{k-1}, \ \ 0\leq k \leq l-r-1, \ \ Ey^{(r)}_0=a^{(r)}_{r-1},$$
$$Ea^{(r)}_n=[n][r-n]a^{(r)}_{n-1}, \ \ 0\leq n \leq r-1, \ \ (with \ \ a^{(r)}_{-1}=0)$$
$$Eb^{(r)}_n=[n][r-n]b^{(r)}_{n-1}+a^{(r)}_{n-1}, \ \ 0\leq n \leq r-1, \ \ Eb^{(r)}_0=x^{(r)}_{p-r-1},$$
$$Fx^{(r)}_k=x^{(r)}_{k+1}, \ \ 0\leq k \leq l-r-2, \ \ Fx^{(r)}_{l-r-1}=a^{(r)}_0$$
$$Fy^{(r)}_k=y^{(r)}_{k+1}, \ \ 0\leq k \leq l-r-1, \ \ (with \ \ y^{(r)}_{l-r}=0)$$
$$Fa^{(r)}_n=a^{(r)}_{n+1}, \ \ 0\leq n \leq r-1, \ \ (with \ \ a^{(r)}_r=0)$$
$$Fb^{(r)}_n=b^{(r)}_{n+1}, \ \ 0\leq n \leq r-2, \ \ Fb^{(r)}_{r-1}=y^{(r)}_0$$

Note that $\{x^{(r)}_k,y^{(r)}_k\}_{k=0}^{l-r-1}\cup\{a^{(r)}_n\}_{n=0}^{r-1}$ generates a submodule $W_r$
of $P_r$, and $\{x^{(r)}_k\}_{k=0}^{l-r-1}\cup\{a^{(r)}_n\}_{n=0}^{r-1}$ generates a submodule $M_r$
of $W_r$. We have a composition series:
\begin{eqnarray}\label{composition series}
P_r\supseteq W_r\supseteq M_r\supseteq V_r\supseteq \{0\}
\end{eqnarray}
with composition factors $V_r$, $V_{l-r}$, $V_{l-r}$, $V_r$.

For a $U_qsl(2)$-module $V$, its q-character is defined to be $\text{qCh}_{V}=Tr_{V}(K^{-1}?)$.
By $S^2(x)=KxK^{-1}$ for $x\in U_qsl(2)$, we know that $\text{qCh}_{V}\in \text{qCh}(U_qsl(2))$. For the irreducible $U_qsl(2)$-modules $V_1,\ldots, V_l$, the images of their $q$-characters in the center are
$$\phi(r):=\phi(\text{qCh}_{V_r})=\sum_{(\Lambda)}\text{Tr}_{V_r}(K^{-1}\Lambda_{(1)})\Lambda_{(2)},\ \ 1\leq r\leq l$$
\begin{prop}
	\begin{eqnarray*}
		\phi(r)&=&\frac{l\sqrt{l}}{[r]^2}N^{+}_r, \ \ 1\leq r \leq h\\
		\phi(r)&=&\frac{l\sqrt{l}}{[r]^2}N^{-}_{p-r}, \ \ h+1\leq r \leq 2h\\
		\phi(l)&=&l\sqrt{l}P_0
	\end{eqnarray*}
\end{prop}
\begin{proof}
	Plugging the formula of cointegral $\Lambda$, we have
	$$\phi(r)=\zeta\sum_{n=0}^{r-1}\sum_{0\leq i,j\leq l}([i]!)^2q^{j(r-1-2n)}{r-n+i-1\brack i}{n\brack i}F^{l-1-i}E^{l-1-i}K^j$$
	By the same calculation for the lemma 4.5.1 in \cite{FGST}, one can see that $\phi(r)$ acts by zero on $P_{r'}$ for $r'\neq r$ and it is proportional to $N^{\pm}$. The proportionality coefficients are determined by the action of $\phi(r)$ on $b_0^{(r)}$ and $N_r^{+}b^{(r)}_0 = a^{(r)}_0$, $N_r^{-}b^{(l-r)}_0 = a^{(l-r)}_0$.
\end{proof}

By the composition series \eqref{composition series}, we know that $\text{qCh}_{P_r}=2\text{qCh}_{V_r}+2\text{qCh}_{V_{l-r}}$ for $r=1, \ldots, l-1$.  So $\{\phi(0)\}\cup\{\phi(r)+\phi(l-r)\}_{r=1}^h$ gives a basis for $\text{Hig}(U_qsl(2))$.

Recall that when restricted on $\text{qCh}(H)$, the Drinfeld map $\chi$ is an isomorphism of commutative algebras between $\text{qCh}(H)$ and the center $Z(H)$.

For the irreducible $U_qsl(2)$-modules $V_1,\ldots, V_l$, the images of their $q$-characters under the Drinfeld map are
$$\chi(r):=\chi(\text{qCh}_{V_r})=(\text{Tr}_{V_r}\otimes id)((K^{-1}\otimes1)Q),\ \ 1\leq r\leq l.$$

\begin{prop}
	$$\chi(r)=\sum_{n=0}^{r-1}\sum_{m=0}^n(q-q^{-1})^{2m}q^{-(m+1)(m+r-1-2n)}{r-n+m-1\brack m}{n\brack m}E^mF^mK^{r-1-2n+m}.$$
	In particular, $\chi(2):=\hat{C}=(q-q^{-1})^2C$.
\end{prop}

\begin{proof}
	Note that by induction
	$$F^rE^r=\prod_{s=0}^{r-1}(C-\frac{q^{2r+1}K+q^{-2r-1}K^{-1}}{(q-q^{-1})^2}),$$
	for $r<l$, Then
	$$Tr_{V_r}(F^mE^mK^l)=([m]!)^2\sum_{n=0}^{r-1}q^{l(r-1-2n)}{r-n+m-1\brack m}{n\brack m}.$$
\end{proof}
By observing the fusion rules, we find $\chi(r)$ can be calculated by
the Chebyshev polynomials of the second type $$U_r(2\cos
t)=\frac{\sin rt}{\sin t}$$ which satisfy the recursive relation
$xU_r(x)=U_{r-1}(x)+U_{r+1}(x)$ for $r\geq 2$. Then $\chi(r)=U_r(\hat{C})$ for $r=1,\ldots,l$.
The following proposition provides a formula expanding $\chi(r)$ in terms of the canonical basis of the center.
\begin{prop}
	For $r=1,\ldots,l-1,$
	\begin{eqnarray*}
		\chi(r)&=&\sum_{j=0}^h\frac{[jr]}{[j]}P_j(C)-\sum_{j=1}^{h}\frac{(r+1)[j(r-1)]-(r-1)[j(r+1)]}{[j]^3}N_j(C)\\
		\chi(l)&=&lP_0(C)+2l\sum_{j=1}^{h}\frac{1}{[j]^2}N_j(C)
	\end{eqnarray*}
\end{prop}

\begin{proof}
	Recall $\hat{C}$
	satisfies
	$$\prod_{j=1}^l(\hat{C}-\hat{b}_j)=0,\ \ \hat{b}_j=q^{j}+q^{-j}$$
	Similarly as the case of $C$, any polynomial $R(\hat{C})$ in
	$\hat{C}$ can be expressed in terms of $\hat{P}_j$ and $\hat{N}_j$
	by
	$$R(\hat{C})=\sum_{j=0}^hR(\hat{b}_j)\hat{P}_j+\sum_{j=1}^{h}\hat{R}'(\hat{b}_j)\hat{N}_j(\hat{C}).$$
	Here $\hat{P}_j(\hat{C})=P_j(C),\ \
	\hat{N}_j(\hat{C})=(q-q^{-1})^2N_j(C)$.
	Then
	\begin{eqnarray*}		\chi(r)&=&U_r(\hat{C})=\sum_{j=0}^hU_r(\hat{b}_j)\hat{P}_j(\hat{C})+\sum_{j=1}^{h}U'_r(\hat{b}_j)\hat{N}_j(\hat{C})\\
		&=&\sum_{j=0}^hU_r(\hat{b}_j)P_j(C)+(q-q^{-1})^2\sum_{j=1}^{h}U'_r(\hat{b}_j)N_j(C)
	\end{eqnarray*}
	\noindent
	By plug in
	$$U_r(\hat{b}_j)=U_r(2\cos\frac{2\pi j}{l})=\sin\frac{2\pi jr}{l}/\sin\frac{2\pi j}{l}=\frac{[jr]}{[j]}$$
	$$U'_r(\hat{b}_j)=-\frac{1}{(q-q^{-1})^2}\frac{(r+1)[j(r-1)]-(r-1)[j(r+1)]}{[j]^3}$$
	We obtain the formulas for $\chi(r)$. Here the derivative is calculated by differentiating both sides of
	$U_r(2\cos t)\sin t=\sin rt$ with respect to $t$, then evaluating at
	$2\cos t=\hat{b}_j$.
\end{proof}

Define $\nu(r):=\chi(r)+\chi(l-r)$ for $r=1, \ldots, h$ and $\nu(0)=\chi(l)$. They span the image of $I(H)$ under the Drinfeld map.
\begin{cor}
	For $r=0, \ldots, h,$
	\begin{eqnarray*}
		\nu(r)&=&lP_0(C)+l\sum_{j=1}^{h}\frac{q^{jr}+q^{-jr}}{[j]^2}N_j(C)
	\end{eqnarray*}
\end{cor}

Let  $\mathcal{D}_l$ and
$\mathcal{R}_l$ be the images of $\text{qCh}_{V_r}$'s under the Drinfeld map and the Radford map, respectively. It is shown in \cite{La} that $\mathcal{D}_l+ \mathcal{R}_l=Z(U_qsl(2))$.
And $\mathcal{P}_{h+1}:=\mathcal{D}_l\cap \mathcal{R}_l$ is the image of the q-characters of projective indecomposable modules under both the Drinfeld map and the Radford map. 

\begin{prop}
	$\mathcal{P}_{h+1}=\text{Hig}(U_qsl(2))$.
\end{prop}
\begin{proof}
	\begin{eqnarray*}
		& &\nu(0)+\sum_{r'=1}^{h}(q^{rr'}+q^{-rr'})\nu(r')=\nu(0)+\frac{1}{2}\sum_{r'=1}^{l-1}(q^{rr'}+q^{-rr'})\nu(r')\\
		&=&\nu(0)+\frac{l}{2}\sum_{r'=1}^{l-1}(q^{rr'}+q^{-rr'})P_0(C)+\frac{l}{2}\sum_{r'=1}^{l-1}\sum_{j=1}^{h}(q^{rr'}+q^{-rr'})\frac{q^{jr'}+q^{-jr'}}{[j]^2}N_j(C)\\
		&=&\nu(0)-lP_0(C)+\frac{l}{2}\sum_{j=1}^{h}\sum_{r'=1}^{l-1}\frac{q^{(r+j)r'}+q^{(r-j)r'}+q^{(-r+j)r'}+q^{(-r-j)r'}}{[j]^2}N_j(C)\\
		&=&2l\sum_{j=1}^{h}\frac{1}{[j]^2}N_j(C)+\frac{l}{2}(\frac{2l}{[r}N_r(C)-4\sum_{j=1}^{h}\frac{1}{[j]^2}N_j(C))\\
		&=&\frac{l^2}{[r]^2}N_r(C)=\frac{1}{\sqrt{p}}(\phi(r)+\phi(l-r))
	\end{eqnarray*}
	
	\begin{eqnarray*}
		& &\nu(0)+2\sum_{r=1}^{h}\nu(r)=\nu(0)+\sum_{r=1}^{l-1}\nu(r)\\
		&=&lP_0(C)+2l\sum_{j=1}^{h}\frac{1}{[j]^2}N_j(C)+(l-1)lP_0(C)+l\sum_{r=1}^{p-1}\sum_{j=1}^{h}\frac{q^{jr}+q^{-jr}}{[j]^2}N_j(C)\\
		&=&l^2P_0(C)=\frac{1}{l\sqrt{l}}\phi(l)
	\end{eqnarray*}
    Note that $\{\phi(0)\}\cup\{\phi(r)+\phi(l-r)\}_{r=1}^h$ is a basis for $\text{Hig}(U_qsl(2))$. Thus $\{\nu_r\}_{r=0}^h$ is also a basis for $\text{Hig}(U_qsl(2))$ and $\mathcal{P}_{h+1}=\text{Hig}(U_qsl(2))$.
\end{proof}

\subsection{$\operatorname{SL}(2,\mathbb{Z})$-representations on the center of $U_qsl(2)$}

As \cite{La}, when restricted to the center $Z(H)$ of
$H$, we may slightly modify the definition of this representation: for
$a\in Z(H)$
$$S_{LM}(a)=(\lambda(S(a))\otimes 1)R_{21}R=\phi(\chi^{-1}(a)),$$
$$T_{LM}(a)=\kappa S_{LM}^{-1}(\theta^{-1}(S_{LM}(a))).$$
The following theorem was conjectured by Kerler in \cite{Kerler1}.

\begin{thm}[Kerler]\label{Kerler} Let $l=2h+1$ be an odd number and $q$ be a $l^{th}$ primitive root of unity. The $\operatorname{SL}(2,\mathbb{Z})$ representation on the center
	$Z$ of $U_qsl(2)$ decomposes as
	$$Z=\mathcal{P}_{h+1}\oplus\mathbb{C}^2\otimes \mathcal{V}_{h}$$
	$\mathcal{P}_{h+1}$ is an $(h+1)$ dimensional representation and $\mathbb{C}^2$
	is the standard representation of $\operatorname{SL}(2,\mathbb{Z})$. $\mathcal{V}_{h}$ is an
	$h$ dimensional representation when restricted on which the matrices
	$S_{LM}$ and $T_{LM}$ are the same as those obtained by
	RT TQFT.
\end{thm}

\begin{proof}
	We choose a basis for $Z$ as
	\begin{eqnarray*}
		\nu(r)&=&\chi(r)+\chi(l-r), \ \ r=1,\ldots,h;~~~~\nu(0)=\chi(l)\\
		\rho(r)&=&\frac{l-r}{l}\chi(r)-\frac{r}{l}\chi(l-r), \ \ r=1,\ldots,h\\
		\varphi(r)&=&\frac{1}{\sqrt{l}}\sum_{j=1}^{h}(q^{jr}-q^{-jr})(\frac{l-r}{l}\phi(r)-\frac{r}{l}\phi(l-r)), \ \ r=1,\ldots,h
	\end{eqnarray*}
	\noindent
	First,
	\begin{eqnarray*}
		{S}_{LM}(\nu(r))&=&\phi\chi^{-1}(\chi(r)+\chi(l-r))=\phi(r)+\phi(l-r)\\
		&=&\frac{1}{\sqrt{l}}(\nu(0)+\sum_{j=1}^{h}(q^{jr}+q^{-jr})\nu(j))\\
		{S}_{LM}(\nu(0))&=&\frac{1}{\sqrt{l}}(\nu(0)+2\sum_{r=1}^{h}\nu(r))
	\end{eqnarray*}
	So $\mathcal{P}_{h+1}$ is invariant under the action of $S_{LM}$.
	\noindent
	Further,
	\begin{eqnarray*}
		S_{LM}(\rho(r))&=&\phi\chi^{-1}(\frac{l-r}{l}\chi(r)-\frac{r}{l}\chi(l-r))=\frac{l-r}{l}\phi(r)-\frac{r}{l}\phi(l-r)\\
		&=&-\frac{1}{\sqrt{l}}\sum_{j=1}^{h}(q^{jr}-q^{-jr})\varphi(r)
	\end{eqnarray*}
	Note the facts that $S_{LM}^2=s_H^{-1}$ and the antipode $s_H$ acts
	identically on the center $Z$. Then
	$S_{LM}(\frac{l-r}{l}\phi(r)-\frac{r}{l}\phi(l-r))=\rho(r)$.
	And
	\begin{eqnarray*}
		S_{LM}(\varphi(r))&=&\frac{1}{\sqrt{l}}\sum_{j=1}^{h}(q^{jr}-q^{-jr})S_{LM}(\frac{l-r}{l}\phi(r)-\frac{r}{l}\phi(l-r))\\
		&=&\frac{1}{\sqrt{l}}\sum_{j=1}^{h}(q^{jr}-q^{-jr})\rho(r)
	\end{eqnarray*}
	
	$\{\rho(r),\varphi(r)\}_{r=1,\ldots,h}$ span the
	$\mathbb{C}^2\otimes \mathcal{V}_h$ that is invariant under the action of
	$S_{LM}$, and the matrix of $S_{LM}$ is
	\[
	\begin{pmatrix}
	0_{h\times h} & -\frac{q-q^{-1}}{\sqrt{l}}S_{semi}\\
	\frac{q-q^{-1}}{\sqrt{l}}S_{semi} & 0_{h\times h}\\
	\end{pmatrix}
	=\frac{q-q^{-1}}{\sqrt{l}}
	\begin{pmatrix}
	0 & -1\\
	1 & 0\\
	\end{pmatrix}
	\otimes S_{semi}
	\]
	where $S_{semi}$ is the semisimple $S$-matrix $([q^{jr}])_{h\times
		h}$.
	
	Recall that
	\begin{align*}
	\theta=&\sum_{r=0}^{h}(-1)^{r+1}q^{-\frac{1}{2}(r^2-1)}P_r-(q-q^{-1})\sum_{r=1}^{h}(-1)^{r+1}q^{-\frac{1}{2}(r^2-1)}(\frac{l-r}{[l-r]}N^{+}_r+\frac{r}{[r]}N^{-}_r)\\
	=&\sum_{r=0}^{h}(-1)^{r+1}q^{-\frac{1}{2}(r^2-1)}P_r+\frac{q-q^{-1}}{\sqrt{l}}\sum_{r=1}^{h}(-1)^{r+1}q^{-\frac{1}{2}(r^2-1)}[r](\frac{l-r}{l}\phi(r)-\frac{r}{l}\phi(l-r))
	\end{align*}
	It is easy to check
	$$\theta^{-1}=\sum_{j=0}^{h}(-1)^{j+1}q^{\frac{1}{2}(j^2-1)}P_j-\frac{q-q^{-1}}{\sqrt{l}}\sum_{j=1}^{h}(-1)^{j+1}q^{\frac{1}{2}(j^2-1)}[j](\frac{l-j}{l}\phi(j)-\frac{j}{l}\phi(l-j))$$	
	Then $\theta^{-1}$ acts on $\phi(r)$ as
	\begin{align*}
	\theta^{-1}\phi(r)=&(-1)^{r+1}q^{\frac{1}{2}(r^2-1)}\phi(r), \ \ for \ r=1, \ldots, h.\\
	\theta^{-1}\phi(r)=&(-1)^{l-r+1}q^{\frac{1}{2}((l-r)^2-1)}\phi(r), \ \ for \ r=h+1,\ldots, 2h.
	\end{align*}
	Then we have
	\begin{align*}
	T_{LM}(\chi(r))=&\kappa S_{LM}^{-1}(\theta^{-1}\phi(r))=(-1)^{r+1}q^{\frac{1}{2}(r^2-1)}\kappa\chi(r), \ \ for \ r=1, \ldots, h.\\
	T_{LM}(\chi(r))=&\kappa  S_{LM}^{-1}(\theta^{-1}\phi(r))=(-1)^{l-r+1}q^{\frac{1}{2}((l-r)^2-1)}\kappa\chi(r), \ \ for \ r=h+1, \ldots, 2h.
	\end{align*}
	Then $T_{LM}(\nu(r))=(-1)^{r+1}q^{\frac{1}{2}(r^2-1)}\kappa\nu(r)$ for $r=0, 1, \ldots, h$. So $\mathcal{P}_{h+1}$ is invariant under $T_{LM}$ and actually a $(h+1)$ dimensional representation of $\operatorname{SL}(2,\mathbb{Z})$. Moreover, $T_{LM}(\rho(r))=(-1)^{r+1}q^{\frac{1}{2}(r^2-1)}\kappa\rho(r)$.
	
	Finally, we want to evaluate $T_{LM}(\varphi(r))$. 
	\begin{align*}
	T_{LM}(\varphi(r))=&\kappa S_{LM}^{-1}(\theta^{-1}\frac{1}{\sqrt{l}}\sum_{j=1}^{h}(q^{jr}-q^{-jr})(\frac{l-j}{l}\chi(j)-\frac{j}{l}\chi(l-j)))\\
	=&\frac{\sqrt{l}}{q^r-q^{-r}}\kappa S_{LM}^{-1}(\theta^{-1}(P_r-\frac{q^{r}+q^{-r}}{[r]^2}N_r))\\
	=&\frac{(-1)^{r+1}\sqrt{l}q^{\frac{1}{2}(r^2-1)}}{q^r-q^{-r}}\kappa S_{LM}^{-1}(P_r-\frac{q^{r}+q^{-r}}{[r]^2}N_r+\frac{q-q^{-1}}{\sqrt{l}}[r](\frac{l-r}{l}\phi(r)-\frac{r}{l}\phi(l-r)))\\
	=&(-1)^{r+1}q^{\frac{1}{2}(r^2-1)}\kappa\varphi(r)+(-1)^{r+1}q^{\frac{1}{2}(r^2-1)}\kappa\rho(r)
	\end{align*}
	Here we used that
	$S_{LM}(\varphi(r))=\frac{1}{\sqrt{l}}\sum_{j=1}^{h}(q^{jr}-q^{-jr})\rho(j)=\frac{\sqrt{l}}{q^r-q^{-r}}(P_r-\frac{q^{r}+q^{-r}}{[r]^2}N_r))$.
	
	$\mathbb{C}^2\otimes \mathcal{V}_h$ is also invariant under the action of $T_{LM}$. The
	matrix the matrix of $T_{LM}$ is
	\[
	\begin{pmatrix}
	\lambda T_{semi} & \lambda T_{semi}\\
	0_{h\times h} & \lambda T_{semi}\\
	\end{pmatrix}
	=\kappa
	\begin{pmatrix}
	1 & 1\\
	0 & 1\\
	\end{pmatrix}
	\otimes T_{semi}
	\]
	Here $T_{semi}=\text{diag}(\ldots, (-1)^{r+1}q^{\frac{1}{2}(r^2-1)}, \ldots)_{h\times
		h}$ is the semisimple T-matrix.
\end{proof}

We collect the $S$ and $T$ for the Higman and semisimple parts in as follows.

\noindent Higman part:

\begin{align}\label{Uqsl2SHig}
	S_N=&\frac{1}{\sqrt{l}}
	\begin{pmatrix}
	1 & 1 &   & \cdots &   & 1\\
	2 & q+q^{-1} &   & \cdots &   & q^{h}+q^{-h}\\
	&   & \ddots &   &   &  \\
	\vdots & \vdots &   & q^{jr}+q^{-jr} &   & \vdots\\
	&   &   &   & \ddots &  \\
	2 & q^{h}+q^{-h} &   & \cdots &   & q^{h^2}+q^{-h^2}\\
	\end{pmatrix}_{(h+1)\times(h+1)}\end{align} \begin{align}\label{Uqsl2THig}
	T_N=&\kappa
	\begin{pmatrix}
	-q^{-\frac{1}{2}} & 0 &   & \cdots &   & 0\\
	0 & 1 &   & \cdots &   & 0\\
	&   & \ddots &   &   &  \\
	\vdots & \vdots &   & (-1)^{r+1}q^{\frac{1}{2}(r^2-1)} &   & \vdots\\
	&   &   &   & \ddots &  \\
	0 & 0 &   & \cdots &   & (-1)^{h+1}q^{\frac{1}{2}(h^2-1)}\\
	\end{pmatrix}_{(h+1)\times(h+1)}
\end{align}

\noindent Semisimple part:

\begin{align}\label{Uqsl2SSO3}
S_V=&\frac{1}{\sqrt{l}}
\begin{pmatrix}
q-q^{-1} & q^2-q^{-2} &   & \cdots &   & q^h-q^{-h}\\
q^{2}-q^{-2}& q^{4}-q^{-4} &   & \cdots &   & q^{2h}-q^{-2h}\\
&   & \ddots &   &   &  \\
\vdots & \vdots &   & q^{jr}-q^{-jr} &   & \vdots\\
&   &   &   & \ddots &  \\
q^{h}-q^{-h} & q^{2h}-q^{-2h} &   & \cdots &   & q^{h^2}-q^{-h^2}\\
\end{pmatrix}_{h \times h}\end{align} 
\begin{align}\label{Uqsl2TSO3}
T_V=&\kappa
\begin{pmatrix}
1 & 0 &   & \cdots &   & 0\\
0 & -q^{\frac{3}{2}} &   & \cdots &   & 0\\
&   & \ddots &   &   &  \\
\vdots & \vdots &   & (-1)^{r+1}q^{\frac{1}{2}(r^2-1)} &   & \vdots\\
&   &   &   & \ddots &  \\
0 & 0 &   & \cdots &   & (-1)^{h+1}q^{\frac{1}{2}(h^2-1)}\\
\end{pmatrix}_{h \times h}
\end{align}

\section{Drinfeld doubles of Nichols Hopf algebras}\label{sec:doubleofnichols}

In this section, we study the Drinfeld doubles $D\mcK_n, n\geq 1$ of Nichols Hopf algebras. These doubled Nichols Hopf algebras exhibit phenomena dramatically different from the small quantum groups.  Moreover, the Nichols Hopf algebras are closely related to the so-called Grassmann numbers in physics, so it is not surprising that $D\mcK_n, n\geq 1$ is used to study fermions. Through our analysis, we recover the decomposition of the center of $D\mcK_n$ presented in \cite{farsad2022symplectic}. Many results for $D\mcK_n, n\geq 1$ with even $n$ should follow from \cite{farsad2022symplectic}, but our approach is completely explicit and self-contained.  Moreover, we additionally investigate $\mcK_n$ and $D\mcK_n$ with $n$ odd.

\subsection{Nichols Hopf algebras and their doubles}
    Let $n$ be a positive integer. The Nichols Hopf algebra $\Kn$ is the complex $2^{n+1}$-dimensional Hopf algebra with algebra generators $K, \xi_1, \dots, \xi_n$ subject to the following relations: for $i, j=1,\dots, n$
    \begin{equation}\label{Kn-pres}
        K^2 = 1,\hspace{5em} \xi_i^2 = 0\hspace{5em}K\xi_i =-\xi_i K,\hspace{5em}\xi_i\xi_j=-\xi_j\xi_i.
    \end{equation}
    The Nichols Hopf algebra is given a Hopf algebra structure with the following comultiplication $\Delta:\Kn\to\Kn\otimes \Kn$, counit $\epsilon:\Kn\to\bC$ and antipode $S:\Kn\to \Kn$:
    \begin{align*}
        \Delta(K) &= K\otimes K, & \epsilon(K) &= 1_{\bC}, & S(K) &= K,\\
        \Delta(\xi_i) &= K\otimes \xi_i + \xi_i\otimes 1, & \epsilon(\xi_i) &= 0_{\bC}, & S(\xi_i) &= -K\xi_i.
    \end{align*}
    Note that $\mcK_1$ is the 4-dimensional Sweedler's Hopf algebra. Nichols Hopf algebras can also be expressed as a crossed product of the exterior algebra of a $n$-dimensional complex vector space $E$ with the $\bZ/2\bZ$-group algebra $\Kn\cong\Lambda^* E\rtimes \bC[\bZ/2\bZ]$.
    
    Let $W$ be the set of all elements in $\Kn$ of the form $\xi_{i_1}\dots\xi_{i_k}$ for $1\leq i_1 < i_2 <\dots < i_k\leq n$ and $k\geq 0$ (if $k=0$, the corresponding element is $1$). Namely, $W$ is the set of words in the $\xi_i$ with increasing index. Define $B = W\cup KW$. Note that $B$ is a basis for $\Kn$. For $b\in B$, let $\mathbf I(b) = \{\xi_{i_1},\dots, \xi_{i_k}\}$ and call $|b|:=|\mathbf I(b)|$ the length of $b$. Inversely, for $I\subseteq\{\xi_1,\dots, \xi_n\}$, define $\mathbf w(I)=\prod_{\xi_i\in I}\xi_i$, where the product occurs in ascending order.
    
    \begin{lemma}\label{SbLemma}
        Let $w\in W$, then  $S(w) = (-K)^{|w|}w$ and $S(Kw)=K^{|w|+1}w$. 
    \end{lemma}

    The Drinfeld double $DH$ of a Hopf algebra $H$ is a Hopf algebra formed from a bicrossed product $DH=H^{*cop}\bowtie H$, where $H^{*cop}$ means the opposite comultiplication dual Hopf algebra of $H$. As a vector space, $DH=H^{*}\otimes H$. The multiplication in $DH$ is given for $\phi,\phi'\in H^{*cop}$ and $h,h'\in H$ by
    $$(\phi\otimes h)(\phi'\otimes h')=\phi\phi'(S^{-1}(h_{(3)})(-)h_{(1)})\otimes h_{(2)}h'.$$
    Since $DH$ is generated by $\epsilon\otimes H$ and $H^{*cop}\otimes 1$, the remaining operations on the Hopf algebra $DH$ are inherited from $H$ and $H^{*cop}$.
    
    For $b\in B$, let $b^*:\Kn\to\bC$ be the linear map generated by $b^*(b')=\delta_{b=b'}$ for all $b'\in B$. Let $\bar K=1^*-K^*$ and $\bar\xi_i = \xi_i^*+(K\xi_i)^*$. We will always identify $f\in \Kn^{*cop}$ with $f\otimes 1\in D\Kn$ and $a\in \Kn$ with $\epsilon\otimes a\in D\Kn$. In particular, $1$ is identified with $1\otimes \epsilon\in D\Kn$. Note that $fa = f\otimes a$ for any $f\in \Kn^{*cop}$ and $a\in \Kn$. 
    
    \begin{lemma}\label{DKn-pres}
    In addition to the defining relations of the Nichols Hopf algebra (Eqn. \ref{Kn-pres}), we have the following additional relationships among elements in the doubled Nichols Hopf algebra $D\Kn$: for $i, j=1,\dots, n$
    \begin{align*}
        \bar K^2&=1,& \bar K\bar\xi_i&=-\bar\xi_i \bar K, &  \bar\xi_i\bar\xi_j&=-\bar\xi_j\bar\xi_i,\\
        (K\bar K)^2 &= 1& K\bar\xi_i&=-\bar\xi_i K, & \bar K\xi_i&=-\xi_i \bar K,\\
        \xi_i\bar\xi_i&= 1-K\bar K-\bar\xi_i\xi_i, & \xi_i\bar\xi_j&=-\bar\xi_j\xi_i, \text{ for $i\neq j$}.
    \end{align*}
    The coalgebra structure and antipode of $\Kn^{*cop}$ (and hence $D\Kn$) are described by:
    \begin{align*}    
        \Delta(\bar K)&=\bar K\otimes \bar K, & \epsilon(\bar K) &= 1_\bC, & S(\bar K)&=\bar K, \\
        \Delta(\bar \xi_i)&=\bar K\otimes \bar \xi_i+\bar\xi_i\otimes 1, & \epsilon(\bar \xi_i) &= 0_\bC, & S(\bar \xi_i)&= -\bar K\bar \xi_i.
    \end{align*}
    In particular, $D\Kn$ is generated by two copies of $\Kn$.
    \end{lemma}

    \begin{lemma}
        Let $w=\xi_{i_1}\dots\xi_{i_k}\in W$. Define $\bar w = \bar\xi_{i_1}\dots\bar\xi_{i_k}$. Then,
        \begin{align*}
            \bar w &= (-1)^{\lfloor k/2\rfloor} (w^* + (Kw)^*),\\
            \bar w\bar K &= (-1)^{\lfloor k/2\rfloor}(w^* - (Kw)^*).
        \end{align*}
    \end{lemma}
    Let $\bar W = \{\bar w| w\in W\}$ and $\bar B = \bar W\cup \bar K\bar W$. It follows from this lemma that $\bar B$ is a linearly independent set. Moreover, it has $2^{n+1}$ elements, so it must be a basis for $\Kn^*$. Therefore, $\bar K$ and the $\bar\xi_i$ generate all of $\Kn^{*cop}$. Thus, $K, \bar K, \xi_i,\bar\xi_i$ generate $D\Kn$. Moreover, the presentation of $D\Kn$ given in Lemma \ref{DKn-pres} is complete since it defines an algebra of dimension $2^{2n+2}=\dim D\Kn$.
    
    \begin{corollary}\label{DKnR}
        The doubled Nichols Hopf algebra $D\Kn$ is quasitriangular with the following $R$-matrix:
        $$R = \sum_{w\in W} (-1)^{\lfloor|w|/2\rfloor}( w\otimes \bar w\bar K^{|w|})Z,$$
        where $Z = \frac{1}{2}(1\otimes 1 + K\otimes 1 + 1\otimes \bar K - K\otimes\bar K)$.
    \end{corollary}
    \begin{lemma}\label{DKn-ribbon}
        Suppose $n$ is even. A ribbon element $\nu\in D\Kn$ for the $R$-matrix in Corollary \ref{DKnR} is
        $$\nu = (1 + K - \bar K + K\bar K)\sum_{w\in W} \frac{(-1)^{\lfloor(|w|+1)/2\rfloor}}{2} w\bar w.$$
    \end{lemma}

    The set $\{\xi_1,\dots, \xi_n,\bar\xi_1,\dots, \bar\xi_n\}$ does not generate an exterior algebra because, for example $\xi_1\bar\xi_1\neq -\bar\xi_1 \xi_1$. However, if $K$ and $\bar K$ were identified in a quotient, for example, then this would generate an exterior algebra of dimension $2^{2n}$ since $1 - K^2 = 0$. It follows that $D\NK_n / \langle K - \bar K\rangle\cong \NK_{2n}$ under the quotient map $\pi:K,\bar K\mapsto K,\xi_i\mapsto \xi_i,\bar\xi_i\mapsto \xi_{i+n}$. In particular, since quasitriangularity is preserved by quotients, $\NK_{2n}$ is quasitriangular. The $R$-matrix is given by
    \begin{align*}
      R_{\mcK_{2n}} &= \sum_{w\in W} (-1)^{\lfloor|w|/2\rfloor}( w\otimes \pi(\bar w)K^{|w|})(\id\otimes\pi)(Z).
    \end{align*}
    There is a large family of $R$-matrices which make the Nichols Hopf algebras quasitriangular, of which these are a particular case \cite{panaite1999quasitriangular}. Indeed, there are many $R$-matrices for $\Kn$ for odd $n$ as well. Moreover, for every $n$ and every choice of $R$-matrix in the large family, $(\Kn, R)$ is ribbon. However, regardless of choice of $R$-matrix, $\Kn$-mod cannot be made modular, since, by Theorem \ref{Kn-modules}, there is no $\Kn$-module which is its own projective cover, contradicting Theorem \ref{thm:S-rank}.

    \begin{lemma}
        For the above $R$-matrix, a ribbon element $\nu_{\mcK_{2n}}\in \mcK_{2n}$ is
        $$\nu_{\mcK_{2n}} = \sum_{w\in W} (-1)^{\lfloor(|w|+1)/2\rfloor} w\pi(\bar w).$$
    \end{lemma}

    \begin{lemma}
        Let $x\in\Kn$ be a left integral (i.e., $hx =\epsilon(h)x$ for all $h\in\Kn$). Then, there is a constant $c\in\bC$ so that
        $$x = c (1+K)\xi_1\dots \xi_n.$$
        If $x\in\Kn$ is a right integral, then there is a constant $c\in\bC$ so that
        $$x = c \xi_1\dots \xi_n(1+K).$$
    \end{lemma}
    \begin{corollary}
        If $n$ is odd, then $\Kn$ is not unimodular. If $n$ is even, then $\Kn$ is unimodular.
    \end{corollary}
    \begin{lemma}
        Let $R = R_{\mcK_{2n}}$ and $F:\mcK_{2n}^*\to \mcK_{2n}$ be the Drinfeld map $F:f\mapsto (\id\otimes f)(R_{21}R)$.
        Then, $F$ is NOT an isomorphism. Thus, $\Kn$ is not factorizable for any $n$.
    \end{lemma}
    \begin{proof}
        First, observe
        $$R_{21}R = \sum_{w,w'\in W}(-1)^{\lfloor|w|/2\rfloor+\lfloor|w'|/2\rfloor}(\pi(\bar w')K^{|w'|}\otimes w')(\id\otimes\pi)(Z)(w\otimes \pi(\bar w)K^{|w|})(\id\otimes\pi)(Z).$$
        When $|w| > 0$, $\epsilon(w) = 0$, so all but the $w=1$ summand disappears under $\id\otimes \epsilon$.
        $$F(\epsilon)=(\id\otimes\epsilon)(R_{21}R) = (\id\otimes \epsilon)(\pi(Z)^2) = (\id\otimes \epsilon)(1\otimes 1) = 1,$$
        but $F(\bar K) = 1$ by the same reasoning, so $F$ is not injective.
    \end{proof}
    \begin{lemma}
        $D\Kn$ is factorizable.
    \end{lemma}

\subsection{Representations of $\Kn$ and $D\Kn$}
The total rank of $\Kn$-mod is 4, with two simples and two projective indecomposables. The total rank of $D\Kn$-mod is 6, with four simples, two of which are projective, and two additional projective indecomposables. The subring of the full fusion ring $F(D\Kn$-mod) generated by the two non-projective simples and their projective covers is isomorphic to $F(\mcK_{2n}$-mod). Hence, the additional projective simples in $D\Kn$-mod come from the relation $\xi_i\bar\xi_i = 1 - K\bar K - \bar\xi_i \xi_i \neq -\bar\xi_i\xi_i$. This can be seen readily in the definition of these modules in the proof of Theorem \ref{DKn-modules}. These two additional projective simples in $D\Kn$-mod are essential; for even $n$, they make $D\Kn$-mod an NSS MC.

\begin{theorem}\label{Kn-modules}
There are two non-isomorphic irreducible $\Kn$-modules $V_\epsilon, V_{\bar K}$ of dimension 1, and each has a projective cover $P_\epsilon, P_{\bar K}$ (resp.) of dimension $2^{n}$. Thus, the total rank of $\Kn$-mod is 4 for all $n$. Their fusion rules are 
$V_{\bar K}\otimes P_{\epsilon}=P_{\bar K}$ and 
$$P_\epsilon^2=P_{\bar K}^2=P_\epsilon\otimes P_{\bar K}=P_{\bar K}\otimes P_{\epsilon}=2^{n-1}P_\epsilon\oplus 2^{n-1}P_{\bar K}.$$
\end{theorem}
\begin{proof}
First, define $V_\epsilon$ and $V_{\bar K}$ to be the trivial and sign modules respectively, whose underlying vector space is $\bC$ and multiplication rules are $h\cdot x := \epsilon(h)x$ and $h\cdot x := \bar K(h)x$ respectively. Let $P_\epsilon=\Lambda^*E\cdot (1+K) \subset \Kn$ and $P_{\bar K}=\Lambda^* E\cdot(1-K)\subset \Kn$ be the $\Kn$-modules generated by left-multiplication. Clearly, $\Kn = P_\epsilon\oplus P_{\bar K}$, so these are each projective. Consider the maps $p_\epsilon:P_\epsilon\to V_K$ and $p_{\bar K}:P_{\bar K}\to V_{\bar K}$ defined by $p_\epsilon = \epsilon|_{P_\epsilon}$ and $p_{\bar K}= \bar K|_{P_{\bar K}}$. These are surjective $\Kn$-module homomorphisms to their respective simple modules. The kernels of these maps are $\ker p_\epsilon=\spa\{w(1+K) | w\in W, |w|\geq 1\}$ and $\ker p_{\bar K}=\spa\{w(1-K) | w\in W, |w|\geq 1\}$. Suppose $M$ is a $\Kn$-module for which $M+\ker p_\epsilon = P_\epsilon$. Then, there is an element of the form $(1 - x)(1+K)\in M$, where $x\in \spa\{W\setminus\{1\}\}$ (namely, because there is an element $m\in M$ so that $m + x(1+K) = 1+K$). Note that $x^{n+1} = 0$ because there are no nonzero $k$-forms for $k > n$ in $\Lambda^* E$. Therefore, 
$$1+K = (1 - x^{n+1})(1+K) = (1 + x + x^2 + \dots + x^n)(1 - x)(1+K) \in M.$$
Hence, $y(1+K)\in M$ for any $y\in \Lambda^* E$ since $M$ is a $\Kn$-module. It follows that $M = P_\epsilon$, so $\ker p_\epsilon$ is a superfluous submodule of $P_\epsilon$. A similar argument applies to $P_{\bar K}$. Thus, $P_\epsilon$ is the projective cover of $V_\epsilon$, and $P_{\bar K}$ is the projective cover of $V_{\bar K}$. These each have dim=$2^n$, and they are distinct. Moreover, this list is complete since 
$$\dim\Kn = \dim(V_\epsilon)\dim(P_\epsilon)+\dim(V_{\bar K})\dim(P_{\bar K}).$$

Next, we discuss the tensor product of these modules. We will use the fact that the Hom functor distributes over direct sums and that 
\begin{align*}
    \dim\Hom(P_\epsilon, V_\epsilon) = \dim\Hom(P_{\bar K}, V_{\bar K}) &= 1,\\
    \dim\Hom(P_\epsilon, V_{\bar K}) = \dim\Hom(P_{\bar K}, V_\epsilon) &= 0.
\end{align*}
Note that $P_\epsilon\otimes V_{\bar K}$ is projective because tensor products with a projective module are always projective. In particular, either $P_\epsilon\otimes V_{\bar K}\cong P_\epsilon$ or $P_\epsilon\otimes V_{\bar K}\cong P_{\bar K}$ because these are the only projective $\Kn$-modules with the correct dimension. Note that $W(1+K)\otimes 1$ is a basis for $P_\epsilon\otimes V_{\bar K}$. Let $f:P_\epsilon\otimes V_{\bar K}\to V_{\bar K}$ be given by $f:w(1+K)\otimes 1\mapsto \bar K(w)$. It's easy to verify this is $\Kn$-linear, so $\dim\Hom(P_\epsilon\otimes V_{\bar K}, V_{\bar K}) \geq 1$, and thus $P_\epsilon\otimes V_{\bar K}\cong P_{\bar K}$.

We show the fusion rule for the projective indecomposable $\Kn$-modules for just $P_\epsilon\otimes P_\epsilon$. The others are similar. Note that $P_\epsilon\otimes P_\epsilon$ is projective because tensor products of projective modules are projective. In particular, $P_\epsilon\otimes P_\epsilon\cong d^\epsilon P_\epsilon\oplus d^{\bar K}P_{\bar K}$. 
In particular, $d^\epsilon = \dim\Hom(P_\epsilon\otimes P_\epsilon, V_\epsilon)$ and $d^{\bar K} = \dim\Hom(P_\epsilon\otimes P_\epsilon, V_{\bar K})$. Moreover, since $P_\epsilon$ and $P_{\bar K}$ are $2^{n}$-dimensional vector spaces and $P_\epsilon\otimes P_\epsilon$ is a $2^{2n}$-dimensional vector space, we must have that $d^\epsilon + d^{\bar K} = 2^n$. Suppose $f:P_\epsilon\otimes P_\epsilon\to V_\epsilon$ is an $\Kn$-module homomorphism. Since $W(1+K)\otimes W(1+K)$ is a basis for $P_\epsilon\otimes P_\epsilon$, $f$ is completely determined by the complex numbers $c_{w, w'} := f(w(1+K)\otimes w'(1+K))$ for $w,w'\in W$. Let $c_{0, w} = c_{0, w} = 0$ for $w\in W$. Then, for any $i$, we have the following:
\begin{align*}
    c_{w, w'} &= f(w(1+K)\otimes w'(1+K)) \\
    &= K\cdot f(w(1+K)\otimes w'(1+K)) \\
    &= f(Kw(1+K)\otimes Kw'(1+K)) = (-1)^{|w|+|w'|}c_{w,w'}\\
    0 &= \xi_i \cdot f(w(1+K)\otimes w'(1+K)) \\
    &= f(Kw(1+K)\otimes \xi_i w'(1+K) + \xi_i w(1+K)\otimes w'(1+K)) \\
    &= \pm c_{w, \pm\xi_i w'} \pm c_{\pm\xi_i w, w'}.
\end{align*}
In particular, if $|w| + |w'|$ is odd, then $c_{w, w'} = 0$. Moreover, if $\mathbf{I}(w)\cap \mathbf{I}(w') \neq \varnothing$, then $c_{w, w'} = 0$. Otherwise, if $w^\cup=\mathbf{w}(\mathbf{I}(w)\cup \mathbf{I}(w'))$, then $c_{w, w'}$ is determined by $c_{1, w^{\cup}}$. Combining these two facts, we see that the solution space subject to these relations is at most $2^{n-1}$-dimensional, so $d^{\epsilon} \leq 2^{n-1}$, but this is also true for $d^{\bar K}$, so $d^{\epsilon} = d^{\bar K} = 2^{n-1}$.  
\end{proof}

    \begin{theorem}\label{DKn-modules}
        There are 2 non-isomorphic irreducible $D\Kn$-modules $V_1$ and $V_{K\bar K}$ of dimension 1, each with a projective cover of dimension $2^{2n}$, $P_1$ and $P_{K\bar K}$. There are two other non-isomorphic irreducibles $V_K$ and $V_{\bar K}$, each of dimension $2^n$, which are also projective. Thus, the total rank of $D\Kn$-mod is 6 for all $n$.
    \end{theorem} 
    \begin{proof}
        The non-projective simple modules $V_1, V_K$ are each 1-dimensional as $\bC$-vector spaces generated by the following multiplications: for each $i=1,\dots, n$ and $z\in \bC$,
        \begin{alignat*}{6}
            &V_1&:\ \ &\ &K\cdot z&=\ &\bar K\cdot z&=&z\ \ \ \ \ \ &&\xi_i\cdot z=\bar\xi_i\cdot z = 0,\\
            &V_{K\bar K}&:\ \ &\ &K\cdot z&=\ &\bar K\cdot z&=-&\,z\ \ \ \ \ \ &&\xi_i\cdot z=\bar\xi_i\cdot z = 0,
        \end{alignat*}
        Their projective covers are 
        \begin{alignat*}{2}
            &P_1 &\ = D\Kn (1+K+\bar K+K\bar K),\\
            &P_{K\bar K} &\ = D\Kn (1-K-\bar K+K\bar K).
        \end{alignat*}
        respectively. Each has dimension $2^{2n}$ as $\bC$-vector spaces. Note $P_1\oplus P_{K\bar K}\oplus D\Kn(K - \bar K) = D\Kn$, so both are clearly projective. Moreover, by a similar analysis for the projective indecomposables $P_\epsilon, P_{\bar K}\in \Kn$-mod, these are indeed projective covers of their associated simples.

        The two remaining simple $D\Kn$-modules are $V_K$ and $V_{\bar K}$, where $V_K = V_{\bar K} = (\bC^2)^{\otimes n}$ as vector spaces. Let $I_2$ be the $2\times 2$ identity matrix, $\sigma_Z = \begin{pmatrix}1 & 0\\0 & -1\end{pmatrix}$ be the Pauli $Z$-matrix, $\Xi = \begin{pmatrix}0&\sqrt{2}\\0 & 0\end{pmatrix}$. Define the following $2^n\times 2^n$-matrices
        \begin{align*}
            \Xi_i &= \sigma_Z^{\otimes i-1}\otimes\Xi\otimes I_2^{\otimes n - i},\ \ i=1,\dots, n,\\
            \bar\Xi_i& = \sigma_Z^{\otimes i-1}\otimes\Xi^T\otimes I_2^{\otimes n - i},\ \ i=1,\dots, n,\\
            \cZ &= \sigma_Z^{\otimes n},
        \end{align*}
        Finally, for $v\in (\bC^2)^{\otimes n}$, define the following two actions: 
        \begin{alignat*}{7}
            &V_K&:\ \ &\ &K\cdot v&=-&\bar K\cdot v&=&\cZ v\ \ \ \ \ \ &&\xi_i\cdot v=\Xi_i v\ \ \ \ \ \ &&\bar\xi_i\cdot v = \bar\Xi_i v,\\
            &V_{\bar K}&:\ \ &-&K\cdot v&=\ &\bar K\cdot v&=&\cZ v\ \ \ \ \ \ &&\xi_i\cdot v=\Xi_i v\ \ \ \ \ \ &&\bar\xi_i\cdot v = \bar\Xi_i v,
        \end{alignat*}
        
        To show $V_K$ and $V_{\bar K}$ are $D\Kn$-modules, it suffices to show these actions preserve the relations of the underlying algebra of $D\Kn$. All of the relations follow from the following identities:
        $$\sigma_Z^2 = I_2,\hspace{3em}\Xi^2 = 0,\hspace{3em}\Xi\sigma_Z = -\sigma_Z\Xi,\hspace{3em}\Xi\Xi^T + \Xi^T\Xi = 2I_2 = I_2 - \sigma_Z(-\sigma_Z).$$

        Next, we show these are simple. We start by defining a few special elements of $D\Kn$. Let $\omega$ be an arbitrary $n$-bit string with $1$s occurring in the $i_1<i_2<\dots<i_k$ spots. Then, let
        $$\Phi^\omega = \frac{\xi_{i_1}}{\sqrt{2}}\frac{\xi_{i_2}}{\sqrt{2}}\dots\frac{\xi_{i_k}}{\sqrt{2}}, \hspace{15pt}\Psi^\omega = \frac{\bar\xi_{i_k}}{\sqrt{2}}\dots\frac{\bar\xi_{i_2}}{\sqrt{2}}\frac{\bar\xi_{i_1}}{\sqrt{2}}.$$ 
        Note that $\Phi^\omega\cdot |0^n\rangle = |\omega\rangle$, so $|0^n\rangle$ generates $V_K$, as vectors of the form $|\omega\rangle\in V_K$ form a basis. Moreover, $\Psi^\omega\cdot |\omega\rangle = |0^n\rangle$, so $|\omega\rangle$ generates $V_K$ for any $\omega$. If the $i$th bit of $\omega$ is 0, let $\Pi_i^\omega = \frac{\xi_i\bar\xi_i}{2}$. Otherwise, let $\Pi_i^\omega = \frac{\bar\xi_i\xi_i}{2}$. Set
        $$\Pi^{\omega} = \Pi_1^\omega \Pi_2^\omega\dots \Pi_n^\omega.$$
        $\Pi^\omega$ is a projection matrix onto the subspace $\bC|\omega\rangle$. Given a nonzero vector $v\in V_K$, we may find an $n$-bit string $\omega$ for which $\Pi^\omega \cdot v = c|\omega\rangle$ for some nonzero $c\in\bC$, as the $|\omega\rangle$ form a basis. Thus, any $v\in V_K$ generates $V_K$, so $V_K$ is simple. Similarly, $V_{\bar K}$ is simple. 

        Suppose $f:V_K\to V_{\bar K}$ is a $D\Kn$-module homomorphism. Then, $f$ is completely determined by its image on $|00\dots 0\rangle$. Suppose $f(|00\dots 0\rangle) = v$. Then, 
        $$-\cZ v = f(K\cdot|00\dots 0\rangle) = f(|00\dots 0\rangle) = v.$$
        However, for the same reasons as above $v = c|00\dots 0\rangle$, so we get $-c|00\dots 0\rangle=-c\cZ |00\dots 0\rangle = c|00\dots 0\rangle$, so $c=0$ and $f = 0$ identically. Thus, $\Hom(V_K, V_{\bar K}) \cong 0$, and these are distinct simple objects.

        This list of simples and projective indecomposables is complete by the decomposition of $D\Kn$ as a $D\Kn$-module:
            $$D\Kn = \bigoplus_{V\text{ simple}} \dim(V) P_V.$$
        Moreover, this implies $V_K$ and $V_{\bar K}$ are both projective.
    \end{proof}
    \begin{corollary}\label{cor:rank-finiteness-fails}
        $D\Kn$-mod and $D\mcK_m$-mod are not equivalent as categories for any $n\neq m$. In particular, there are infinitely many inequivalent modular categories of total rank $6$. 
    \end{corollary}
    \begin{proof}
        Suppose $D\Kn$-mod and $D\mcK_m$-mod were equivalent. Then, the projective indecomposable object $P_1\in D\Kn$-mod must be sent, under the equivalence, to a projective indecomposable object $P\in D\mcK_m$-mod. Note that 
        $$\End_{D\Kn}(P_1)\neq\bC\cong \End_{D\mcK_m}(V_K)\cong \End_{D\mcK_m}(V_{\bar K}),$$ 
        so $P_1\in D\Kn$-mod may not be sent to $V_K, V_{\bar K}\in D\mcK_m$-mod. 
        
        Let $f_1^i\in \End_{D\Kn}(P_1)$ and $f_{K\bar K}^i\in \End_{D\Kn}(P_{K\bar K})$ be generated by 
        \begin{align*}
            f_1^i&:(1+K+\bar K+K\bar K)\mapsto \xi_i\bar\xi_i (1+K+\bar K+K\bar K),\\
            f_{K\bar K}^i&:(1-K-\bar K+K\bar K)\mapsto \xi_i\bar\xi_i (1-K-\bar K+K\bar K).
        \end{align*} 
        Let $S_1 = \{f_1^1,\dots, f_1^n\}$ and $S_{K\bar K} = \{f_{K\bar K}^1,\dots, f_{K\bar K}^n\}$. The sets $S_1$ and $S_{K\bar K}$ are $n$-element sets of nilpotents for which all products over ordered subsets of $S_1$ and $S_{K\bar K}$ are nonzero.  There are no larger subsets of $\End_{D\Kn}(P_1)$ or $\End_{D\Kn}(P_{K\bar K})$ with this property. In particular, this invariant distinguishes the monoid $\End_{D\Kn}(P_1)$ from $\End_{D\mcK_m}(P_1)$ and $\End_{D\mcK_m}(P_{K\bar K})$.
    \end{proof}
    The argument presented above may be slightly modified to show $\Kn$-mod and $\mcK_m$-mod are also inequivalent as categories for any $n\neq m$.
    
    \begin{theorem}
        The Cartan matrix for $D\Kn$-mod is given by 
        $$C = \begin{pmatrix}
            2^{2n-1} & 2^{2n-1} & 0 & 0\\
            2^{2n-1} & 2^{2n-1} & 0 & 0\\
            0 & 0 & 1 & 0\\
            0 & 0 & 0 & 1
        \end{pmatrix}.$$
    \end{theorem}
    \begin{proof}
        Let $T_\ell = \{w\bar w'\ |\ |w|+|\bar w'| = \ell, w,w'\in W\} = \{t_1^\ell, t_2^\ell,\dots, t_{k_i}^\ell\}$. Let $U_k^\ell = \spa(\bigcup_{j=\ell+1}^{2n} T_j \cup \{t^{\ell}_{k+1}, t^{\ell}_{k+2},\dots, t^{\ell}_{k_i}\})$.  Set $e_1 = (1 + K + \bar K + K\bar K)$ and $e_{K\bar K} = (1 - K - \bar K + K\bar K)$. Note that $U_k^\ell e_h$ are $D\Kn$-modules because, for any $x, b\in B$, there are $b'\in B$ and $\lambda\in \{-1, 0, 1\}$ such that $xbe_h = \lambda b'e_h$ and $\mathbf{I}(b)\subseteq \mathbf{I}(b')$. We can think of $U_k^\ell$ as a removing all $j$-forms for $j < \ell$ and the first $k$ $\ell$-forms and form a basis for a subspace of the exterior algebra $\Lambda\bC^{2n}\cong_{\bC} P_h$ from the remaining elements. The following are composition series for $P_1$ and $P_{K\bar K}$:
        \begin{center}
            \begin{tabular}{llll}
                $0 = U_{1}^{2n}e_h$ &&&\\
                $\subsetneq U_{2n}^{2n-1}e_h$ & $\subsetneq U_{2n-1}^{2n-1}e_h$ & $\subsetneq\dots$ & $\subsetneq U_{1}^{2n-1}e_h$\\
                $\subsetneq U_{k_{2n-2}}^{2n-2}e_h$ & $\subsetneq U_{k_{2n-2} - 1}^{2n-2}e_h$ & $\subsetneq\dots$ & $\subsetneq U_{1}^{2n-2}e_h$\\
                $\vdots$ & $\vdots$ & $\vdots$ & $\vdots$\\
                $\subsetneq U_{k_2}^{2}e_h$ & $\subsetneq U_{k_2-1}^{2}e_h$ & $\subsetneq\dots$ & $\subsetneq U_{1}^{2}e_h$\\
                $\subsetneq U_{2n}^{1}e_h$ & $\subsetneq U_{2n-1}^{1}e_h$ & $\subsetneq\dots$ & $\subsetneq U_{1}^{1}e_h$\\
                $\subsetneq U_1^{0}e_h$ & $\subsetneq D\Kn e_h = P_h$ &
            \end{tabular}
        \end{center}
        The composition factors are determined by the sign of $Kt_k^\ell e_h = \pm t_k^\ell e_h$. In particular, the composition factors are 
        $$U_k^\ell e_h / {U_{k-1}^\ell e_h} \cong \begin{cases}
            V_1 & \ell\text{ even and }h=1\text{ or }\ell\text{ odd and }h=K\bar K\\
            V_{K\bar K} & \ell\text{ odd and }h=1\text{ or }\ell\text{ even and }h=K\bar K\\
        \end{cases}$$
        where we set $U_{0}^\ell := U_{k_{\ell-1}}^{\ell-1}$. Because $k_\ell = \binom{2n}{\ell}$ and $\sum_{m=0}^n k_{2m} = \sum_{m=0}^{n-1} k_{2m+1} = 2^{2n-1}$, there are precisely this many copies of each of $V_1$ and $V_{K\bar K}$ among the composition factors of $P_h$ for $h=1,K\bar K$.

        Both $V_K$ and $V_{\bar K}$ are irreducible, so they are their own only composition factor.
    \end{proof}
    
    \begin{theorem}
        Suppose $n$ is even. Then, the fusion rules in $D\Kn$-mod are
        \begin{align*}
            V_{K\bar K}^2&\cong V_1, & V_KV_{K\bar K}&\cong V_{\bar K}, & V_{\bar K}V_{K\bar K}&\cong V_{K},\\
            V_K^2\cong V_{\bar K}^2&\cong P_1, & V_K V_{\bar K}&\cong P_{K\bar K},
        \end{align*}
        \begin{align*}
            V_{K} P_1\cong V_{K} P_{K\bar K}\cong V_{\bar K} P_1\cong V_{\bar K} P_{K\bar K}&\cong 2^{2n-1}V_K\oplus 2^{2n-1}V_{\bar K},\\
            P_{K\bar K}^2\cong P_{1}^2\cong P_{K\bar K}P_1&\cong 2^{2n-1}P_1\oplus 2^{2n-1}P_{K\bar K}.
        \end{align*}
    \end{theorem}
    Minor changes to the proof below show that, if $n$ is odd, all fusion rules are the same except $V_K^2\cong V_{\bar K}^2\cong P_{K\bar K}$ and $V_K V_{\bar K}\cong P_1$.
    \begin{proof}
        Clearly $V_{K\bar K}^2\cong V_1$ because $V_1$ and $V_{K\bar K}$ are the only 1-dimensional $D\Kn$-modules (up to isomorphism) and $V_{K\bar K}$ must have a dual of the same dimension.
        
        Note that $V_K$ is projective, so $V_K\otimes V_{K\bar K}$ is also projective. Moreover, this is of dimension $2^{n}$, so either $V_K\otimes V_{K\bar K}\cong V_K$ or $V_K\otimes V_{K\bar K}\cong V_{\bar K}$. It suffices to check if $\Hom(V_K\otimes V_{K\bar K}, V_K) = 0$. Suppose $f:V_K\otimes V_{K\bar K}\to V_K$. Then, $f(\Delta(K)|00\dots 0\rangle\otimes 1) = f(K|00\dots 0\rangle\otimes K1) = -f(|00\dots 0\rangle\otimes 1)$. This implies $Kf(|00\dots 0\rangle\otimes 1) = Kf(|00\dots 0\rangle\otimes 1)$. Moreover,
        $$\Delta(\xi_i)v\otimes 1 = (K\otimes \xi_i + \xi_i\otimes 1)v\otimes 1 = (\xi_i v)\otimes 1.$$
        The same is true for $\bar\xi_i$. In particular, this implies that 
        $$c|00\dots 0\rangle = \Pi f(|00\dots 0\rangle\otimes 1) =  f((\Pi|00\dots 0\rangle)\otimes 1) = f(|00\dots 0\rangle\otimes 1) = -Kf(|00\dots 0\rangle \otimes 1)$$
        and $$ f(|00\dots 0\rangle\otimes 1) = -f(|00\dots 0\rangle\otimes 1),$$
        for some $c\in\bC$. As before $|00\dots 0\rangle\otimes 1$ generates the entire $D\Kn$-module, so $f = 0$. Thus, $V_K\otimes V_{K\bar K}\cong V_{\bar K}$. Showing $V_{\bar K}\otimes V_{K\bar K}\cong V_{K}$ is similar. 

        Define the linear map $f:P_1\to V_K^2$ by $1+K+\bar K+K\bar K\mapsto |00\dots 0\rangle\otimes |11\dots 1\rangle$ and extending $D\Kn$-linearly. Note that $P_1\cong D\Kn/\langle 1=K=\bar K\rangle$ as $D\Kn$-modules. Moreover, $f$ lifts to a $D\Kn$-linear map $\tilde f:D\Kn\to V_K^2$ generated by $1\mapsto \frac{|00\dots 0\rangle\otimes |11\dots 1\rangle}{4}$ because $K$ and $\bar K$ fix $|00\dots 0\rangle\otimes |1\dots 1\rangle$. In particular, this implies that $f$ is well-defined. If $f:P_1\to M$ is a $D\Kn$-linear map (and $M$ is a $D\Kn$-module) whose kernel is nonzero, then the element $\xi_1\dots\xi_n\bar\xi_1\dots\bar\xi_n(1+K+\bar K+K\bar K)\in \ker f$. Note that $f(\bar\xi_1\dots\bar\xi_n(1+K+\bar K+K\bar K)) = \sqrt{2}^n |00\dots 0\rangle\otimes |00\dots 0\rangle$. Finally, $\Delta(\xi_1\dots\xi_n)|00\dots 0\rangle\otimes |00\dots 0\rangle = \sqrt{2}^n\sum  |\omega\rangle\otimes |\omega^c\rangle$, where, for a bit string $\omega$, $\omega^c$ is the bit string will all bits flipped. Clearly, this sum is nonzero, so $f$ has zero kernel and hence is injective. However, $P_1$ and $V_K^2$ are vector spaces of the same dimension, so $f$ is an isomorphism. Similar arguments show that $V_{\bar K}^2\cong P_1$ and $V_K V_{\bar K}\cong P_{K\bar K}$.

        By applying $V_K\otimes$the composition series of $P_1$ and noting that $H$-submodules and quotients behave well under $\otimes$, we see that $V_K\otimes P_1$ has a composition series whose composition factors are $2^{2n-1}$ copies of $V_K$ and $2^{2n-1}$ copies of $V_{\bar K}$. Since $V_K\otimes P_1$ is projective, it must decompose into a direct sum of projective indecomposables, but the only way the direct sum decomposition agrees with the composition series is if $V_K\otimes P_1\cong 2^{2n-1}V_K\oplus 2^{2n-1}V_{\bar K}$. The same arguments apply for $V_{K} P_{K\bar K}, V_{\bar K} P_1, V_{\bar K} P_{K\bar K}$.
 
        The last line of fusion rules follows from the others.
    \end{proof}
    
    \begin{proposition}\label{non-modular}
        The semi-simplification of $D\Kn$-$\mathrm{mod}$ is not modular.
    \end{proposition}
    \begin{proof}
        The only remaining simples after semi-simplification are $V_1$ and $V_{K\bar K}$. We show that the semisimple $S$-matrix indexed by these two simples is singular. Note that $\ev_{V_1}:V_1\otimes V_1\to V_1$ and $\coev:V_1\to V_1\otimes V_1$ are generated by $1_{V_1}\otimes 1_{V_1}\mapsto 1_{V_1}$ and $1_{V_1}\mapsto 1_{V_1}\otimes 1_{V_1}$ respectively. Similarly, $\ev_{V_{K\bar K}}:V_{K\bar K}\otimes V_{K\bar K}\to V_1$ and $\coev:V_1\to V_{K\bar K}\otimes V_{K\bar K}$ are generated by $1_{V_{K\bar K}}\otimes 1_{V_{K\bar K}}\mapsto 1_{V_1}$ and $1_{V_1}\mapsto 1_{V_{K\bar K}}\otimes 1_{V_{K\bar K}}$ respectively. 
        
        Let $h, k\in \{1, K\bar K\}$. Let $\beta$ be the braiding induced by $R$. Then, for any $x\in V_h$ and $y\in V_k$,
        \begin{align*}
            \beta_{V_k, V_h}\circ \beta_{V_h, V_k}(x\otimes y) &= R_{21}R(x\otimes y)\\
            &= \frac{1}{4}(1 + (-1)^{\delta_{h=K\bar K}} + (-1)^{\delta_{k=K\bar K}} - (-1)^{\delta_{h=K\bar K} + \delta_{k=K\bar K}})^2x\otimes y = x\otimes y,
        \end{align*}
        where the middle equality follows from the fact that only the $w = 1$ term in $R$ and $R_{21}$ has a nonzero action on $V_h\otimes V_k$. Therefore, for $h, k\in\{1, K\bar K\}$, we have, by the Hopf link formula for the semisimple $S$-matrix
         \begin{align*}
             S_{hk} &= \frac{1}{D}(\ev_{V_h}\otimes \ev_{V_k})\circ(\id_{V_h}\otimes (\beta_{V_k, V_h}\circ\beta_{V_h, V_k})\otimes \id_{V_k})\circ(\coev_{V_h}\otimes \coev_{V_k})(1_{V_1})\\ 
             &= \frac{1}{\sqrt{2}} (\mathrm{qdim}(V_h)\mathrm{qdim}(V_k)) = \frac{1}{\sqrt{2}},
         \end{align*}
         Thus, the semi-simplified $S$-matrix is $\dfrac{1}{\sqrt{2}}\begin{pmatrix}
             1&1\\
             1&1
         \end{pmatrix}$ and is not invertible, so the result follows.
    \end{proof}

    While we have defined $V_K$ and $V_{\bar K}$ abstractly, $D\Kn$ must contain $2^n$ copies of each as direct summands (in addition to the one of each of $P_1$ and $P_{K\bar K}$). The following lemma gives a description of the corresponding inclusions using their image on the generator $|00\dots 0\rangle$.
    \begin{lemma}\label{CKinclusions}
        Suppose $\iota:V_K\hookrightarrow D\Kn$ is a $D\Kn$-linear inclusion. Then, there are constants $c_w\in \bC, w\in W$ such that
        $$\iota(|00\dots 0\rangle) = (1 + K - \bar K - K\bar K)\xi_1\xi_2\dots\xi_n\sum_{w\in W}c_w \bar w.$$
        Moreover, all such choices of constants $c_w,w\in W$, so long as at least one is nonzero, generate inclusions. A similar statement is true for $\iota:V_{\bar K}\hookrightarrow D\Kn$ where
        $$\iota(|00\dots 0\rangle) = (1 - K + \bar K - K\bar K)\xi_1\xi_2\dots\xi_n\sum_{w\in W}c_w \bar w.$$
    \end{lemma} 
    \begin{proof}
        Let $\iota:V_K\hookrightarrow D\Kn$ be an inclusion. Note that $W\bar W\cup KW\bar W\cup \bar KW\bar W\cup K\bar KW\bar W$ is a basis for $D\Kn$, so $\iota(|00\dots 0\rangle)$ can be expressed as
        $$\iota(|00\dots 0\rangle) = \sum_{w,w'\in W} (c_1(w, w') + c_K(w, w')K + c_{\bar K}(w, w')\bar K + c_{K\bar K}(w, w')K\bar K)w\bar w'.$$
        In $V_K$, we have the following properties: $K|00\dots 0\rangle = |00\dots 0\rangle$, $\bar K|00\dots 0\rangle = -|00\dots 0\rangle$, and $\xi_i|00\dots 0\rangle = 0$ for any $i=1,\dots,n$. By $D\Kn$-linearity of the inclusion, the first two imply that 
        $$c_1(w, w') = c_K(w, w') = - c_{\bar K}(w, w') = - c_{K\bar K}(w, w').$$
        Set $c(w, w') = c_1(w, w')$, then 
        $$\iota(|00\dots 0\rangle) = \sum_{w,w'\in W} c(w, w')(1 + K - \bar K - K\bar K)w\bar w'.$$
        The last identity implies that, for any $i=1,\dots, n$,
        \begin{align*}
            0 &= \sum_{w,w'\in W} c(w, w')(1 - K + \bar K - K\bar K)\xi_i w\bar w'\\
            \intertext{We split the summands into two cases: $\xi_i\in \mathbf{I}(w)$ or $\xi_i\notin \mathbf{I}(w)$. In the case where $\xi_i\in \mathbf{I}(w)$, note that}
            \xi_iw\bar w' &= 0.
            \intertext{In the case where $\xi_i\notin \mathbf{I}(w)$, let $w = w_1 w_2$ where $\mathbf{I}(w_1)$ contains no $\xi_j$ for $j \geq i$ and $\mathbf{I}(w_2)$ contains no $\xi_j$ for $j\leq i$. Then,}
            \xi_iw\bar w' &= (-1)^{|w_1|}w_1\xi_i w_2 \bar w'.
            \intertext{Note that a nonzero coefficient on the basis element $w_1\xi_i w_2 \bar w'\in W\bar W\cup KW\bar W\cup \bar KW\bar W\cup K\bar KW\bar W$ occurs only on this choice of $w$ and $w'$. In particular, this implies that if $\xi_i\notin \mathbf{I}(w)$, then $c(w, w') = 0$. Thus, the only terms where $c(w, w')$ could possibly be nonzero are those where $w = \xi_1\xi_2\dots\xi_n$. For $w\in W$, set $c_{w} = c(\xi_1\xi_2\dots\xi_n, w)$. Then, $\iota(|00\dots 0\rangle)$ can be expressed as}
            \iota(|00\dots 0\rangle) &= (1 + K - \bar K - K\bar K)\xi_1\xi_2\dots\xi_n\sum_{w\in W} c_{w}\bar w.
        \end{align*}
        By this expression, there are $2^n$ linearly independent choices of $\iota(|00\dots 0\rangle)$ and $|00\dots 0\rangle$ generates $V_K$, a simple $D\Kn$-module. Thus, all choices of $c_w$ which are not all zero must generate inclusions. The argument for $V_{\bar K}$ is similar.
    \end{proof}
    \begin{corollary}\label{CKidempotents}
        Let $\iota:V_K\hookrightarrow D\Kn$ or $\iota:V_{\bar K}\hookrightarrow D\Kn$ be an $D\Kn$-linear inclusion so that $$\iota(|00\dots 0\rangle) = (1 \pm K \mp \bar K - K\bar K)\xi_1\xi_2\dots\xi_n\sum_{w\in W}c_w \bar w.$$
        For each $w\in W$, let $w^c = \mathbf{w}(\mathbf{I}(w)^c)$ and defined $\gamma_w\in \{\pm 1\}$ by $ww^c = \gamma_w\xi_1\dots\xi_n$. Then, for any choice of $d_w\in\bC, w\in W$ satisfying
        $$\sum_{w\in W} \gamma_w d_{w^c} c_{w} = 1,$$ 
        the element
        $$e = \frac{(-1)^{\lfloor n/2\rfloor}}{2^{n+2}}\sum_{w\in W} d_w \bar w \iota(|00\dots 0\rangle)$$
        is an idempotent and $D\Kn e = \mathrm{im}\,\iota$. Moreover, all idempotents satisfying $D\Kn e = \mathrm{im}\,\iota$ are of this form. 
    \end{corollary}
    By Corollary \ref{CKidempotents}, two particular idempotents which realize $V_K\cong D\Kn e_K$ and $V_{\bar K}\cong D\Kn e_{\bar K}$ are
    \begin{align*}
        e_K &= \frac{(-1)^{\lfloor n/2\rfloor}}{2^{n+2}}(1 + K - \bar K - K\bar K)\xi_1\xi_2\dots\xi_n\bar\xi_1\bar\xi_2\dots\bar\xi_n,\\
        e_{\bar K} &= \frac{(-1)^{\lfloor n/2\rfloor}}{2^{n+2}}(1 - K + \bar K - K\bar K)\xi_1\xi_2\dots\xi_n\bar\xi_1\bar\xi_2\dots\bar\xi_n.
    \end{align*}

\subsection{$\mathrm{SL}(2,\bZ)$-representations on the center of $D\Kn$}
    \begin{theorem}
    The center $Z(D\Kn)$ of $D\Kn$ decomposes as follows: 
    \begin{align*}
        Z(D\Kn) &= Z_H\oplus Z_\Lambda\\
        Z_H &= \spa\left\{1 - K\bar K, (K+\bar K)\xi_1\dots\xi_n\bar\xi_1\dots\bar\xi_n, (K - \bar K)\sum_{w\in W} (-1)^{\lfloor (|w|+1) / 2\rfloor}w\bar w\right\},\\
        Z_\Lambda &= \spa\{(1 + K\bar K)w\bar w'| w,w'\in W, |w| + |w'|\text{ is even} \}. 
    \end{align*}
    \end{theorem}
    $Z_\Lambda$ comes from the near exterior algebra structure in $D\Kn$. In particular, consider the (non-unital) algebra inclusion $\Lambda^* \bC^{2n}\hookrightarrow D\Kn$ generated by $\xi_i\mapsto \frac{1}{2}\xi_i(1+K\bar K)$ and $\xi_{i+n}\mapsto \frac{1}{2}\bar\xi_i(1+K\bar K)$ for $0\leq i \leq n$. $Z_\Lambda$ is precisely the image of $Z(\Lambda^* \bC^{2n})$. We will show in Lemma \ref{DKn-two-bases} that $Z_H = \mathrm{Hig}(D\Kn)$ for even $n$.
    \begin{proof}
        By expressing $z\in Z(D\Kn)$ in terms of the basis $W\bar W\cup KW\bar W\cup \bar KW\bar W\cup K\bar KW\bar W$, this becomes a routine calculation. Natural restrictions on the coefficients arise by breaking into cases based on the values of $a, b$, whether $\xi_i\in \mathbf{I}(w)$, and whether $\xi_i\in \mathbf{I}(w')$ for each $i=1,\dots,n$ and basis element of the form $K^a\bar K^b w\bar w'$. These lead to the following general formula:
        \begin{align*}
            z = \alpha &+ \beta(K + \bar K)\xi_1\dots\xi_n\bar\xi_1\dots\bar\xi_n + \delta\sum_{w\in W} (-1)^{\lfloor (|w|+1) / 2\rfloor}(K - \bar K)w\bar w \\
            &+ \sum_{\substack{(w, w')\in W^2\\ |w|+|w'|\text{ even}}} c(w, w')(1 + K\bar K)w\bar w'.
        \end{align*}
    \end{proof}
    \begin{proposition}
       The center $Z(\Kn)$ of $\Kn$  decomposes as follows:
        \begin{align*}
            Z(\Kn) &= Z_\Lambda\oplus \begin{cases}
                \bC K\xi_1\dots\xi_n, & n \text{ even},\\
                0, & n\text{ odd},
            \end{cases}\\
            Z_\Lambda &= \spa\{w\in W|\ |w|\text{ is even} \}. 
        \end{align*}
    \end{proposition}
    \begin{proof}
        Express a general element in terms of the basis $W\cup KW$. By commutativity with $K$, there are no nonzero coefficients on basis elements of the form $w\in W$ with $|w|$ odd. If $|w|\neq n$, choose $i\notin \mathbf{I}(w)$. By commutativity with $\xi_i$, there are no nonzero coefficients on the basis element $Kw$.
    \end{proof}
    \begin{lemma}
        $D\Kn$ is unimodular for all $n$. Let $x\in D\Kn$ be a left or right integral. Then, there is a constant $c\in\bC$ so that
        $$x = c(1 + K + \bar K + K\bar K)\xi_1\dots \xi_n\bar\xi_1\dots \bar\xi_n.$$
    \end{lemma}
    \begin{lemma}
        Let $g:H\to \bC$ be a left integral of $(D\Kn)^*$, then there is a constant $c\in\bC$ such that
        $$g = c(\xi_1\dots \xi_n\bar\xi_1\dots\bar\xi_n)^*,$$
        using the dual basis of $W\bar W\cup KW\bar W\cup \bar KW\bar W\cup K\bar KW\bar W$.
    \end{lemma}

    For the remainder of this section, we must choose an integral/cointegral pair $(\Lambda, \lambda)$ so that $\lambda(\Lambda) = 1$. We set $\lambda = 2(\xi_1\dots \xi_n\bar\xi_1\dots\bar\xi_n)^*$ and $\Lambda = \frac{1}{2}(1+K+\bar K+K\bar K)\xi_1\dots\xi_n\bar\xi_1\dots\bar\xi_n$. Note that $\lambda = 2(-1)^n(\bar\xi_1\dots\bar\xi_n\xi_1\dots \xi_n)^*$ using the basis $\bar WW\cup K\bar WW\cup \bar K\bar WW\cup K\bar K\bar WW$ for $D\Kn$. 
    
    For each $w\in W$, once again, let $w^c = \mathbf{w}(\mathbf{I}(w)^c)$ and defined $\gamma_w\in \{\pm 1\}$ by $ww^c = \gamma_w\xi_1\dots\xi_n$.
    
    \begin{theorem}
         Suppose $n$ is even, $a, b\in\{0,1\}$ and $w_1,w_2\in W$. The Lyubashenko-Majid map on $D\Kn$ can be expressed as
        \begin{align*}
            S_{LM}(K^a\bar K^b \bar w_1 w_2) &= \frac{(-1)^{s(w_1, w_2)}\gamma_{w_1^c}\gamma_{w_2}}{2} w_1^c (1+(-1)^b K+(-1)^a\bar K+(-1)^{a+b}K\bar K))\bar w_2^c,\\
            T_{LM}(K^a\bar K^b \bar w_1 w_2) &= K^{a}\bar K^{b}(1+K-\bar K+K\bar K)\sum_{w\in W} \frac{(-1)^{\lfloor(|w|+1)/2\rfloor}}{2}w\bar w \bar w_1w_2, 
        \end{align*}
        where $s(w_1, w_2) = \lfloor|w_1^c|/2\rfloor+\lfloor|w_2^c|/2\rfloor + (a+b)|w_1|$.
    \end{theorem}
    Note that the ribbon element provided in Lemma \ref{DKn-ribbon} is only valid for even $n$. Thus, the action of $T\in\mathrm{SL}(2,\bZ)$ on $Z(D\Kn)$ described above only applies for even $n$. 
    \begin{proof}
        Note that by the formula for quasitriangular Hopf algebras, $$R^{-1} = (S\otimes \id)(R) = \sum_{w\in W} \frac{(-1)^{\lfloor|w|/2\rfloor}}{2}( w(1+K)\otimes \bar w + w(1-K)\otimes \bar K\bar w).$$
        Finally, $R_{21}^{-1} = (R^{-1})_{21}$. Therefore,
        \begin{align*}
            &{S}_{LM}(x)\\ &=(\id\otimes\lambda)(R^{-1}(1\otimes x)R_{21}^{-1})\\
            &= \sum_{w,w'\in W}\frac{(-1)^{\lfloor|w|/2\rfloor+\lfloor|w'|/2\rfloor}}{4}( \lambda(\bar w x w'(1+K))w(1+K)\bar w' + \lambda(\bar K\bar w x w'(1+K)) w(1-K)\bar w' \\& + \lambda(\bar w x w'(1-K))w(1+K)\bar K\bar w' + \lambda(\bar K\bar w x w'(1-K))w(1-K)\bar K\bar w')
        \end{align*}
        Suppose $x$ is of the form $x = K^a\bar K^b \bar w_1 w_2$ for $a,b\in\{0,1\}$ for $w_1, w_2\in W$. Then, every term in this sum except that which $w = w_1^c$ and $w'=w_2^c$ disappears because $\lambda$ will be zero on all of these terms, leaving
        \begin{align*}
            {S}_{LM}(x) &= \frac{(-1)^{\lfloor|w_1^c|/2\rfloor+\lfloor|w_2^c|/2\rfloor + (a+b)|w_1^c|}\gamma_{w_1^c}\gamma_{w_2}}{4}( \lambda(K^a\bar K^b(1+K)\bar\xi_1\dots\bar\xi_n\xi_1\dots\xi_n)w_1^c(1+K)\bar w_2^c\\ &+ \lambda(K^a\bar K^{b+1}(1+K)\bar\xi_1\dots\bar\xi_n\xi_1\dots\xi_n) w_1^c(1-K)\bar w_2^c \\& + \lambda(K^{a}\bar K^b(1-K)\bar\xi_1\dots\bar\xi_n\xi_1\dots\xi_n)w_1^c(1+K)\bar K\bar w_2^c \\&+ \lambda(K^{a}\bar K^{b+1}(1-K)\bar K\bar\xi_1\dots\bar\xi_n\xi_1\dots\xi_n)w_1^c(1-K)\bar K\bar w_2^c)\\\\
            &= \frac{(-1)^{\lfloor|w_1^c|/2\rfloor+\lfloor|w_2^c|/2\rfloor + (a+b)|w_1^c|}\gamma_{w_1^c}\gamma_{w_2}}{2}(-1)^n w_1^c\times\\ 
            &\times(\delta_{b=0}((1+K)+(-1)^a(1+K)\bar K) + \delta_{b=1}((1-K)+(-1)^a(1-K)\bar K))\bar w_2^c.
        \end{align*}
        The formula for $S_{LM}$ follows. The formula for $T_{LM}$ is easy to verify from its definition.
    \end{proof}

    \begin{corollary}
        Suppose $n$ is even. Restricted to $Z_\Lambda$, $S_{LM}$ and $T_{LM}$ are given by
        \begin{align*}
            S_{LM}((1 + K\bar K)w_1 \bar w_2) &= (-1)^{\lfloor|w_1^c|/2\rfloor+\lfloor|w_2^c|/2\rfloor + |w_1|}\gamma_{w_1^c}\gamma_{w_2} (1+K\bar K) w_1^c\bar w_2^c,\\
            T_{LM}((1 + K\bar K)w_1\bar w_2) &= (1+K\bar K)\sum_{w\in W} (-1)^{|w|} w_1 \bar w_2\hat w,
        \end{align*}
        for $w_1, w_2\in W$.
    \end{corollary}

    Define the set of words $\tilde W = \{\xi_1^{a_1}\bar\xi_1^{b_1}\xi_2^{a_2}\bar\xi_2^{b_2}\dots\xi_n^{a_n}\bar\xi_n^{b_n}|a_k,b_k\in\{0,1\}\}$. Note that $\tilde W\cup K\tilde W\cup \bar K\tilde W\cup K\bar K\tilde W$ is a basis for $D\Kn$. This alternate basis will allow easier expression of the characters. Extend $\mathbf{I}(\tilde w) = \{\xi_i | a_i = 1\}\cup \{\bar\xi_i | b_i = 1\}$ and $|\tilde w| = |\mathbf{I}(\tilde w)|$.
    \begin{theorem}
        Let $n$ be arbitrary. Define $n_i:\tilde W\to \{0,1,2\}$ as $n_i(\tilde w) = \#(\mathbf{I}(\pm\tilde w)\cap \{\xi_i, \bar\xi_i\})$. The irreducible and projective indecomposable characters of $D\Kn$ are generated by
        \begin{align*}
            \chi_{1}(K^a \bar K^b \tilde w) &= \delta_{\tilde w=1}\\
            \chi_{K\bar K}(K^a \bar K^b \tilde w) &= (-1)^{a+b}\delta_{\tilde w=1}\\
            \chi_{K}(K^a \bar K^b \tilde w) &= (-1)^b 2^n\prod_{i=1}^n(\delta_{a=b}\delta_{n_i(\tilde w)=0} + \delta_{n_i(\tilde w)=2})\\
            \chi_{\bar K}(K^a \bar K^b \tilde w) &= (-1)^a 2^n\prod_{i=1}^n(\delta_{a=b}\delta_{n_i(\tilde w)=0} + \delta_{n_i(\tilde w)=2})\\
            p_{1}(K^a \bar K^b \tilde w) = p_{K\bar K}(K^a \bar K^b \tilde w) &= 2^{2n}\delta_{a=b}\delta_{\tilde w=1}
        \end{align*}
        for $a,b\in\{0, 1\}$ and $\tilde w\in \tilde W$.
    \end{theorem}    
    \begin{proof}
        Note that the modules $V_1$ and $V_{K\bar K}$ are both 1-dimensional, so for their corresponding characters, $\chi(h)=h\cdot 1$ for all $h\in D\Kn$. If $|\tilde w| > 0$, then $\chi(K^a\bar K^b \tilde w) = 0$. For $\tilde w = 1$, $\chi_1(K^a \bar K^b) = 1$ and $\chi_{K\bar K}(K^a \bar K^b) = (-1)^{a+b}$.

        For the third line, note that $\tr(A\otimes B) = \tr(A)\tr(B)$ for any square matrices $A, B$. In particular, because the matrix associated with $\tilde w$ in $V_K$ is a pure tensor, $\chi_K(K^a\bar K^b\tilde w) = \prod_{i=1}^n \tr(t_i)$, where $t_i$ is the $i$th tensor component. If $\xi_i\in I(\tilde w)$ but $\bar \xi_i\notin I(\tilde w)$, then 
        $$\chi_K(K^a\bar K^b\tilde w) = \pm z\tr\begin{pmatrix}
            0 & \sqrt{2}\\
            0 & 0
        \end{pmatrix} = 0,$$ 
        where $z$ is the product of the traces of the tensor components of index $\neq i$. Similarly, if $\xi_i\notin I(\tilde w)$ but $\bar \xi_i\in I(\tilde w)$, we also find $\chi_K(K^a\bar K^b\tilde w) = 0$. Thus, if $\chi(K^a \bar K^b\tilde w)\neq 0$, then $n_i(\tilde w) \neq 1$ for all $i$. If $n_i(\tilde w) = 0$ and $a\neq b$, then
        $$\chi_K(K^a\bar K^b\tilde w) = \pm z\tr\begin{pmatrix}
            1 & 0\\
            0 & -1
        \end{pmatrix} = 0,$$
        where again $z$ is the product of the traces of the tensor components of index $\neq i$. Thus, if $a\neq b$ and $\chi_K(K^a\bar K^b\tilde w) \neq 0$, then we must have $\tilde w = \xi_1\bar\xi_1\xi_2\bar\xi_2\dots\xi_n\bar\xi_n$. In this case,
        $$\chi_K(K^a\bar K^b\tilde w) = (-1)^b\left(\tr\begin{pmatrix}
            2 & 0\\
            0 & 0
        \end{pmatrix}\right)^n = (-1)^{b} 2^n = (-1)^{b}2^n.$$
        Finally, suppose $a = b$, $n_i(\tilde w)\neq 1$ for all $i$ and $n_i(\tilde w) = 2$ for exactly $k$ choices of $i$. Then, 
        $$\chi_K(K^b\bar K^b\tilde w) = (-1)^b\left(\tr\begin{pmatrix}
            2 & 0\\
            0 & 0
        \end{pmatrix}\right)^k\left(\tr\begin{pmatrix}
            1 & 0\\
            0 & 1
        \end{pmatrix}\right)^{n-k} = (-1)^{b} 2^n.$$
        Note that the claimed expression for $\chi_K$ agrees with all of our observations, so the third equality holds. The fourth expression can be proven similarly.

        Using $\tilde W(1+K+\bar K+K\bar K)$ as a basis for $P_1$, 
        $$K^a\bar K^b \tilde w\cdot w(1 + K + \bar K + K\bar K) = (-1)^{(a+b)(|\tilde w| +|w|)}\tilde ww(1 + K + \bar K + K\bar K).$$
        In particular, if $|\tilde w| > 0$, then for all $\tilde w'\in \tilde W$, $\tilde w\tilde w'(1 + K + \bar K + K\bar K)$ has 0 coefficient on $\tilde w'(1 + K + \bar K + K\bar K)$, meaning $p_1(K^a\bar K^b\tilde w) = 0$. Next note that 
        $$K\tilde w'(1+K+\bar K+K\bar K) = \bar K\tilde w'(1+K+\bar K + K\bar K) = (-1)^{|\tilde w'|} w(1+K+\bar K + K\bar K).$$
        For each fixed length $\ell$, there are precisely $\binom{2n}{\ell}$ words in $\tilde W$ of length $\ell$. Thus, 
        $$p_1(K) = p_1(\bar K) = \sum_{\ell=0}^n (-1)^\ell \binom{2n}{\ell} = 0.$$
        Finally, $1\tilde w'(1+K+\bar K+K\bar K) = K\bar K\tilde w'(1+K+\bar K+K\bar K)=\tilde w'(1+K+\bar K+K\bar K)$, so $p_1(1) = p_1(K\bar K) = \tr I_{2^{2n}} = 2^{2n}$. The same analysis applies for $p_{K\bar K}$.
    \end{proof}

        
    \begin{lemma}\label{DKn-two-bases}
        Suppose $n$ is even. The two bases $B_{\chi}, B_\tau$ of $\Hig(D\Kn)$ (as described in \eqref{eq:twobases}) are respectively given by:
        \begin{align*}
            \left\{2^{2n-1}(1 - K\bar K), 2^{n-1}\left(\sum_{w\in W} (-1)^{\lfloor (|w|+1)/2\rfloor}(K - \bar K)w\bar w) + (-1)^{n/2}(K + \bar K)\xi_1\dots\xi_n\bar\xi_1\dots\bar\xi_n \right)\right.&,\\
            2^{n-1}\left(\sum_{w\in W} (-1)^{\lfloor (|w|+1)/2\rfloor}(K - \bar K)w\bar w) - (-1)^{n/2}(K + \bar K)\xi_1\dots\xi_n\bar\xi_1\dots\bar\xi_n\right)&\Bigg\}.
        \end{align*}
        and
        \begin{align*}
            \left\{2^{2n-1}(K + \bar K)\xi_1\dots\xi_n\bar\xi_1\dots\bar\xi_n, (-1)^{n/2}2^{n-1}\left(\sum_{w\in W} (-1)^{\lfloor (|w|+1)/2\rfloor}(K - \bar K)w\bar w) + (1 - K\bar K)\right)\right.&,\\
            (-1)^{n/2}2^{n-1}\left(\sum_{w\in W} (-1)^{\lfloor (|w|+1)/2\rfloor}(K - \bar K)w\bar w) - (1 - K\bar K)\right)&\Bigg\}.
        \end{align*}
    \end{lemma}
    \begin{proof}
        First, we compute $B_\chi$. Recall $$R = \sum_{w\in W} \frac{(-1)^{\lfloor|w|/2\rfloor}}{2}( w\otimes \bar w (1+\bar K)+ Kw\otimes \bar w(1-\bar K)).$$ 
        Note that all three indecomposable projective modules have characters which are nonzero on $K^a \bar K^b \tilde w$ for $\tilde w\in W$ only if $n_i(\tilde w)\neq 1$ for all $i$. For $w\in W$, let $\hat w = \prod_{\xi_i\in \mathbf{I}(w)} \xi_i\bar\xi_i\in\tilde W$ The following identity is useful in what follows:
        $$\bar w w = (-1)^{\lfloor |w|/2\rfloor} \sum_{\mathbf{I}(w')\subseteq \mathbf{I}(w)}(1-K\bar K)^{|w|-|w'|}(-1)^{|w'|}\hat w'.$$
        Note that $K$ is the balancing element for $D\Kn$. For a character $\chi$, define $\hat \chi = \chi\leftharpoonup K$. The shifted Drinfeld map is given by $\hat f_Q(\chi) = (\hat\chi\otimes \id)(R_{21}R)$. As a consequence of our projective indecomposable characters being well-behaved and the identity for $\bar w w$ in terms of $\tilde W$, we find that, for $\chi\in \{\chi_{K}, \chi_{\bar K}, p_1, p_{K\bar K}\}$,
        \begin{align*}
            \hat f_Q(\chi) = \sum_{w\in W} \frac{(-1)^{\lfloor|w|/2\rfloor}}{4}(&\chi((K + (-1)^{|w|}1 + (-1)^{|w|}K\bar K + \bar K)\hat w) \\
            + &\chi((K + (-1)^{|w|}1 - (-1)^{|w|}K\bar K - \bar K)\hat w)K\\ 
            + &\chi((K - (-1)^{|w|}1 + (-1)^{|w|}K\bar K - \bar K)\hat w)\bar K \\
            + &\chi((K - (-1)^{|w|}1 - (-1)^{|w|}K\bar K + \bar K)\hat w)K\bar K)w\bar w.
        \end{align*}
        Using this formula, it is easy to verify that
        \begin{align*}
            \hat f_Q(p_1) &= \hat f_Q(p_{K\bar K}) = 2^{2n-1}(1-K\bar K),\\
            \hat f_Q(\chi_K) &= 2^{n-1}\left(\sum_{w\in W} (-1)^{\lfloor (|w|+1)/2\rfloor}(K - \bar K)w\bar w) + (-1)^{\lfloor n/2\rfloor}(K + \bar K)\xi_1\dots\xi_n\bar\xi_1\dots\bar\xi_n\right),\\
            \hat f_Q(\chi_{\bar K}) &= 2^{n-1}\left(\sum_{w\in W} (-1)^{\lfloor (|w|+1)/2\rfloor}(K - \bar K)w\bar w) - (-1)^{\lfloor n/2\rfloor}(K + \bar K)\xi_1\dots\xi_n\bar\xi_1\dots\bar\xi_n\right).
        \end{align*}

        Next, we compute $B_\tau$. Note that $e_1 = \frac{1}{4}(1 + K + \bar K + K\bar K)$ and $e_{K\bar K} = \frac{1}{4}(1 - K - \bar K + K\bar K)$ are idempotents so that $D\Kn e_1 = P_1$ and $D\Kn e_{K\bar K} = P_{K\bar K}$. Note that the comultiplication for the integral
        $$\Lambda = \frac{1}{2}(1+K+\bar K + K\bar K)\xi_1\dots\xi_n\bar\xi_1\dots\bar\xi_n,$$
        is given by
        \begin{align*}
            \Delta(\Lambda) &= \frac{1}{2}\sum_{i,j\in\{0,2\}}\sum_{w_1,w_2\in W}(-1)^{|w_1|\cdot |w_2^c|}\gamma_{w_1^c}\gamma_{w_2^c}K^{i+|w_1^c|}\bar K^{j+|w_2^c|}w_1\bar w_2\otimes K^{i}\bar K^{j}w_1^c\bar w_2^c.
            \intertext{Therefore, the Radford map is given by}
            \tau(x) &= \frac{1}{2}\sum_{i,j\in\{0,2\}}\sum_{w_1,w_2\in W}\gamma_{w_1^c}\gamma_{w_2^c}w_1\bar w_2 K^{i+|w_1^c|}\bar K^{j+|w_2^c|} x K^{i+|w_1^c|+1}\bar K^{j+|w_2^c|}\bar w_2^c w_1^c ).
            \intertext{By reindexing, we get}
            &= \frac{1}{2}\sum_{i,j\in\{0,1\}}\sum_{w_1,w_2\in W}\gamma_{w_1^c}\gamma_{w_2^c}w_1 \bar w_2 K^{i}\bar K^{j} x K^{i+1}\bar K^j\bar w_2^c w_1^c.
        \end{align*}
        Now, using this expression, we can compute the image of the idempotents. Note that $\gamma_{w}\gamma_{w^c} = (-1)^{|w|(n-|w|)} = (-1)^{|w|}$, so
        \begin{align*}
            \tau(e_1) 
            &=  \frac{1}{2}\sum_{w_1,w_2\in W}(-1)^{|w_1|+|w_2|}(1 + (-1)^{|w_1|+|w_2|}K + (-1)^{|w_1|+|w_2|}\bar K + K\bar K) \xi_1\xi_2\dots\xi_n \bar\xi_1\bar \xi_2\dots\bar\xi_n.
            \intertext{By summing over $|w_1|+|w_2|$, we get the formula}
            &=  \frac{1}{2}\sum_{\ell=0}^{2n}\binom{2n}{\ell}((-1)^{\ell} + K + \bar K + (-1)^{\ell}K\bar K) \xi_1\dots\xi_n \bar\xi_1\dots\bar\xi_n\\
            &= 2^{2n-1}(K+\bar K)\xi_1\dots\xi_n \bar\xi_1\dots\bar\xi_n.
        \end{align*}
        Next, we compute the image on $e_K$.
        \begin{align*}
            \tau(e_K) &= \frac{(-1)^{n/2}}{2^{n+1}}\sum_{w_1,w_2\in W}\gamma_{w_1^c}\gamma_{w_2^c}w_1 \bar w_2 (1 + K - \bar K - K\bar K) \xi_1\dots\xi_n\bar\xi_1\dots\bar\xi_n\bar w_2^c w_1^c\\
            \intertext{Note that this is only nonzero when $w_2 = \xi_1\dots\xi_n$. In this case, by careful manipulation, we find}
            &= \frac{1}{2}\sum_{w_1\in W}((-1)^{|w_1|} + K - \bar K - (-1)^{|w_1|}K\bar K)w_1 \bar w_1\bar w_1^c w_1^c\\
            &= \frac{1}{2}\sum_{w\in W}2^{n - |w|}\sum_{w_1w_2=\pm w}((-1)^{|w_1|} + K - \bar K - (-1)^{|w_1|}K\bar K)(-1)^{\lfloor |w_1^c|/2\rfloor}\times \\&\times  (-1)^{\lfloor (|w_2|+1)/2\rfloor+|w_1||w_2|}w \bar w.
        \end{align*}
        Note that, for fixed $w\in W$,
        \begin{align*}
            &\sum_{\ell=0}^{|w|}\binom{|w|}{\ell}(-1)^{\lfloor(n - \ell)/2\rfloor + \lfloor (|w| - \ell+1) / 2\rfloor + \ell(|w|-\ell)} = (-1)^{n/2 + \lfloor (|w| + 1)/2\rfloor}2^{|w|},\\ 
            &\sum_{\ell=0}^{|w|}\binom{|w|}{\ell}(-1)^{\lfloor(n - \ell)/2\rfloor + \lfloor (|w| - \ell+1) / 2\rfloor + \ell(|w|-\ell) + \ell}  = (-1)^{n/2}\delta_{|w|=0}.
        \end{align*}
        In particular, it follows that
        $$\tau(e_K) = 2^{n-1}(-1)^{n/2}\left(\sum_{w\in W} (-1)^{\lfloor(|w|+1)/2\rfloor} (K-\bar K)w\bar w + (1 - K\bar K)\right).$$
        Similarly,
        $$\tau(e_{\bar K}) = 2^{n-1}(-1)^{n/2}\left(\sum_{w\in W} (-1)^{\lfloor(|w|+1)/2\rfloor} (K-\bar K)w\bar w-(1 - K\bar K)\right).$$
    \end{proof}
    \begin{corollary}
        Suppose $n$ is even. The $S_{CW}$- and $T_{CW}$-matrices on the Higman ideal of $D\Kn$ are given by
        \begin{equation}\label{eq:DKnST}
            S_{CW} = (-1)^{n/2}\begin{pmatrix}
            0 & 2^{-n} & -2^{-n}\\
            2^{n-1} & 1/2 & 1/2\\
            -2^{n-1} & 1/2 & 1/2
        \end{pmatrix},\hspace{20mm}T_{CW} = \begin{pmatrix}
            1 & 0 & 0\\
            0 & 1 & 0\\
            0 & 0 & -1
        \end{pmatrix}.
        \end{equation}
    \end{corollary}
    \begin{corollary}    
        Suppose $n$ is even. The mixed fusion matrices and their associated diagonalizations for $D\Kn$ are given by
        \begin{align*}
        N^{K\bar K} &= \begin{pmatrix}
            1 & 0 & 0\\
            0 & 0 & 1\\
            0 & 1 & 0
        \end{pmatrix}, &S_{CW}N^{K\bar K}S_{CW}^{-1} &= \begin{pmatrix}
            -1 & 0 & 0\\
            0 & 1 & 0\\
            0 & 0 & 1
        \end{pmatrix},\\
        N^{K} = N^{\bar K} &= \begin{pmatrix}
            0 & 1 & 1\\
            2^{2n-1} & 0 & 0\\
            2^{2n-1} & 0 & 0
        \end{pmatrix}, & S_{CW}N^K S_{CW}^{-1} = S_{CW}N^{\bar K}S_{CW}^{-1} &= \begin{pmatrix}
            0 & 0 & 0\\
            0 & 2^n & 0\\
            0 & 0 & -2^n
        \end{pmatrix}.
        \end{align*}
    \end{corollary}

    \begin{theorem}\label{DKn-CW-decomp}
        Suppose $n$ is even. The $\mathrm{SL}(2,\bZ)$-action on $\Hig(D\Kn)$ decomposes as $\bC_{\mathrm{triv}}\oplus N_1$, where $N_1$ is the level 2 Weil representation.
        The $\mathrm{SL}(2,\bZ)$-action on $Z_\Lambda$ decomposes as
        $$\bigoplus_{\substack{0\leq k\leq n\\k\text{ even}}} \binom{n}{k}2^{n-k}(\bC^2_{\mathrm{std}})^{\otimes k}.$$
        where $(\bC^2_{\mathrm{std}})^{\otimes 0}:=\bC_{\mathrm{triv}}$. 
    \end{theorem}
    \begin{proof}
        Consider the $S_{CW}$ and $T_{CW}$ in Equation (\ref{eq:DKnST}) with a change-of-basis matrix 
        \[Q=\begin{pmatrix}1 & \frac{2^{1-n}}{\sqrt{3}} & 0 \\ 2^n & -\frac{1}{\sqrt{3}} & 0 \\ 0 & 0 & 1\end{pmatrix}.\]
        Then we have 
        \[Q^{-1} S_{CW} Q=(-1)^{n/2}\begin{pmatrix}1 & 0 & 0 \\ 0 & -\frac{1}{2} & -\frac{\sqrt{3}}{2} \\ 0 & -\frac{\sqrt{3}}{2} & \frac{1}{2}\end{pmatrix},\qquad Q^{-1} T_{CW} Q=\begin{pmatrix}1 & 0 & 0 \\ 0 & 1 & 0 \\ 0 & 0 & -1\end{pmatrix}.\]
        This implies the $\operatorname{SL}(2,\mathbb Z)$-representation on the Higman ideal of $D \mathcal{K}_n$ is a direct sum of a one-dimensional and a two-dimensional irreducible representation. In particular, the two-dimensional representation in the decomposition is equivalent to the level $2$ Weil representation $N_1$ \cite[Satz 4, p. 474]{Nobs1}.
        
        If $w_1 = \xi_1^{a_1}\dots\xi_n^{a_n}$ and $w_2 = \xi_1^{b_1}\dots\xi_n^{b_n}$, then let 
        $$\Gamma_{w_1, w_2} = (-1)^{\sum_{i=1}^{n-1}b_i\sum_{j=i+1}^n a_j}.$$
        Note that, for $w_1, w_2\in W$, we have $w_1\bar w_2 = \Gamma_{w_1, w_2}\tilde w$ for some $\tilde w\in \tilde W$. Moreover, $\gamma_{w} = \Gamma_{w, w^c}$. Finally,
        \begin{align*}
            \Gamma_{w_1, w_2}\Gamma_{w_2^c, w_1^c} &= (-1)^{(\sum_{i=1}^{n} a_j) (\sum_{j=1}^n b_i) + \sum_{i=1}^n a_i b_i + n/2 + \sum_{i=2}^{n}(i - 1)a_i + \sum_{i=2}^{n}(i-1)b_i}.
            \intertext{One can verify that $\sum_{i=2}^{n}(i - 1)a_i \equiv \lfloor \sum_{i=1}^n a_i/2\rfloor + \sum_{i=1}^{n-1}(1 - a_i)\sum_{j=i+1}^n a_j\pmod{2}$. Thus, we have }
            &= (-1)^{n/2 + \sum_{i=1}^n a_i b_i}[(-1)^{\lfloor|w_1^c|/2\rfloor + \lfloor |w_2^c|/2\rfloor + |w_1|}\gamma_{w_1^c}\gamma_{w_2}].
        \end{align*}
        For $\tilde w = \xi_1^{a_1}\bar\xi_1^{b_1}\dots\xi_n^{a_n}\bar\xi_n^{b_n}\in \tilde W$, define $\tilde w^s = \xi_1^{1 - b_1}\bar\xi_1^{1 - a_1}\dots\xi_n^{1 - b_n}\bar\xi_n^{1 - a_n}$. By the identities for $\Gamma_{w_1, w_2}$, we have
        $$S_{LM}((1+K\bar K)\tilde w) = (-1)^{n/2 + \sum_{i=1}^n a_i b_i}(1 + K\bar K)\tilde w^s.$$
        Moreover,
        $$T_{LM}((1+K\bar K)\tilde w) = (1 + K\bar K)\sum_{w'\in W}(-1)^{|w'|}\tilde w\hat w'.$$
        Note that $\hat w$ is a product of elements like $\xi_i\bar\xi_i$ and any $\xi_j, \bar\xi_k\in \mathbf{I}(\tilde w)$ commutes with $\xi_i\bar\xi_i$ (up to a factor which cancels with the $(1+K\bar K)$). Thus, $\tilde w\hat w' = \tilde w''\in \tilde W$ when nonzero.
        
        Consider the linear inclusion $p:Z_\Lambda\hookrightarrow (\bC^2\otimes \bC^2)^{\otimes n}$ generated by 
            $$p:(1+K\bar K)\xi_1^{a_1}\bar\xi_1^{b_1}\dots\xi_n^{a_n}\bar\xi_n^{b_n}\mapsto |a_1 b_1\rangle\otimes \dots\otimes |a_n b_n\rangle.$$ 
        for $a_i,b_i\in \{0,1\}$. The subspace $p(Z_\Lambda)\subset (\bC^2\otimes \bC^2)^{\otimes n}$ has a basis consisting of bit strings with an even number of 1s. Then, $p\circ S_{LM}|_{Z_\Lambda}\circ p^{-1} = (-1)^{n/2} \tilde s^{\otimes n}|_{p(Z_\Lambda)}$ and $p\circ T_{LM}|_{Z_\Lambda}\circ p^{-1} = \tilde t^{\otimes n}|_{p(Z_\Lambda)}$, where 
        $$\tilde s = \begin{pmatrix}
            0 & 0 & 0 & 1\\
            0 & 1 & 0 & 0\\
            0 & 0 & 1 & 0\\
            -1 & 0 & 0 & 0
        \end{pmatrix},\hspace{30mm}\tilde t = \begin{pmatrix}
            1 & 0 & 0 & -1\\
            0 & 1 & 0 & 0\\
            0 & 0 & 1 & 0\\
            0 & 0 & 0 & 1
        \end{pmatrix}.$$
        The $4\times 4$ matrices $s$ and $t$ have joint invariant subspaces of the form $(1+K\bar K)\spa\{|00\rangle, |11\rangle\}$ on which $\mathrm{SL}(2,\bZ)$ has a standard representation, as well as $(1+K\bar K)\spa\{|01\rangle\}$ and $(1+K\bar K)\spa\{|10\rangle\}$ where $\mathrm{SL}(2,\bZ)$ has a trivial representation. Thus, the representation of $\mathrm{SL}(2,\bZ)$ given by $s\mapsto \tilde s^{\otimes n}$ and $t\mapsto \tilde t^{\otimes n}$ decompose as
        $$(\bC^2_{\mathrm{std}} \oplus 2\bC_{\mathrm{triv}})^{\otimes n} \cong \bigoplus_{k=0}^n \binom{n}{k}2^k(\bC^2_{\mathrm{std}})^{\otimes n-k}\otimes \bC_{\mathrm{triv}}^{\otimes k} \cong\bigoplus_{k=0}^n \binom{n}{k}2^k(\bC^2_{\mathrm{std}})^{\otimes n-k},$$
        where we define $(\bC^2_{\mathrm{std}})^{\otimes 0}:=\bC_{\mathrm{triv}}$. However, only those for even $k$ correspond to representations acting on $Z_\Lambda$, as we must have an even number of $1$s. Thus, the $\mathrm{SL}(2,\bZ)$-representation on $Z_\Lambda$ decomposes as
        $$\bigoplus_{k\text{ even}} \binom{n}{k}2^k(\bC^2_{\mathrm{std}})^{\otimes n-k}.$$
    \end{proof}

\section*{Acknowledgments}
Z.W. thanks V. Ostrik and I. Runkel for communication and discussion on finite tensor categories and semi-simplification.  He also thanks C. Galindo for pointing out the references on Nichols Hopf algebras, and X. Cui, N. Geer, and E. Rowell for helpful comments.  Z.W. is partially supported by NSF grant CCF 2006463 and ARO MURI contract W911NF-20-1-0082. L.C.'s contribution is based on work done while he was a UCSB student.  L.C and Q.Z. thank Yilong Wang for enlightening discussions on Weil representations.

\bibliographystyle{abbrv}
\bibliography{zbib}
\end{document}